\documentclass[times,final,11pt]{elsarticle}

\usepackage{jcomp}
\usepackage{framed,multirow}
\usepackage{hyperref} 
\usepackage{algpseudocode}
\usepackage{amssymb,amsthm,amsmath,bm,amsfonts}

\DeclareMathAlphabet{\mathbbm}{U}{bbm}{m}{n}

\usepackage{dsfont}
\usepackage{latexsym}
\usepackage{subcaption}
\usepackage{url}
\usepackage{xcolor}
\usepackage{mathtools}
\usepackage{cleveref}
\usepackage{showkeys}
\usepackage{multirow}
\usepackage{booktabs}
\usepackage{enumitem}  
\usepackage{array}
\usepackage{cancel}
\usepackage{algorithm}
\usepackage{algpseudocode}
\usepackage{mathabx} 

\journal{Paper Submitted to J. Comput. Phys.}

\usepackage{tikz}
\usetikzlibrary{positioning}
\usetikzlibrary{shapes, arrows}
\usetikzlibrary{arrows.meta}
\tikzstyle{terminator} = [rectangle, draw, text centered, rounded corners]
\tikzstyle{process} = [rectangle, draw, text centered]
\tikzstyle{decision} = [diamond, draw, text centered]
\tikzstyle{data}=[trapezium, draw, text centered, trapezium left angle=60, trapezium right angle=120]
\tikzstyle{connector} = [draw, -{latex[length=2mm]}, thick, blue]
\tikzstyle{connector2} = [draw, -{latex[length=2mm]}, thick, green]
\pgfdeclarelayer{background}
\pgfsetlayers{background,main}

\newtheorem{theorem}{Theorem}[section]
\newtheorem{lemma}[theorem]{Lemma}

\newtheorem{remark}{Remark}[section]
\newtheorem{assumption}{Assumption}

\numberwithin{equation}{section}

\theoremstyle{definition}
\newtheorem{exa}[remark]{Example}

\theoremstyle{remark}

\newcommand{\BU}{\mathbf{U}}

\begin{document}

\verso{W. Chen, S. Cui, K. Wu, T. Xiong, B. Yu}

\begin{frontmatter}
	
	\title {{\bf Bound-Preserving WENO Schemes for Temple--class systems}\tnoteref{tnote1}
	}
	
	\tnotetext[tnote1]{
		The work of Wei Chen and Tao Xiong was partially supported by National Key R\&D Program of China No. 2022YFA1004500, NSFC grant No. 11971025, and NSF of Fujian Province grant No. 2023J02003.
		The work of Shumo Cui was partially supported by Shenzhen Science and Technology Program grant No. RCJC20221008092757098.
	}

	\author[XU]{Wei \snm{Chen}}
	\ead{weichenmath@stu.xmu.edu.cn}
	
	\author[ICM,SUSTech]{Shumo \snm{Cui}\corref{cor1}}
	\ead{cuism@sustech.edu.cn}
	
	\author[SUSTech,ICM,NCAMS]{Kailiang \snm{Wu}}
	\ead{wukl@sustech.edu.cn}
	
	\author[USTC]{Tao \snm{Xiong}}
	\ead{txiong@xmu.edu.cn}

	\author[SUSTech]{Baoyue \snm{Yu}}
	\ead{???}
	
	\address[XU]{School of Mathematical Sciences, Xiamen University, Xiamen, Fujian 361005, China.}
	\address[SUSTech]{Department of Mathematics, Southern University of Science and Technology, Shenzhen 518055, China.}
	\address[ICM]{Shenzhen International Center for Mathematics, Southern University of Science and Technology, Shenzhen 518055, China.}
	\address[NCAMS]{National Center for Applied Mathematics Shenzhen (NCAMS), Shenzhen 518055, China.}
	\address[USTC]{School of Mathematical Sciences, University of Science and Technology of China, Hefei, Anhui 230026, China.}
	
	\cortext[cor1]{Corresponding author.}

\begin{abstract}
	This paper explores numerical schemes for Temple-class systems, which are integral to various applications including one-dimensional two-phase flow, elasticity, traffic flow, and sedimentation. Temple-class systems are characterized by conservative equations, with different pressure function expressions leading to specific models such as the Aw-Rascle-Zhang (ARZ) traffic model and the sedimentation model. Our work extends existing studies by introducing a moving mesh approach to address the challenges of preserving non-convex invariant domains, a common issue in the numerical simulation of such systems.
	Our study outlines a novel bound-preserving (BP) and conservative numerical scheme, designed specifically for non-convex sets in Temple-class systems, which is critical for avoiding non-physical solutions and ensuring robustness in simulations.
	We develop both local and global BP methods based on finite difference schemes, with numerical experiments demonstrating the effectiveness and reliability of our methods. 
	Furthermore, a parameterized flux limiter is introduced to restrict high-order fluxes and maintain bound preservation. This innovation marks the first time such a parameterized approach has been applied to non-convex sets, offering significant improvements over traditional methods.
	The findings presented extend beyond theoretical implications, as they are applicable to general Temple-class systems and can be tailored to ARZ traffic flow networks, highlighting the versatility and broad applicability of our approach. The paper contributes significantly to the field by providing a comprehensive method that maintains the physical and mathematical constrains of Temple-class systems.
\end{abstract}
	
	\begin{keyword}
	\MSC 65M60 \sep 65M12 \sep 76A30\sep 35L65 \sep 35R02
	
	\\
	 \KWD  Temple-class systems, Aw-Rascle-Zhang traffic flow networks, moving mesh, bound-preserving methods, non-convex sets, parameterized flux limiter.
	 \end{keyword}
	
\end{frontmatter}

\section{Introduction}

\subsection{The Temple-class system}\label{sec2}

In this paper, we are interested in the numerical approximaion of Temple-class system, which appears in various applications such as one-dimensional two-phase flow, elasticity, traffic flow, and sedimentation \cite{gu2012existence, keyfitz1980system,temple1982global,klingenberg2001stability,betancourt2014modeling,zhang2002non,betancourt2018random,aw2000resurrection}. 
The Temple-class system is described by the equation
\begin{equation}\label{TC}
    \partial_t \mathbf{U} + \partial_x \mathbf{F}(\mathbf{U}) = 0,
\end{equation}
where the conservative variable $\mathbf{U}$ and the flux function $\mathbf{F}$ are given by $\mathbf{U} = (\phi, y)^\top$, $\mathbf{F}(\mathbf{U}) = v \mathbf{U}$ with the variables $\phi(x, t)$ and $y(x, t)$ being the unknowns. 
Based on the following Assumption \ref{assum1}, we will yield the Aw-Rascle-Zhang (ARZ) traffic model \cite{aw2000resurrection,zhang2002non} and the sedimentation model \cite{betancourt2018random,betancourt2014modeling} with two particular function expressions for $v$.
\begin{assumption}
\label{assum1}
	The variable $y$ can be written as a function of $v$ and $\phi$. For each constant value of $v$, $y$ is strictly convex with respect to $\phi > 0$. Additionally, the partial derivative $\frac{\partial y}{\partial v}$ is positive for every constant $\phi$.
\end{assumption}

For ARZ model, the $\phi$ stands for the density of cars, and $y$ is the generalized momentum. The ``driver behavior'' $k$ and positivity (non-negativity) velocity $v$ are defined by
\begin{equation*} 
	\label{kv}
	k = k(\mathbf{U}) := y/\phi, \quad v = v(\mathbf{U}) := k - p(\phi).
\end{equation*}
Here, the pressure $p(\phi)$ describes the response of drivers to changes in the density of the vehicle ahead. In this paper, we will consider the following pressure function:
\begin{equation*}\label{P}
    p(\phi) = \left\{
    \begin{aligned}
        & \frac{v_{\rm ref}}{\gamma} \phi^\gamma, \; && {\rm if} \; \gamma > 0, \\
        & v_{\rm ref} \log{\phi}, \; && {\rm if} \; \gamma = 0
    \end{aligned}
    \right.
\end{equation*}
with a given reference velocity $v_{\rm ref}$.

We now briefly describe the sedimentation model. The variable $\phi$ represents the solids volume fraction, and $y$ denotes a conserved quantity that describes the local concentration of particles, incorporating additional information on the property $k$ that the particles possess. In this model, the expression of property $k$ and the settling velocity of the solids $v$ are given as
\begin{equation*} 
	\label{kv2}
	k = k(\mathbf{U}) := y/\phi, \quad v = v(\mathbf{U}) := kp(\phi).
\end{equation*}
characterized by the dimensionless constitutive function $p(\phi) = (1 - \phi)^2$. 

From a physical standpoint, both the ARZ and sedimentation models should ensure that $\phi$, representing density or solid volume fraction, is non-negative to avoid unphysical negative values. And $\phi$ should also be less than or equal to $1$ after normalization. Furthermore, since both models simulate unidirectional flow, the velocity should also be non-negative. Mathematically, the Temple-class systems possess two Riemann invariants, $k$ and $v$, both of which adhere to the minimum and maximum principles. Specifically, the relationship is given by:
\begin{equation}\label{eq:328}
	\BU(x,t) \in \mathcal{I}^{\rm G} := 
	\left\{ 
	\BU
	\in \mathbb{R}^2 
	: 
	v \in [v^{\rm G}_{\min}, v^{\rm G}_{\max}], 
	k \in [k^{\rm G}_{\min}, k^{\rm G}_{\max}]
	\right\},
\end{equation}
where $\mathcal{I}^{\rm G}$ is the global invariant domain. The global upper and lower bounds of Riemann invariants defined as
\begin{equation*}\label{eq:387}
\begin{gathered}
	v^{\rm G}_{\min} := \min_x v(x,0), \quad v^{\rm G}_{\max} := \max_x v(x,0),\\
	k^{\rm G}_{\min} := \min_x k(x,0), \quad k^{\rm G}_{\max} := \max_x k(x,0). 
\end{gathered}
\end{equation*}
The hyperbolic system described by \eqref{TC} exhibits finite propagation speed; therefore, one can extend the global bound in \eqref{eq:328} to a local version:
$\BU(x,t) \in \mathcal{I}^{\rm L}(x,t,t_0)$, $t > t_0 \geq 0$,
where the local invariant domain $\mathcal{I}^{\rm L}(x,t,t_0)$ is defined as follows:
\begin{equation}\label{eq:154}
	\mathcal{I}^{\rm L}(x,t,t_0)
	:=
	\left\{ 
	\BU= \left( \phi, y \right)^\top \in \mathbb{R}^2 
	: 
	\renewcommand\arraystretch{1.2}
	\begin{array}{r}
		v \in [v^{\rm L}_{\min}(x,t,t_0), v^{\rm L}_{\max}(x,t,t_0)] \\  
		k \in [k^{\rm L}_{\min}(x,t,t_0), k^{\rm L}_{\max}(x,t,t_0)] 
	\end{array}
	\right\},
\end{equation}
where the local bounds of Riemann invariants are defined as
\begin{equation*}\label{eq:412}
\begin{aligned}
	v^{\rm L}_{\min}(x,t,t_0) := \min_{\xi \in  \mathcal{X}} v(\xi,t_0), \quad v^{\rm L}_{\max}(x,t,t_0) := \max_{\xi \in  \mathcal{X}} v(\xi,t_0), \\
	k^{\rm L}_{\min}(x,t,t_0) := \min_{\xi \in  \mathcal{X}} k(\xi,t_0), \quad k^{\rm L}_{\max}(x,t,t_0) := \max_{\xi \in  \mathcal{X}} k(\xi,t_0).
\end{aligned}
\end{equation*}
Here, the local domain of determination is defined as $\mathcal{X}:=[x-\alpha(t-t_0),x+\alpha(t-t_0)]$, and $\alpha$ denotes the maximum propagation speed. Please note that $\mathcal{I}^{\rm L} \subseteq \mathcal{I}^{\rm G}$.

\subsection{Related works}
In the numerical simulations of Temple-class systems, it is crucial to preserve both the physical properties of the system (such as $\phi > 0$ and $v \geq 0$) and the characteristics of the equations (with Riemann invariants constrained by Riemann invariant domains). Failure to maintain these properties can result in non-physical solutions or even program crashes. Therefore, constructing a bound-preserving (BP) scheme for Temple-class systems is essential.

In recent years, the design of high-order BP schemes for hyperbolic systems has attracted considerable attention. These include preserve minimum/maximum principles \cite{zhang2010maximum,xu2014parametrized,Lv2014EntropyboundedDG,AndersoDobrevKolevKuzminQuezadaRiebenTomov2017}, positivity \cite{zhang2011maximum,wu2018positivity,WuShu2019}, and other types of bounds \cite{wu2021minimum,WuTang2015,wu2017design,mabuza2018local,Dmitri2021}. Zhang and Shu proposed a general framework for BP schemes based on arbitrary high-order finite volume Weighted Essentially Non-Oscillatory (WENO) and discontinuous Galerkin methods \cite{zhang2011maximum,zhang2010maximum}. Wu and Shu further introduced the geometric quasilinearization approach, providing an effective new method for BP analysis and design in systems with nonlinear and complex constraints \cite{wu2023geometric}. Xu \cite{xu2014parametrized} proposed a parametrized flux limiter method based on finite difference WENO. This method modifies the flux of high-order numerical schemes to ensure the enforcement of the discrete maximum principle, while using first-order BP numerical schemes as its foundation. To our knowledge, these BP schemes have primarily been applied to convex sets. However, the Riemann invariant domain in the Temple-class system discussed in this paper is inherently a non-convex set, making the design of BP schemes for such a system highly challenging.

The non-convex region within the invariant domain of the Temple-class system is a consequence of the constraint $v \leq v_{\max}$. In \cite{CHEN20251135007}, it was observed that for ARZ model, without an imposed upper bound on velocity, simulations involving low-density numerical examples produced velocities exceeding acceptable limits by two orders of magnitude. This led to extremely small time steps and, in some cases, caused the simulations to terminate. In an attempt to address this issue, an upper limit of $k \leq k_{\max}$ was introduced. However, this approach did not effectively resolve the fundamental problem of ensuring $v \leq v_{\max}$. 
For ARZ model, Chalons and Goatin \cite{chalons2007transport} addressed the issue of velocity overshoot by employing a random-sampling strategy in the numerical treatment of contact discontinuities, while using Godunov’s method elsewhere. 
Additionally, they extended the scheme to account for cases where the phase space is non-convex \cite{chalons2008godunov}. 
The random-sampling strategy has also been applied to the Temple-class system \cite{betancourt2018random}. 
However, the numerical schemes proposed in \cite{betancourt2018random,chalons2007transport,chalons2008godunov} are only first-order and lack conservation properties. Proposition 4.1 in \cite{betancourt2018random} states that no strictly conservative, first-order scheme with a consistent two-point numerical flux can satisfy the velocity maximum principle on a fixed mesh. This raises the question: can a strictly conservative, high-order scheme, particularly one that maintains the bound $v \leq v_{\max}$, be constructed using a moving mesh?

To date, moving mesh methods have been successfully applied to a wide range of problems in science and engineering \cite{zhang2018adaptive,huang2001variational,NMTMA-16-111,lei2021high,duan2021entropy,barlow2016arbitrary,boscheri2017high}. 
The readers are also referred to the review papers \cite{tang2005moving,budd2009adaptivity,donea2004arbitrary} and references therein.
These approachs are characterized by their flexibility in selecting mesh velocities. 
Specifically, the grid velocity can be chosen independently of the local fluid velocity, allowing for the recovery of fully Eulerian algorithms on fixed grids when the mesh velocity is set to zero. 
However, the design of these schemes is mainly to improve the robustness, accuracy, and ease of use of simulations by introducing mesh motion. 
This paper will achieve the purpose of non-convex bound preservation by using a moving mesh.

\subsection{Contributions}

The main contributions of this paper are as follows:
\begin{enumerate}[leftmargin=8mm]
	\item[1)] A numerical scheme that is both BP and conservative is constructed based on moving mesh, overcoming the challenge of maintaining the non-convex invariant domain in Temple-class systems. This provides a promising direction for solving the difficult problem of designing BP schemes for non-convex sets.

	\item[2)] Under the assumption that the grid movement speed is known, the numerical scheme is analyzed in detail. The analysis shows that for a first-order scheme to maintain bound preservation and conservation, the grid movement speed should strictly match the fluid velocity.

	\item[3)] A parameterized flux limiter is introduced to restrict high-order fluxes and maintain bound preservation. It is the first time that a parameterized approach for non-convex sets is provided, offering a solution to the challenge of maintaining velocity upper bounds.
	
	\item[4)] Both local and global BP methods based on finite difference schemes are developed. Numerical experiments show that the local method is more effective and robust compared to the global approach, offering a reliable solution for Temple-class systems.

	\item[5)] The proposed scheme is not only applicable to general Temple-class systems but is also extended to ARZ traffic flow networks, showing its versatility and potential for broader applications.
\end{enumerate}

The remainder of this paper is organized as follows. Section \ref{0sec2} briefly explains why it is impossible to design a BP scheme with a consistent two-point numerical flux for Temple-class systems on a fixed mesh, and outlines the challenges encountered when constructing a BP and conservative scheme on a moving mesh. In the Section \ref{0sec3}, we design a first-order BP scheme based on a moving mesh and provide the corresponding time-step constraints. We extend this first-order scheme to high-order scheme using a parameterized flux limiter, along with a method for selecting the appropriate parameters in Section \ref{0sec4}. In Section \ref{0sec5}, we present numerical examples for Temple-class systems and extend the approach to the ARZ traffic flow network. Finally, Section \ref{0sec6} offers a summary of the findings.

\section{The feasibility of bound-preserving schemes}\label{0sec2}

Based on Proposition 4.1 of \cite{betancourt2018random} and Section 3.2 of \cite{chalons2007transport}, we state the difficulty for disregarding $v \leq v_{\max}$ in the following theorem that is more appropriate for this paper.
\begin{theorem}
\label{contradiction}
	Consider a Temple-class system satisfying Assumption \ref{assum1} with the Riemann initial conditions:
	\begin{equation*}
		\mathbf{U}^0 (x) = \left\{
		\begin{aligned}
			& \mathbf{U}_L \quad {\rm for} \; x \leq x_{S + \frac{1}{2}}, \\
			& \mathbf{U}_R \quad {\rm otherwise},
		\end{aligned}
		\right.
	\end{equation*}
	where $v_L = v_R = v^{\dag}$ and $\phi > 0$. 
	For fixed computational mesh, no conservative and consistent scheme under the CFL condition $0 < \Delta t_n v^{\dag} / \Delta x \leq 1$
	will produce an approximate solution contained in $\mathcal{I}^{\rm G} \cap \{\phi > 0\}$, and the upper bound $v_{\rm max}^{\rm G}$ must be destroyed.
\end{theorem}

\begin{figure}[hbtp]
	\begin{center}
		\mbox{{\includegraphics[width = 0.5\textwidth,trim=32 1 50 40,clip]{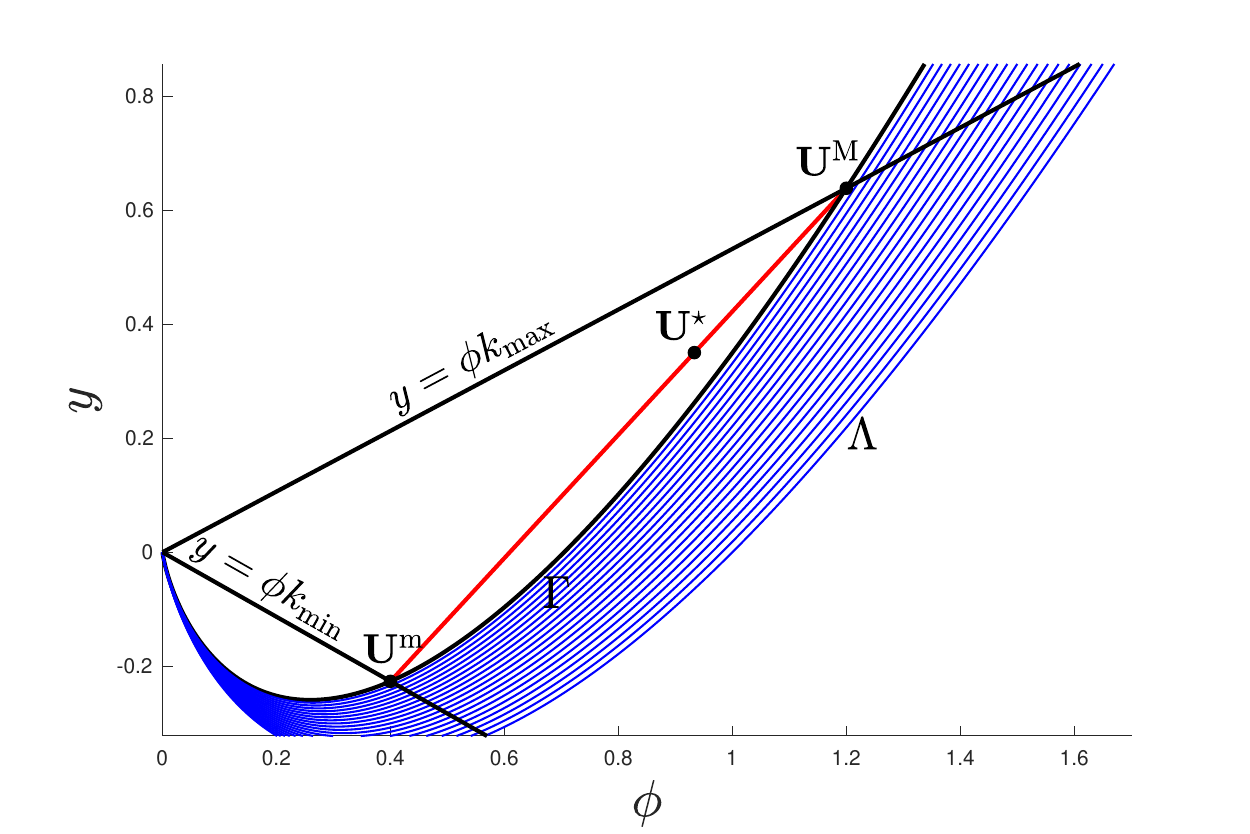}}}     
		\caption{\sf The $y (v, \phi)$ curves in Theorem \ref{contradiction}.}
		\label{z_curves}
	\end{center}
\end{figure}
\begin{proof}
After the first updates, a conservative and consistent scheme will yield
\begin{equation*}
    \mathbf{U}_S^1 = \mathbf{U}_L - \frac{\Delta t_0}{\Delta x} \left( \widehat{\mathbf{F}}^0_{S + \frac{1}{2}} - v^{\dag} \mathbf{U}_L \right), \quad \mathbf{U}_{S + 1}^1 = \mathbf{U}_R - \frac{\Delta t_0}{\Delta x} \left( v^{\dag} \mathbf{U}_R - \widehat{\mathbf{F}}^0_{S + \frac{1}{2}} \right).
\end{equation*}
With the definition $\lambda := \left( 1 + \Delta t_0 v^\dag / \Delta x \right) / 2$ , we note that
\begin{equation}
	\label{equal}
	\frac{1}{2} \left( \mathbf{U}_{S}^1 + \mathbf{U}_{S + 1}^1 \right) = \lambda \mathbf{U}_L + (1 - \lambda) \mathbf{U}_R =: \mathbf{U}^\star.
\end{equation}
In Figure \ref{z_curves}, we plot the curve $\Gamma: y \left( v^\dag, \phi\right)$, where $\mathbf{U}^{\rm M}$ indicates $\mathbf{U}_L$ or $\mathbf{U}_R$ which are more dense (larger $\phi$) and $\mathbf{U}^{\rm m}$ indicates which are less. For such Riemann problem, we note that the line passing through the origin and $\mathbf{U}^{\rm M}$ ($\mathbf{U}^{\rm m}$) stands for the bound $k_{\rm max}^{\rm G}$ ($k_{\rm min}^{\rm G}$). To satisfy both conditions $\mathbf{U}_{S}^1, \mathbf{U}_{S + 1}^1 \in \mathcal{I}^{\rm G}$ and $v(\mathbf{U}_{S}^1), v (\mathbf{U}_{S + 1}^1) \in [v_{\rm min}^{\rm G}, v_{\rm max}^{\rm G}] = \{ v^\dag \}$, points $\mathbf{U}_{S}^1$ and $\mathbf{U}_{S + 1}^1$ can only be in the $[\phi^{\rm m}, \phi^{\rm M}]$ section of curve $\Gamma$. Since the curve is convex, the only possibilities for the updates are $(\mathbf{U}_{S}^1, \mathbf{U}_{S + 1}^1) = (\mathbf{U}^{\rm m}, \mathbf{U}^{\rm M})$ or $(\mathbf{U}_{S}^1, \mathbf{U}_{S + 1}^1) = (\mathbf{U}^{\rm M}, \mathbf{U}^{\rm m})$, which contradicts \eqref{equal} since $\lambda > 1/2$.

Now, we are going to state the upper bound $v_{\rm max}^{\rm G}$ must be destroyed. As shown in Figure \ref{z_curves}, we plot the family of curves $y(v, \phi)$ where $v < v^\dag = v_{\rm max}^{\rm G}$. Suppose that both $v(\mathbf{U}_{S}^1)$ and $ v (\mathbf{U}_{S + 1}^1)$ are less than or equal to $v_{\rm max}^{\rm G}$ and define the right region bounded by the lines $k_{\rm min}^{\rm G}$, $k_{\rm max}^{\rm G}$ and the $[\phi^{\rm m}, \phi^{\rm M}]$ section of curve $\Gamma$ as $\Lambda$. Denote by $\Lambda^o$ the region $\Lambda$ with points $\mathbf{U}^{\rm m}$ and $\mathbf{U}^{\rm M}$ removed. Clearly, in region $\Lambda^o$, we cannot find points $\mathbf{U}_{S}^1$ and $\mathbf{U}_{S + 1}^1$ such that the line segment passing through these two points intersects the $\mathbf{U}^{\rm M} \mathbf{U}^{\rm m}$ segment. Then, the upper bound $v_{\rm max}^{\rm G}$ must be destroyed.
\end{proof}
Theorem \ref{contradiction} highlights the difficulty of designing a BP, conservative and consistent scheme for the Temple-class system on a fixed mesh.  Based on a fixed mesh, \cite{chalons2007transport,chalons2008godunov,betancourt2018random} utilize an averaging step to subtly modify Godunov’s approach and adopt remapping strategies to achieve a statistically conservative property. This scheme is not strictly conservative and maintains only first-order accuracy. To achieve these properties, we relax the ``fixed mesh'' condition and consider the more general case of a ``moving mesh''. Following \cite{duan2021entropy}, we derive the Temple-class system in curvilinear coordinates. Let $\mathcal{D}_{p}$ denote the physical domain of system \eqref{TC}, and $\mathcal{D}_c$ represent the computational domain with the artificially chosen coordinate $\xi$ , used to facilitate movement. The adaptive moving meshes for $\mathcal{D}_p$ are generated as the images of a reference mesh in $\mathcal{D}_c$ through a time-dependent, differentiable, one-to-one coordinate mapping $x = x (\xi, t)$, which can be expanded as
\begin{equation*}
	t = \tau, \quad x = x(\xi, \tau), \quad \xi \in \mathcal{D}_c,
\end{equation*}
with the determinant of the Jacobian matrix
\begin{equation*}
	J := \left\vert \frac{\partial (t, x)}{\partial (\tau, \xi)} \right\vert = \frac{\partial x}{\partial \xi}.
\end{equation*}
Then, under such a  transformation, the system \eqref{TC} can be written as
\begin{equation}\label{newsystem}
    \partial_\tau \boldsymbol{\mathcal{U}} + \partial_\xi \mathbf{G}(\boldsymbol{\mathcal{U}}, c) = 0,
\end{equation}
with
\begin{equation*}
    \begin{aligned}
        \boldsymbol{\mathcal{U}} = (J\phi, Jy, J)^\top, \quad \mathbf{G}(\boldsymbol{\mathcal{U}},c) = \left((v - c)\phi, (v - c)y, - c\right)^\top
    \end{aligned}
\end{equation*}
where $c := \partial x / \partial t$ represents the rate of change of the physical domain with respect to time. It should be noted that the last equation in \eqref{newsystem} represents the volume conservation law from the geometric conservation laws, and the surface conservation law can be omitted in this one-dimensional system.

\begin{remark}
	If $c \equiv 1$ and $J \equiv 1$ are assumed, then the system described by \eqref{newsystem} simplifies to the original system \eqref{TC}.
\end{remark}
In this paper, we consider the following invariant domains of the system \eqref{newsystem}:
\begin{align}
	\Omega &= \left\{\boldsymbol{\mathcal{U}} \in \mathbb{R}^3: J > \epsilon, \phi \in (0, 1), k \in [k_{\min}, k_{\max}], v\in [v_{\min}, v_{\max}] \right\}, \label{eq:270} \\
	\Omega_1 &= \left\{\boldsymbol{\mathcal{U}} \in \mathbb{R}^3: J > \epsilon, \phi \in (0, 1), k \in [k_{\min}, k_{\max}] \right\}, \nonumber \\
	\Omega_2 &= \left\{\boldsymbol{\mathcal{U}} \in \mathbb{R}^3: J > \epsilon, \phi > 0, v \geq v_{\min} \right\}. \nonumber
\end{align}
Here, we ensure that \(J\) remains larger than a small positive number \(\epsilon\), to prevent distortion of the physical domain \(\mathcal{D}_p\). It is important to note that the primary difference between systems \eqref{TC} and \eqref{newsystem} is the additional variable $c$ in the flux. Therefore, finding a suitable $c$ is crucial for obtaining a BP scheme. In the following theorem, we show the preperties of above invariant domains and identify the infeasible region for $c$.
\begin{lemma}
\label{778}
	The invariant domains $\Omega_1$ and $\Omega_2$ are convex sets, while $\Omega$ is nonconvex.
\end{lemma}
\begin{proof}
	For any two states $\boldsymbol{\mathcal{U}}_0 = (J_0\phi_0, J_0y_0, J_0)^\top$ and $\boldsymbol{\mathcal{U}}_1 = (J_1\phi_1, J_1y_1, J_1)^\top \in \Omega_1$, we denote 
	\begin{equation*}
		\boldsymbol{\mathcal{U}}_\theta := (J_\theta\phi_\theta, J_\theta y_\theta, J_\theta)^\top = (1 - \theta)\boldsymbol{\mathcal{U}}_0 + \theta\boldsymbol{\mathcal{U}}_1
	\end{equation*}
	with any $\theta \in [0, 1]$. We will prove that $\boldsymbol{\mathcal{U}}_\theta \in \Omega_1$ in the following. Obviously, the Jacobian matrix satisfy $J_\theta = (1 - \theta)J_0 + \theta J_1 \geq \min \{J_0, J_1\} > \varepsilon$. In the given definitions 
	\begin{equation*}
		\theta_\phi := \frac{\theta J_1}{(1 - \theta)J_0+ \theta J_1}, \quad \theta_k := \frac{\theta J_1 \phi_1}{(1 - \theta)J_0\phi_0 + \theta J_1\phi_1},
	\end{equation*}
	where both variables lie within the interval $[0, 1]$, we derive the folloing equations:
	\begin{equation*}
		\begin{aligned}
			\phi_\theta &= \frac{J_\theta \phi_\theta}{J_\theta} = \frac{(1 - \theta)J_0\phi_0 + \theta J_1\phi_1}{(1 - \theta)J_0+ \theta J_1} \\
			& = (1 - \theta_\phi) \phi_0 + \theta_\phi \phi_1 \in \left[\min \{\phi_0, \phi_1\}, \max \{\phi_0, \phi_1\} \right] \subseteq (0, 1),\\
			k_\theta &= \frac{J_\theta y_\theta}{J_\theta \phi_\theta} = \frac{(1 - \theta)J_0 y_0 + \theta J_1 y_1}{(1 - \theta)J_0\phi_0 + \theta J_1\phi_1} \\
			&= (1 - \theta_k) k_0 + \theta_k k_1 \in \left[\min \{k_0, k_1\}, \max \{k_0, k_1\} \right] \subseteq [k_{\min}, k_{\max}].
		\end{aligned}
	\end{equation*}
	Thus, $\boldsymbol{\mathcal{U}}_\theta \in \Omega_1$ and $\Omega_1$ is convex.

	Now, let us assume $\boldsymbol{\mathcal{U}}_0$ and $\boldsymbol{\mathcal{U}}_1$ are elements of $\Omega_2$. Our objective is to demonstrate $v_\theta \geq v_{\min}$. Consider a function of $\phi$ defined as $\eta(\phi) := y(v_{\min}, \phi)$. Given that $\partial y / \partial v \geq 0$ as stated in Assumption \ref{assum1}, the condition $v(\boldsymbol{\mathcal{U}}) \geq v_{\min}$ is equivalent to $y \geq \eta(\phi)$. Therefore,
	\begin{equation*}
		\begin{aligned}
			y_\theta &= \frac{J_\theta y_\theta}{J_\theta} = \frac{(1 - \theta)J_0y_0 + \theta J_1y_1}{(1 - \theta)J_0+ \theta J_1}\\ 
			&= (1 - \theta_\phi) y_0 + \theta_\phi y_1 \\
			&\geq (1 - \theta_\phi) \eta(\phi_0) + \theta_\phi \eta(\phi_1) \\
			&\geq \eta((1 - \theta_\phi) \phi_0 + \theta_\phi \phi_1)\\
			& = \eta(\phi_\theta),
		\end{aligned}
	\end{equation*}
	where the convexity of $\eta$ as specified in Assumption \ref{assum1} has been utilized. Thus, we conclude that $\boldsymbol{\mathcal{U}}_\theta \in \Omega_2$ and $\Omega_2$ is convex.

	To show that the set $\Omega$ is nonconvex, we assume $\boldsymbol{\mathcal{U}}_0$ and $\boldsymbol{\mathcal{U}}_1 \in \Omega$, and we consider a special case where $v_{\min} \equiv v_{\max}$ and $\phi_0 \neq \phi_1$. Given the strict convexity of $\eta(\phi)$, for all $\theta \in (0, 1)$, results $y_\theta > \eta(\phi_\theta) = y(v_{\max}, \phi_\theta)$, implying $v_{\theta} > v_{\max}$. Consequently, $\boldsymbol{\mathcal{U}}_\theta \notin \Omega$ and the set $\Omega$ is nonconvex.

	The proof is completed.
\end{proof}

\begin{theorem}
\label{infeasible}
	For the system \eqref{newsystem}, it is impossible to construct a stable BP numerical scheme when $c \neq v$.
\end{theorem}
\begin{proof}
	The numerical scheme for system \eqref{newsystem} can be described by:
	\begin{equation*}
		\boldsymbol{\mathcal{U}}_j^{n + 1} = \boldsymbol{\mathcal{U}}_j^n - \frac{\Delta \tau_n}{\Delta \xi} \left( \widehat{\mathbf{G}}_{j + \frac{1}{2}}^n - \widehat{\mathbf{G}}_{j - \frac{1}{2}}^n \right),
	\end{equation*}
	where $\widehat{\mathbf{G}}_{j \pm \frac{1}{2}}^n$ represents the numerical flux. Note that $\boldsymbol{\mathcal{U}}_j^{n + 1}$ can be regarded as a function of $\Delta \tau_n$, namely, $\boldsymbol{\mathcal{U}}_j^{n + 1} (\Delta \tau_n)$, if the computation of numerical flux has been fixed. Given that $c \neq v$, we can construct a Riemann initial data
	\begin{equation*}
		\boldsymbol{\mathcal{U}}^0 (\xi) = \left\{
		\begin{aligned}
			& \boldsymbol{\mathcal{U}}_L \quad {\rm for} \; \xi \leq \xi_{S + \frac{1}{2}}, \\
			& \boldsymbol{\mathcal{U}}_R \quad {\rm otherwise},
		\end{aligned}
		\right.
	\end{equation*}
	satisfies
	\begin{equation}
	\label{90}
		v \left(\boldsymbol{\mathcal{U}}^0 \right) \equiv v_{\min} \equiv v_{\max}, \quad \frac{\widehat{\mathbf{G}}_{S - \frac{1}{2}}^{0;(1)} - \widehat{\mathbf{G}}_{S + \frac{1}{2}}^{0;(1)}}{\widehat{\mathbf{G}}_{S - \frac{1}{2}}^{0;(3)} - \widehat{\mathbf{G}}_{S + \frac{1}{2}}^{0;(3)}} \neq \frac{0}{0} \; {\rm and} \; \neq \phi \left( \boldsymbol{\mathcal{U}}_S^{0} \right). 
	\end{equation}
	Assuming both $\boldsymbol{\mathcal{U}}_S^{1}(\Delta \tau_0)$ and $\boldsymbol{\mathcal{U}}_S^{0}$ are contained in $\Omega$, then $v \left(\boldsymbol{\mathcal{U}}_S^{1}(\Delta \tau_0) \right) \equiv v_{\min} \equiv v_{\max}$ . Noting that a stable BP numerical scheme should ensure $\boldsymbol{\mathcal{U}}_S^{1}(\Delta \tau) \in \Omega$ for any $\Delta \tau \in (0, \Delta \tau_0)$. Since
	\begin{equation*}
		\phi \left( \boldsymbol{\mathcal{U}}_S^{1}(\Delta \tau) \right) = \frac{\boldsymbol{\mathcal{U}}_S^{1;(1)}(\Delta \tau)}{\boldsymbol{\mathcal{U}}_S^{1;(3)}(\Delta \tau)} = \frac{\boldsymbol{\mathcal{U}}_S^{0;(1)} - \frac{\Delta \tau}{\Delta \xi} \left( \widehat{\mathbf{G}}_{S + \frac{1}{2}}^{0;(1)} - \widehat{\mathbf{G}}_{S - \frac{1}{2}}^{0;(1)} \right)}{\boldsymbol{\mathcal{U}}_S^{0;(3)} - \frac{\Delta \tau}{\Delta \xi} \left( \widehat{\mathbf{G}}_{S + \frac{1}{2}}^{0;(3)} - \widehat{\mathbf{G}}_{S - \frac{1}{2}}^{0;(3)} \right)},
	\end{equation*}
	we emphasize that within the constraint $\Delta \tau > 0$, the sole condition under which 
	\begin{equation*}
		\phi \left( \boldsymbol{\mathcal{U}}_S^{1}(\Delta \tau) \right) = \boldsymbol{\mathcal{U}}_S^{0;(1)} / \boldsymbol{\mathcal{U}}_S^{0;(3)} = \phi \left( \boldsymbol{\mathcal{U}}_S^{0} \right),
	\end{equation*}
	presents a contradiction to \eqref{90}. Therefore, $\phi \left( \boldsymbol{\mathcal{U}}_S^{1}(\Delta \tau_0) \right) \neq \phi \left( \boldsymbol{\mathcal{U}}_S^{0} \right)$. Clearly, $\boldsymbol{\mathcal{U}}_S^{1}(\Delta \tau) $can be expressed as 
	\begin{equation}
	\label{76}
		\begin{aligned}
			\boldsymbol{\mathcal{U}}_S^{1}(\Delta \tau) 
			&= \boldsymbol{\mathcal{U}}_S^0 - \frac{\Delta \tau}{\Delta \xi} \left( \widehat{\mathbf{G}}_{j + \frac{1}{2}}^0 - \widehat{\mathbf{G}}_{j - \frac{1}{2}}^0 \right) \\
			&= \frac{\Delta \tau}{\Delta \tau_0} \left( \boldsymbol{\mathcal{U}}_S^0 - \frac{\Delta \tau_0}{\Delta \xi} \left( \widehat{\mathbf{G}}_{j + \frac{1}{2}}^0 - \widehat{\mathbf{G}}_{j - \frac{1}{2}}^0 \right)\right) 
			+ \left( 1 - \frac{\Delta \tau}{\Delta \tau_0}\right)\boldsymbol{\mathcal{U}}_S^0\\
			&=\frac{\Delta \tau}{\Delta \tau_0}\boldsymbol{\mathcal{U}}_S^{1}(\Delta \tau_0) + \left( 1 - \frac{\Delta \tau}{\Delta \tau_0}\right) \boldsymbol{\mathcal{U}}_S^{0}
		\end{aligned} 
	\end{equation}
	From $\phi \left( \boldsymbol{\mathcal{U}}_S^{1}(\Delta \tau_0) \right) \neq \phi \left( \boldsymbol{\mathcal{U}}_S^{0} \right)$, $v \left( \boldsymbol{\mathcal{U}}_S^{1}(\Delta \tau_0) \right) = v \left( \boldsymbol{\mathcal{U}}_S^{0} \right) = v_{\min} = v_{\max}$, and \eqref{76}, we deduce $v \left(\boldsymbol{\mathcal{U}}^1_S (\Delta \tau) \right) > v_{\max}$ as the last part of the proof of Lemma \ref{778}. Therefore, it is impossible to construct a stable BP numerical scheme when $c \neq v$ for the system \eqref{newsystem}.
\end{proof}

\section{First-order BP scheme}\label{0sec3}

We implement a spatial discretization for the domain $[\xi_L, \xi_R]$ as follows:
\begin{equation*}
	\xi_L = \xi_{\frac{1}{2}} < \xi_{\frac{3}{2}} \cdots < \xi_{N_\xi + \frac{1}{2}} = \xi_R,
\end{equation*}
where cell $I_j = [\xi_{j - \frac{1}{2}}, \xi_{j + \frac{1}{2}}]$ has mesh size $\Delta \xi = \frac{\xi_R - \xi_L}{N_\xi}$ and center $\xi_{j} = \frac{1}{2} \left(\xi_{j - \frac{1}{2}} + \xi_{j + \frac{1}{2}} \right)$. The time interval is defined as $\tau_{n + 1} = \tau_n + \Delta \tau_n$, where $\Delta \tau_n$ is the time step at time level $\tau_n$. Let $\boldsymbol{\mathcal{U}}_j^{n}$ denote the solution at grid point $\xi_j$ and time $\tau_n$. From Theorem \ref{infeasible}, it is a straightforward approach to design a first-order BP scheme for system \ref{newsystem} by setting $c \equiv v$. Define $\lambda = \Delta \tau_n / \Delta \xi$ and $c_j^n = v_j^n$. We design the following first-order scheme:
\begin{equation}\label{1stscheme}
    \boldsymbol{\mathcal{U}}_j^{n + 1} = \boldsymbol{\mathcal{U}}_j^{n} - \lambda \left(\widehat{\mathbf{g}}_{j + \frac{1}{2}}^* - \widehat{\mathbf{g}}_{j - \frac{1}{2}}^*\right) \Leftrightarrow \left\{
    \begin{aligned}
        (J \phi)_j^{n + 1} &= (J \phi)_j^{n}, \\
        (J y)_j^{n + 1} &= (J y)_j^{n}, \\
        J_j^{n + 1} &= J_j^n + \lambda ( v_{j + 1}^n - v_j^n ),
    \end{aligned}
    \right.
\end{equation}
where 
\begin{equation}
\label{g}
	\widehat{\mathbf{g}}_{j + \frac{1}{2}}^* = (0, 0, - v^n_{j + 1})^\top, \quad \widehat{\mathbf{g}}_{j - \frac{1}{2}}^* = (0, 0, - v^n_{j})^\top.
\end{equation}
We will now explore the feasibility of ensuring BP for this scheme \eqref{1stscheme}.
\begin{theorem}
\label{TH8}
	Assume that $\boldsymbol{\mathcal{U}}_j^n \in \Omega_1$, then $0 < \Delta \tau_n < \Delta \tau_n^{*}$ provided that $\boldsymbol{\mathcal{U}}_j^{n + 1} \in \Omega_1$, with 
	\begin{equation*}
	    \Delta \tau_n^{*} = \left\{ 
	    \begin{aligned}
	        &\frac{\Delta \xi }{v_j^n - v_{j + 1}^n} \min \{J_j^n - \epsilon, J_j^n (1 - \phi_j^n)\}, \quad &&{\rm if} \; v_j^n > v_{j + 1}^n,\\
	        &+ \infty \quad &&{\rm otherwise}.
	    \end{aligned}
	    \right.
	\end{equation*}
\end{theorem}

\begin{proof}
	If  $v_j^n > v_{j + 1}^n$, then 
	\begin{equation*}
	    J_j^{n + 1} = J_j^n + \frac{\Delta \tau_n}{\Delta \xi} \left( v_{j + 1}^n - v_j^n \right) > J_j^n + \frac{\Delta \tau_n^{*}}{\Delta \xi} \left( v_{j + 1}^n - v_j^n \right) \geq \epsilon.
	\end{equation*}
	It is straightforward to verify that for all $\Delta \tau_n > 0$, $J_j^{n + 1} > \epsilon$ when $v_j^n \leq v_{j + 1}^n$.
	According to \eqref{1stscheme}, given that $J_j^{n + 1} > \epsilon$ and $J_j^{n} > \epsilon$, we derive
	\begin{equation*}
	    \begin{aligned}
	        \phi_j^{n + 1} &= \frac{(J\phi)_j^{n + 1}}{J_j^{n + 1}} = \frac{J_j^{n}}{J_j^{n + 1}} \phi_j^{n} > 0 \\
	        k_j^{n + 1} &= \frac{(Jy)_j^{n + 1}}{(J\phi)_j^{n + 1}} = \frac{(Jy)_j^{n}}{(J\phi)_j^{n}} = \frac{y_j^n}{\phi_j^n} = k_j^n \in [k_{\min}, k_{\max}].
	    \end{aligned}
	\end{equation*}
	It should be noted that $\phi_j^{n + 1} < 1$ corresponds to the condition:  
	\begin{equation}
	\label{99}
		J_j^n \phi_j^n < J_j^{n + 1} = J_j^n + \frac{\Delta \tau_n}{\Delta \xi} \left( v_{j + 1}^n - v_j^n \right).
	\end{equation}
	Given that $\phi_j^{n} < 1$, \eqref{99} holds for any $\Delta \tau_n > 0$ when $v_{j + 1}^n \geq v_j^n$, and is correct for any $\Delta \tau_n \in (0, \Delta \tau_n^{*})$ when $v_{j + 1}^n < v_j^n$.
	In summary, $\boldsymbol{\mathcal{U}}_j^{n + 1} \in \Omega_1^*$. The proof is completed.
\end{proof}

Since a Temple-class system with varying expressions of $v$ results in different feasible maximum time steps, we will focus solely on the ARZ and sedimentation models in the subsequent discussion.
\begin{theorem}
	{\rm (ARZ model with $\gamma > 0$).} Assume that $\boldsymbol{\mathcal{U}}_j^{n} \in \Omega$ , then $0 < \Delta \tau_n < \min \{\Delta \tau_n^{*}, \Delta \tau_n^{**} \}$ provided that $\boldsymbol{\mathcal{U}}_j^{n + 1} \in \Omega$, with 
	\begin{equation}
	\label{09}
	    \Delta \tau_n^{**} = \left\{ 
	    \begin{aligned}
	        &\frac{\Delta \xi J_j^n (\phi_j^n - \check{\phi}_j^n)}{\check{\phi}_j^n (v_{j + 1}^n - v_{j}^n)}, \quad &&{\rm if} \; v_j^n < v_{j + 1}^n \; {\rm and} \; k_j^n > v_{\max},\\
	        &\frac{\Delta \xi J_j^n (\hat{\phi}_j^n - \phi_j^n)}{\hat{\phi}_j^n (v_{j}^n - v_{j + 1}^n)}, \quad &&{\rm if} \; v_j^n > v_{j + 1}^n,\\
	        &+ \infty \quad &&{\rm otherwise},
	    \end{aligned}
	    \right.
	\end{equation}
	where $\check{\phi}_j^n := \left(\frac{\gamma}{v_{\rm ref}} (k_j^n - v_{\max})\right)^{1 / \gamma}$ and $\hat{\phi}_j^n := \left(\frac{\gamma}{v_{\rm ref}} (k_j^n - v_{\min})\right)^{1 / \gamma}$.
\end{theorem}

\begin{proof}
	Since $\boldsymbol{\mathcal{U}}_j^{n} \in \Omega \subset \Omega_1$ and $0 < \Delta \tau_n < \Delta \tau_n^{*}$, we can obtain $\boldsymbol{\mathcal{U}}_j^{n + 1} \in \Omega_1$ from Theorem \ref{TH8}. 
	Note that $k_j^{n + 1} = k_j^{n}$ and define 
	\begin{equation*}
	\begin{aligned}
	f_{\max} &:= v_{\max} - v_j^{n + 1} = v_{\max} - k_j^{n + 1} + \frac{v_{\rm ref}}{\gamma} (\phi_j^{n + 1})^\gamma = v_{\max} - k_j^n + \frac{v_{\rm ref}}{\gamma} \left( \frac{J_j^n}{J_j^{n + 1}} \phi_j^n \right)^\gamma, \\
	f_{\min} &:= v_j^{n + 1} - v_{\min} = k_j^{n + 1} - \frac{v_{\rm ref}}{\gamma} (\phi_j^{n + 1})^\gamma - v_{\min} = k_j^n - \frac{v_{\rm ref}}{\gamma} \left( \frac{J_j^n}{J_j^{n + 1}} \phi_j^n \right)^\gamma - v_{\min}.
	\end{aligned}
	\end{equation*}
	
	{\sf Part I.} Keep $v_j^{n + 1} \leq v_{\max}$, namely, $f_{\max} \geq 0$.

	For any $k_j^n \leq v_{\max}$, we have $f_{\max} \geq 0$ based on the conditions $\phi_j^n$, $J_j^n$, and $J_j^{n + 1} > 0$.

	Considering the case $k_j^n > v_{\max}$, we can get 
	\begin{equation*}
	    \check{\phi}_j^n = \left(\frac{\gamma}{v_{\rm ref}} (k_j^n - v_{\max}) \right)^{1 / \gamma} \leq \left(\frac{\gamma}{v_{\rm ref}} (k_j^n - v_j^n)\right)^{1 / \gamma} = \phi_j^n,
	\end{equation*}
	and 
	\begin{equation*}
	    f_{\max} \geq 0 \Leftrightarrow \frac{J_j^n}{J_j^{n + 1}} \phi_j^n \geq \left(\frac{\gamma}{v_{\rm ref}} (k_j^n - v_{\max})\right)^{1 / \gamma} \Leftrightarrow J_j^n \phi_j^n \geq \left( J_j^n + \frac{\Delta \tau_n}{\Delta \xi} (v_{j + 1}^n - v_j^n) \right) \check{\phi}_j^n.
	\end{equation*}
	Noting that $\check{\phi}_j^n \leq \phi_j^n$, then $f_{\max} \geq 0$ is always true when $v_j^n \geq v_{j + 1}^n$. If $v_j^n < v_{j + 1}^n$, the condition $\Delta \tau_n \leq \frac{\Delta \xi J_j^n (\phi_j^n - \check{\phi}_j^n)}{\check{\phi}_j^n (v_{j + 1}^n - v_{j}^n)}$ implies that $f_{\max} \geq 0$.
	
	{\sf Part II.} Keep $v_j^{n + 1} \geq v_{\min}$, namely, $f_{\min} \geq 0$.

	Since $k_j^n = v_j^n + p(\phi_j^n) \geq v_j^n \geq v_{\min}$, we deduce
	\begin{equation*}
	    \hat{\phi}_j^n = \left(\frac{\gamma}{v_{\rm ref}} (k_j^n - v_{\min})\right)^{1 / \gamma} \geq \left(\frac{\gamma}{v_{\rm ref}} (k_j^n - v_j^n)\right)^{1 / \gamma} = \phi_j^n. 
	\end{equation*}
	According the relation 
	\begin{equation*}
	    f_{\min} \geq 0 \Leftrightarrow \left(\frac{\gamma}{v_{\rm ref}} (k_j^n - v_{\min})\right)^{1 / \gamma} \geq \frac{J_j^n}{J_j^{n + 1}} \phi_j^n \Leftrightarrow \left( J_j^n + \frac{\Delta \tau_n}{\Delta \xi} (v_{j + 1}^n - v_j^n) \right) \hat{\phi}_j^n \geq J_j^n \phi_j^n,
	\end{equation*}
	we can obtain $f_{\min} \geq 0$ holds for any $v_j^n \leq v_{j + 1}^n$, and is true under the condition $\Delta \tau_n \leq \frac{\Delta \xi J_j^n (\hat{\phi}_j^n - \phi_j^n)}{\hat{\phi}_j^n (v_{j}^n - v_{j + 1}^n)}$ when $v_j^n > v_{j + 1}^n$. 

	Therefore, it holds that $\boldsymbol{\mathcal{U}}_j^{n + 1} \in \Omega$ under the condition $0 < \Delta \tau_n < \min \{\Delta \tau_n^{*}, \Delta \tau_n^{**} \}$, and the proof is completed.
\end{proof}

\begin{remark}
	{\rm (Sedimentation model).} Assume that $\boldsymbol{\mathcal{U}}_j^{n} \in \Omega$, then $0 < \Delta \tau_n < \min \{\Delta \tau_n^{*}, \Delta \tau_n^{**} \}$ provided that $\boldsymbol{\mathcal{U}}_j^{n + 1} \in \Omega$, where $\Delta \tau_n^{**}$ has the same expression as \eqref{09}
	with $\check{\phi}_j^n := 1 - \sqrt{v_{\max} / k_j^n}$ and $\hat{\phi}_j^n := 1 - \sqrt{v_{\min} / k_j^n}$.
\end{remark}

\begin{remark}
	{\rm (ARZ model with $\gamma = 0$).} Assume that $\boldsymbol{\mathcal{U}}_j^{n} \in \Omega$, then $0 < \Delta \tau_n < \min \{\Delta \tau_n^{*}, \Delta \tau_n^{**} \}$ provided that $\boldsymbol{\mathcal{U}}_j^{n + 1} \in \Omega$, with 
	\begin{equation*}
	    \Delta \tau_n^{**} = \left\{ 
	    \begin{aligned}
	        &\frac{\Delta \xi J_j^n (\phi_j^n - \check{\phi}_j^n)}{\check{\phi}_j^n (v_{j + 1}^n - v_{j}^n)}, \quad &&{\rm if} \; v_j^n < v_{j + 1}^n ,\\
	        &\frac{\Delta \xi J_j^n (\hat{\phi}_j^n - \phi_j^n)}{\hat{\phi}_j^n (v_{j}^n - v_{j + 1}^n)}, \quad &&{\rm if} \; v_j^n > v_{j + 1}^n,\\
	        &+ \infty \quad &&{\rm otherwise},
	    \end{aligned}
	    \right.
	\end{equation*}
	where $\check{\phi}_j^n := \exp{\frac{k_j^n - v_{\max}}{v_{\rm ref}}}$ and $\hat{\phi}_j^n := \exp{\frac{k_j^n - v_{\min}}{v_{\rm ref}}}$.
\end{remark}

\section{High-order BP scheme}\label{0sec4}

The finite difference scheme evolves the point values of the solution in a conservative form
\begin{equation}
\label{134}
	\frac{d}{dt} \boldsymbol{\mathcal{U}}_j (\tau) = - \frac{1}{\Delta \xi} \left( \widehat{\mathbf{G}}_{j + \frac{1}{2}} - \widehat{\mathbf{G}}_{j - \frac{1}{2}} \right) =: \mathcal{L} \left(\boldsymbol{\mathcal{U}}_j \right),
\end{equation}
where the numerical flux can be reconstructed from neighboring flux functions with high order accuracy using weighted essentially non-oscillatory (WENO) reconstructions \cite{shu1998essentially,jiang1996efficient}. By adaptively assigning nonlinear weights to neighboring stencils, WENO reconstruction preserves the high-order accuracy of the linear scheme in smooth regions of the solution while ensuring sharp, non-oscillatory capture of discontinuities. The ODE form \eqref{134} can be discretized in time with, for example, the classic explicit third-order total variation diminishing (TVD) Runge–Kutta (RK) method \cite{gottlieb2001strong}: 
\begin{equation}
	\label{RK3}
	\begin{aligned}
		&\boldsymbol{\mathcal{U}}_j^{n, 1} = \boldsymbol{\mathcal{U}}_j^{n} + \Delta \tau_n \mathcal{L} \left(\boldsymbol{\mathcal{U}}_j^{n} \right), \\
		&\boldsymbol{\mathcal{U}}_j^{n, 2} = \boldsymbol{\mathcal{U}}_j^{n} + \frac{1}{4} \Delta \tau_n \left[ \mathcal{L} \left(\boldsymbol{\mathcal{U}}_j^{n} \right) + \mathcal{L} \left(\boldsymbol{\mathcal{U}}_j^{n, 1} \right) \right], \\
		&\boldsymbol{\mathcal{U}}_j^{n + 1} = \boldsymbol{\mathcal{U}}_j^{n} + \frac{1}{6} \Delta \tau_n \left[ \mathcal{L} \left(\boldsymbol{\mathcal{U}}_j^{n} \right) + \mathcal{L} \left(\boldsymbol{\mathcal{U}}_j^{n, 1} \right) + 4 \mathcal{L} \left(\boldsymbol{\mathcal{U}}_j^{n, 2} \right) \right].
	\end{aligned}
\end{equation}
Given that $\lambda = \Delta \tau_n / \Delta \xi$, \eqref{RK3} can be rewritten as 
\begin{equation}
\label{RKWENO}
    \boldsymbol{\mathcal{U}}_j^{n + 1} = \boldsymbol{\mathcal{U}}_j^{n} - \lambda \left(\widehat{\mathbf{G}}_{j + \frac{1}{2}}^* - \widehat{\mathbf{G}}_{j - \frac{1}{2}}^*\right)
\end{equation}
with 
\begin{equation}
\label{hflux}
	\widehat{\mathbf{G}}_{j \pm \frac{1}{2}}^* =\frac{1}{6} \widehat{\mathbf{G}}_{j \pm \frac{1}{2}}^{n} + \frac{1}{6}\widehat{\mathbf{G}}_{j \pm \frac{1}{2}}^{n,1} + \frac{2}{3}\widehat{\mathbf{G}}_{j \pm \frac{1}{2}}^{n,2}.
\end{equation} 
In general, the numerical results of finite difference RK WENO scheme \eqref{RKWENO} may not satisfy the invariant domain $\Omega$. 
To ensure this BP property, we utlize a parametrized BP flux limiter to modify the high-order flux $\widehat{\mathbf{G}}_{j \pm \frac{1}{2}}^*$ in \eqref{hflux} towards to the first-order flux $\widehat{\mathbf{g}}_{j \pm \frac{1}{2}}^*$ in \eqref{1stscheme}, resulting in
\begin{equation}
\label{G}
    \widetilde{\mathbf{G}}_{j \pm \frac{1}{2}} := \theta_{j \pm \frac{1}{2}} \left(\widehat{\mathbf{G}}_{j \pm \frac{1}{2}}^* - \widehat{\mathbf{g}}_{j \pm \frac{1}{2}}^*\right) + \widehat{\mathbf{g}}_{j \pm \frac{1}{2}}^*, \quad \theta_{j \pm \frac{1}{2}} \in [0, 1].
\end{equation}
The original high-order flux $\widehat{\mathbf{G}}_{j \pm \frac{1}{2}}^*$ in \eqref{RKWENO} is then replaced by the modified flux $\widetilde{\mathbf{G}}_{j \pm \frac{1}{2}}$, namely
\begin{equation}
\label{fl}
    \boldsymbol{\mathcal{U}}_j^{n + 1} = \boldsymbol{\mathcal{U}}_j^{n} - \lambda \left(\widetilde{\mathbf{G}}_{j + \frac{1}{2}} - \widetilde{\mathbf{G}}_{j - \frac{1}{2}}\right)
\end{equation} above.
In the following, we will introduce the recipe of approximate invariant domain ${\Omega}_j^{n}$ for numerical solution at point $(\xi_j, \tau_n)$ in \Cref{sec:1596} and the parametrized BP limiter in \Cref{sec:1594}.

\subsection{Approximate invariant domains}\label{sec:1596}

In this section, we introduce the approximation of invariant domain $\Omega^{n}_j$ for $j = 1,\dots,N_\xi$.
In Subection~\ref{sec2}, we introduced two types of invariant domains:
\begin{enumerate}[leftmargin=8mm]
	\item[(I)] The invariant domain $\mathcal{I}^{\rm L}$, as defined in \eqref{eq:154}, where the lower and upper bounds of the Riemann invariants are determined by the minimum and maximum of the Riemann invariants within the local domain of dependence, $\mathcal{X}$, respectively.
	\item[(II)] The invariant domain $\mathcal{I}^{\rm G}$, as defined in \eqref{eq:328}, where the lower and upper bounds of the Riemann invariants are determined by the global minimum and maximum of the Riemann invariants at time $t = 0$, respectively.
\end{enumerate}
In a similar manner, the upper and lower bounds of the Riemann invariants in the definition of the invariant domain $\Omega$ (see \eqref{eq:270}) can be approximated using both local and global approaches, which are discussed in the following.

Define $E_j^n = [x_{j - 1}^n, x_{j + 1}^n]$, where $x_j^n := x_j^{n - 1} + c_j^{n - 1} \Delta \tau_{n - 1}$ represents the physical point corresponding to the computational grid point $\xi_j$ at time $\tau_n$. Let $\mathcal{C}: C(E_j^n) \mapsto \mathbb{P}^2(E_j^n)$ denote a quadratic interpolation operator over $E_j^n$, such that $[\mathcal{C}f](x) = f(x)$, $x \in \big\{x_{j-1}^n,x_{j}^n,x_{j+1}^n\big\}$.
At each time level $\tau_{n + 1}$, we estimate the local invariant domain
\begin{equation}\label{eq:1795}
	\Omega^{n + 1}_j :=
	\left\{
	\boldsymbol{\mathcal{U}} \in \mathbb{R}^3
	:
	\renewcommand\arraystretch{1.2}
	\begin{array}{r}
		J > \epsilon \quad, v \in \left[ (v_{\min})^{n + 1}_j, (v_{\max})^{n + 1}_j \right] \\
		\phi \in (0, 1),	k \in \left[ (k_{\min})^{n + 1}_j, (k_{\max})^{n + 1}_j  \right]
	\end{array}
	\right\}
\end{equation}
where
\begin{equation}
\label{eq:1745}
	\begin{aligned}
		(v_{\min})^{n + 1}_{j} &= \max\big\{\min\limits_{E_j} \mathcal{C} v\big(\boldsymbol{\mathcal{U}}_j^{n}\big) - \epsilon, 0 \big\}, \quad &&(v_{\max})^{n + 1}_{j} = \max\limits_{E_j} \mathcal{C} v\big(\boldsymbol{\mathcal{U}}_j^{n}\big) + \epsilon, \\
		(k_{\min})^{n + 1}_{j} &= \min\limits_{E_j} \mathcal{C} k\big(\boldsymbol{\mathcal{U}}_j^{n}\big) - \epsilon, \quad &&(k_{\max})^{n + 1}_{j} = \max\limits_{E_j} \mathcal{C} k\big(\boldsymbol{\mathcal{U}}_j^{n}\big)+ \epsilon. 
	\end{aligned}
\end{equation}
\begin{remark}
	The approximate local invariant domain $\Omega_j^{n+1}$ may not exactly coincide with the true global invariant domain $\Omega_j^{n+1,*}$. If the difference $\Omega_j^{n+1,*} \backslash \Omega_j^{n+1}$ is non-empty, enforcing the BP property with respect to $\Omega_j^{n+1}$ might be overly restrictive. In such cases, a reference solution with values overlapping $\Omega_j^{n+1,*} \backslash \Omega_j^{n+1,}$ could undergo unnecessary modifications by the BP limiter, potentially degrading the accuracy of the numerical results. To mitigate this issue, we introduce a small positive parameter $\epsilon$ in the formula \eqref{eq:1745} to slightly enlarge $\Omega_j^{n+1}$. In all the numerical experiments, we set $\epsilon = 10^{-12}$.
\end{remark}
When considering global invariant domains, the invariant domains are uniform across all cells. At the beginning of computation, we estimate the initial global invariant domain 
\begin{equation*}
	\Omega^{1} = \bigcup_{j = 1}^{N_\xi} \Omega^{1}_j,
\end{equation*}
with the definition \eqref{eq:1795}. 
The global invariant domain should be adjusted to accommodate boundary conditions. 
We update the approximate global invariant domain using the following formula:
\begin{equation}
\label{764}
	\Omega^{n + 1} = \Omega^{n} \cup \Omega^{n + 1}_{1} \cup \Omega^{n + 1}_{N_\xi}.
\end{equation}
As demonstrated in \eqref{764}, estimating the global invariant domain for each temporal layer (barring the initial one) requires only the additional determination of $\Omega^{n + 1}_{1}$ and $\Omega^{n + 1}_{N_\xi}$. This significantly reduces computational costs compared to estimating the local invariant domain. However, the global invariant domain estimates are overly broad, resulting in calculations that are markedly inferior to those of the local invariant domain, as illustrated in Example \ref{net1}.
\begin{remark}
	The previous BP analysis for first-order scheme \eqref{1stscheme} are based on global invariant domains. When considering local invariant domains, we should confirm the time step satisfies 
	\begin{equation*}
		0 < \Delta \tau_n < \min_{j = 1}^{N_\xi} \{\min \{\Delta \tau_{n,j}^*, \Delta \tau_{n,j}^{**}\}\},
	\end{equation*}
	where $\Delta \tau_{n,j}^*$ and $\Delta \tau_{n,j}^{**}$ are $\Delta \tau_{n}^*$ and $\Delta \tau_{n}^{**}$ at $\xi_j$, respectively. 
\end{remark}

\subsection{Parametrized BP flux limiter}\label{sec:1594}

In the following, we will demonstrate how to select $\boldsymbol{\theta} := (\theta_{j - \frac12}, \theta_{j + \frac12}) \in [0, 1]^2$ to ensure that $\boldsymbol{\mathcal{U}}_j^{n + 1} \in \Omega_j^{n + 1}$ given that $\boldsymbol{\mathcal{U}}_j^{n}$ resides within the local invariant domain $\Omega_j^{n + 1}$. 
The global invariant domain can be considered as a special case of the local invariant domain; for brevity, it is omitted here.
Define the function 
\begin{equation*}
	 h (\boldsymbol{\mathcal{U}}; s) := \left\{
	 \begin{aligned}
	 	&Jy - J\phi \left(s + p \left(J\phi / J\right) \right) \quad {\rm for} \; {\rm ARZ},\\
	 	&(Jy) p \left(J\phi / J\right) - (J\phi) s \quad {\rm for} \; {\rm Sedimentation},
 	\end{aligned}
 	\right.
\end{equation*}
where $s$ is a constant. It is important to note that $h (\boldsymbol{\mathcal{U}}; s) = 0$ is equal to $v(\boldsymbol{\mathcal{U}}) = s$. Since $\Omega_j^{n + 1}$ is convex, then the function $v(\boldsymbol{\mathcal{U}}) \equiv (v_{\min})_j^{n + 1}$ is concave when $J > \epsilon$ and $\phi > 0$. Consequently, by replacing $(v_{\min})_j^{n + 1}$ with any constant $s$, the function $h (\boldsymbol{\mathcal{U}}; s) = 0$ also exhibits concavity under the conditions $J > \epsilon$ and $\phi > 0$.
Notably, with the definition \eqref{g} and \eqref{G}, then \eqref{fl} can be rewritten as
\begin{equation}
\label{124}
	\left\{
    \begin{aligned}
        (J \phi)_j^{n + 1} &= (J \phi)_j^{n} - \lambda \left(\theta_{j + \frac12} \widehat{\mathbf{G}}_{j + \frac{1}{2}}^{*,(1)} - \theta_{j - \frac12} \widehat{\mathbf{G}}_{j - \frac{1}{2}}^{*,(1)} \right), \\
        (J y)_j^{n + 1} &= (J y)_j^{n} - \lambda \left(\theta_{j + \frac12} \widehat{\mathbf{G}}_{j + \frac{1}{2}}^{*,(2)} - \theta_{j - \frac12} \widehat{\mathbf{G}}_{j - \frac{1}{2}}^{*,(2)} \right), \\
        J_j^{n + 1} &= J_j^n - \lambda \left( \left(\theta_{j + \frac12} \left( \widehat{\mathbf{G}}_{j + \frac{1}{2}}^{*,(3)} + v_{j + 1}^n \right) -  v_{j + 1}^n \right) - \left(\theta_{j - \frac12} \left(\widehat{\mathbf{G}}_{j - \frac{1}{2}}^{*,(3)} +  v_{j}^n \right) -  v_{j}^n \right)\right).
    \end{aligned}
    \right.
\end{equation}

\begin{enumerate}[label=(\roman*),leftmargin=6mm]
	
	\item[\textbf{Step 1}:]	Enforce the constraints:
	\begin{equation}
	\label{C1}
		\left\{
		\begin{aligned}
			&J_j^{n + 1} \geq \epsilon, \quad (J\phi)_j^{n + 1} \in (0, J_j^{n + 1}), \\
			&(Jy)_j^{n + 1} \in [(k_{\min})_j^{n + 1}(J\phi)_j^{n + 1}, (k_{\max})_j^{n + 1}(J\phi)_j^{n + 1}],
		\end{aligned}
		\right.
	\end{equation}
	which are equivalent to $J_j^{n + 1} \geq \epsilon$, $\phi_j^{n + 1} \in (0, 1)$ , and $k_j^{n + 1} \in [(k_{\min})_j^{n + 1}, (k_{\max})_j^{n + 1}]$.
	
	Given that 
	\begin{equation*}
	    \begin{aligned}
	    \Gamma_j^{[1]} &:= J_j^n + \lambda ( v_{j + 1}^n - v_j^n ) - \epsilon, \quad \Gamma_j^{[2]} :=  J_j^n + \lambda ( v_{j + 1}^n - v_j^n ) - (J\phi)_j^n, \\
	    \Gamma_j^{[4]} &:= (J\phi)_j^n, \quad \Gamma_j^{[4]} := (Jy)_j^n - (k_{\min})_j^{n + 1} (J\phi)_j^n, \quad \Gamma_j^{[5]} := (k_{\max})_j^{n + 1} (J\phi)_j^n - (Jy)_j^n,
	    \end{aligned}
	\end{equation*}
	since the first-order scheme \eqref{1stscheme} satisfies the constraints in \eqref{C1}, we will have $\Gamma_j^{[i]} \geq 0$, for $i = 1,2,\cdots, 5.$
	For the high-order scheme \eqref{124}, the constraints \eqref{C1} imply that suitable values of $\Lambda_{\pm \frac{1}{2}, I_j}^1$ should be found such that for all $0 \leq \theta_{j \pm \frac12} \leq \Lambda_{\pm \frac{1}{2}, I_j}^1$, the following equations hold:

	\begin{equation}\label{step1}
	    \begin{aligned}
	    	\mathcal{T}_j^{[i]} := \Gamma_j^{[i]} - \lambda \left( \theta_{j + \frac{1}{2}} \widehat{\mathbf{G}}_{j + \frac{1}{2}}^{\dag,[i]} - \theta_{j - \frac{1}{2}} \widehat{\mathbf{G}}_{j - \frac{1}{2}}^{\dag,[i]} \right) \geq 0, \quad {\rm for} \; i = 1, 2, \cdots, 5,
	    \end{aligned}
	\end{equation}
	where 
	\begin{equation*}
		\begin{aligned}
			&\widehat{\mathbf{G}}_{j \pm \frac{1}{2}}^{\dag,[1]} = \widehat{\mathbf{G}}_{j \pm \frac{1}{2}}^{*,(3)} + v_{j \pm \frac12 + \frac12}^n, \quad 
			\widehat{\mathbf{G}}_{j \pm \frac{1}{2}}^{\dag,[2]} = \widehat{\mathbf{G}}_{j \pm \frac{1}{2}}^{*,(3)} + v_{j \pm \frac12 + \frac12}^n - \widehat{\mathbf{G}}_{j \pm \frac{1}{2}}^{*,(1)}, \\
			&\widehat{\mathbf{G}}_{j \pm \frac{1}{2}}^{\dag,[3]} = \widehat{\mathbf{G}}_{j \pm \frac{1}{2}}^{*,(1)}, \quad
			\widehat{\mathbf{G}}_{j \pm \frac{1}{2}}^{\dag,[4]} = \widehat{\mathbf{G}}_{j \pm \frac{1}{2}}^{*,(2)} - (k_{\min})_j^{n + 1}\widehat{\mathbf{G}}_{j \pm \frac{1}{2}}^{*,(1)}, \quad
			\widehat{\mathbf{G}}_{j \pm \frac{1}{2}}^{\dag,[5]} = (k_{\max})_j^{n + 1}\widehat{\mathbf{G}}_{j \pm \frac{1}{2}}^{*,(1)} - \widehat{\mathbf{G}}_{j \pm \frac{1}{2}}^{*,(2)}.
		\end{aligned}
	\end{equation*}
	The decoupling of each $\mathcal{T}_j^{[i]}$ on cell $I_j$ gives:
	\begin{enumerate}
		\item[(a)] If $\widehat{\mathbf{G}}_{j - \frac{1}{2}}^{\dag,[i]} \geq 0$ and $\widehat{\mathbf{G}}_{j + \frac{1}{2}}^{\dag,[i]} \leq 0$,
		let $(\Lambda_{- \frac{1}{2}, I_j}^{[i]}, \Lambda_{+ \frac{1}{2}, I_j}^{[i]}) = (1, 1);$
		\item[(b)] If $\widehat{\mathbf{G}}_{j - \frac{1}{2}}^{\dag,[i]} \geq 0$ and $\widehat{\mathbf{G}}_{j + \frac{1}{2}}^{\dag,[i]} > 0$,
		let $(\Lambda_{- \frac{1}{2}, I_j}^{[i]}, \Lambda_{+ \frac{1}{2}, I_j}^{[i]}) = (1, \min (1, \frac{-\Gamma_j^{[i]}}{-\lambda \widehat{\mathbf{G}}_{j + \frac{1}{2}}^{\dag,[i]} - \epsilon_1}));$
		\item[(c)] If $\widehat{\mathbf{G}}_{j - \frac{1}{2}}^{\dag,[i]} < 0$ and $\widehat{\mathbf{G}}_{j + \frac{1}{2}}^{\dag,[i]} \leq 0$,
		let $(\Lambda_{- \frac{1}{2}, I_j}^{[i]}, \Lambda_{+ \frac{1}{2}, I_j}^{[i]}) = (\min (1, \frac{-\Gamma_j^{[i]}}{\lambda \widehat{\mathbf{G}}_{j - \frac{1}{2}}^{\dag,[i]} - \epsilon_1}), 1);$
		\item[(d)] If $\widehat{\mathbf{G}}_{j - \frac{1}{2}}^{\dag,[i]} < 0$ and $\widehat{\mathbf{G}}_{j + \frac{1}{2}}^{\dag,[i]} > 0$, \\
		\begin{itemize}
			\item[--] \vspace{-1em} if $(\theta_{j - \frac12}, \theta_{j + \frac12}) = (1, 1)$ satisfies $\mathcal{T}_j^{[i]} \geq 0$, let $(\Lambda_{- \frac{1}{2}, I_j}^{[i]}, \Lambda_{+ \frac{1}{2}, I_j}^{[i]}) = (1, 1)$;
			\item[--] otherwise, let $(\Lambda_{- \frac{1}{2}, I_j}^{[i]}, \Lambda_{+ \frac{1}{2}, I_j}^{[i]}) = (\frac{-\Gamma_j^{[i]}}{\lambda \widehat{\mathbf{G}}_{j - \frac{1}{2}}^{\dag,[i]} - \lambda \widehat{\mathbf{G}}_{j + \frac{1}{2}}^{\dag,[i]} - \epsilon_1}, \frac{-\Gamma_j^{[i]}}{\lambda \widehat{\mathbf{G}}_{j - \frac{1}{2}}^{\dag,[i]} - \lambda \widehat{\mathbf{G}}_{j + \frac{1}{2}}^{\dag,[i]} - \epsilon_1})$.
		\end{itemize}
	\end{enumerate}

	Here, $\epsilon_1 = 10^{-13}$ is a small positive number to avoid the denominator to be zero. Setting $\Lambda_{\pm \frac{1}{2}, I_j}^{1} = \min_{i = 1}^{5} \Lambda_{\pm \frac{1}{2}, I_j}^{[i]}$, we can define a set for the \eqref{step1}
	\begin{equation*}
	    S_1 = \left\{(\boldsymbol{\theta}: 0 \leq \theta_{j - \frac{1}{2}} \leq \Lambda_{- \frac{1}{2}, I_j}^1, 0 \leq \theta_{j + \frac{1}{2}} \leq \Lambda_{+ \frac{1}{2}, I_j}^1  \right\},
	\end{equation*}
	which is plotted as the rectangle bounded by the dash line in Figure \ref{fig:sub1}.
	
	\item[\textbf{Step 2}:]Enforce the constraint $h_{j}^{n + 1} (\boldsymbol{\theta}; (v_{\min})_j^{n + 1}) \geq 0$, namely, $v_j^{n + 1} \geq (v_{\min})_j^{n + 1}$ where
	\begin{equation*}
		 h_{ j}^{n + 1} (\boldsymbol{\theta}; s) := h (\boldsymbol{\mathcal{U}}_j^{n + 1}; s).
	\end{equation*}
	Due to the linear dependence of $\boldsymbol{\mathcal{U}}_j^{n + 1}$ on $\boldsymbol{\theta}$ and the concavity of $h (\boldsymbol{\mathcal{U}}_j^{n + 1}; s)$, we observe that $h_j^{n + 1}(\boldsymbol{\theta}; (v_{\min})_j^{n + 1})$ is a concave function over the set $S_1$. Furthermore, if $h_{j}^{n + 1}(\boldsymbol{\theta}^l; (v_{\min})_j^{n + 1}) \geq 0$ for $l = 1, 2$, with $\boldsymbol{\theta}^l = (\theta^l_{j - \frac{1}{2}}, \theta^l_{j + \frac{1}{2}})$, then $h_{j}^{n + 1}(\alpha \boldsymbol{\theta}^1 + (1 - \alpha) \boldsymbol{\theta}^2; (v_{\min})_j^{n + 1}) \geq 0$ for $\alpha \in [0, 1]$. The vertices of rectangle $S_1$, excluding $(0, 0)$, are defined as:
	\begin{equation*}
	    a_1 = (\Lambda_{- \frac{1}{2}, I_j}^1, 0), \quad a_2 = (\Lambda_{- \frac{1}{2}, I_j}^1, \Lambda_{+ \frac{1}{2}, I_j}^1), \quad a_3 = (0, \Lambda_{+ \frac{1}{2}, I_j}^1).
	\end{equation*}
	\begin{enumerate}
		\item[(a)] Let $r_i = 1$, if $h_{j}^{n + 1}(a_i; (v_{\min})_j^{n + 1}) \geq 0$ for $i = 1, 2, 3$; otherwise find $r_i \in [0, 1]$, such that $h_{j}^{n + 1}(r_i a_i; (v_{\min})_j^{n + 1}) \geq 0$ and set $b_i = r_i a_i$. The convex polygonal region $Ob_1b_2b_3$ (denoted as $P$) must be inside $S_1$.
		\item[(b)] Then, we define 
		\begin{equation*}
		    S_2 = \left\{\boldsymbol{\theta}: 0 \leq \theta_{j - \frac{1}{2}} \leq \Lambda_{- \frac{1}{2}, I_j}^2, 0 \leq \theta_{j + \frac{1}{2}} \leq \Lambda_{+ \frac{1}{2}, I_j}^2  \right\},
		\end{equation*}
		with
		\begin{equation*}
		    \Lambda_{- \frac{1}{2}, I_j}^2 = \min \{ r_1, r_2 \}\Lambda_{- \frac{1}{2}, I_j}^1, \quad \Lambda_{+ \frac{1}{2}, I_j}^2 = \min \{ r_2, r_3 \}\Lambda_{+ \frac{1}{2}, I_j}^1.
		\end{equation*}
		As shown in Figure \ref{fig:sub1}, $S_2$ is the largest rectangle included in $P$ which satisfies $v_j^{n + 1} \geq (v_{\min})_j^{n + 1}$.
		
	\end{enumerate}

	\item[\textbf{Step 3}:]Enforce the constraint $h_{j}^{n + 1} (\boldsymbol{\theta}; (v_{\max})_j^{n + 1}) \leq 0$, which is equal to $\Omega_{{\max};j}^{n+1} := \{\boldsymbol{\theta}: v_j^{n + 1} \leq (v_{\max})_j^{n + 1}\}$. \\
	We note that $h_{j}^{n + 1}(\boldsymbol{\theta}^l; (v_{\max})_j^{n + 1}) \leq 0$ cannot obtain $h_{j}^{n + 1}(\alpha \boldsymbol{\theta}^1 + (1 - \alpha) \boldsymbol{\theta}^2; (v_{\max})_j^{n + 1}) \leq 0$ anymore. Let $m(l)$ be the slope of a straight line $l$, and $C$ as the curve $h_{j}^{n + 1}(\boldsymbol{\theta}; (v_{\max})_j^{n + 1}) = 0$. Define the index set $\mathbb{I}$ as consisting of indices $j$ for which $h_{j}^{n + 1} (\Lambda_{- \frac{1}{2}, I_j}^2, \Lambda_{+ \frac{1}{2}, I_j}^2; (v_{\max})_j^{n + 1}) > 0$, and set 
	\begin{equation*}
		\mathbb{K}_j = \{j - 1, j + 1\} \cap \{ 1, 2, \cdots, N_\xi \}. 
	\end{equation*}
	For each $j \in \mathbb{I}$, do the following loop until $\mathbb{I} = \emptyset$.
	\begin{enumerate}
		\item[(a)] Let $a = (\Lambda_{- \frac{1}{2}, I_j}^2, \Lambda_{+ \frac{1}{2}, I_j}^2)$, the straight line $Oa$ intersects the curve $C$ , which is convex, at only one point $b_1$. Make the tangent $l_1$ to the curve $C$ at point $b_1$.

		\item[(b)] If $m(l_1) \in (-\inf, 0)$. Let $(\Lambda_{- \frac{1}{2}, I_j}^3, \Lambda_{+ \frac{1}{2}, I_j}^3) = \boldsymbol{\theta}(b_1)$, as shown in Figure \ref{fig:sub2}.

		\item[(c)] If $m(l_1) \in (0, m(Oa))$. Define the point $b_2 := (\theta_{j - \frac12}(b_1), 0)$, and determine whether $b_2$ is in the set $\Omega_{{\max};j}^{n+1}$.
		\begin{enumerate}
			\item[(i)] If $b_2 \notin \Omega_{{\max};j}^{n+1}$. The curve $C$ will intersect the line $Ob_2$ at point $b_3$. Then the tangent $l_2$ to $C$ at point $b_3$ will intersect the line $l_1$ at point $b_4$, as shown in Figure \ref{fig:sub3}.
			\item[(ii)] If $b_2 \in \Omega_{{\max};j}^{n+1}$. The curve $C$ will intersect the line $b_1b_2$ at point $b_3$. 
			If the tangent $l_2$ to $C$ at point $b_3$ can intersect the line $l_1$ at a point $b_* \in \Omega_{{\max};j}^{n+1}$ with property ${\theta}_{j - \frac12}(b_*) \leq {\theta}_{j - \frac12}(b_1)$. As shown in Figure \ref{fig:sub4}, we set $b_4 = b_*$. Otherwise, let $b_4$ be the point as the intersection of lines $Ob_2$ and $l_1$, as shown in Figure \ref{fig:sub5}. 
		\end{enumerate}
		Then, we can set $(\Lambda_{- \frac{1}{2}, I_j}^3, \Lambda_{+ \frac{1}{2}, I_j}^3) = ({\theta}_{j - \frac12}(b_4), \Lambda_{+ \frac{1}{2}, I_j}^2)$. 

		\item[(d)] If $m(l_1) \in (m(Oa), + \inf)$. This case is a mirror image of case (c), here we omit the exact process, and set $(\Lambda_{- \frac{1}{2}, I_j}^3, \Lambda_{+ \frac{1}{2}, I_j}^3) = (\Lambda_{- \frac{1}{2}, I_j}^2, {\theta}_{j + \frac12}(b_4))$. 

		\item[(e)] If $m(l_1) = + \inf$. Let $(\Lambda_{- \frac{1}{2}, I_j}^3, \Lambda_{+ \frac{1}{2}, I_j}^3) = (({\theta}_{j - \frac12}(b_1), \Lambda_{+ \frac{1}{2}, I_j}^2)$.

		\item[(f)] If $m(l_1) = 0$. Let $(\Lambda_{- \frac{1}{2}, I_j}^3, \Lambda_{+ \frac{1}{2}, I_j}^3) = (\Lambda_{- \frac{1}{2}, I_j}^2, {\theta}_{j + \frac12}(b_1))$.

		\item[(g)] Let $\mathbb{K}_j^*$ be the subset of indices $k$ from $\mathbb{K}_j$ that do not satisfy $v_k^{n + 1} \leq (v_{\max})_k^{n + 1}$. Upadte the index set $\mathbb{I} = \mathbb{I} / \{j\} \cup \mathbb{K}_j^*$.
	\end{enumerate}
	Finally, we define 
	\begin{equation*}
	    S_3 = \left\{\boldsymbol{\theta}: 0 \leq \theta_{j - \frac{1}{2}} \leq \Lambda_{- \frac{1}{2}, I_j}^3, 0 \leq \theta_{j + \frac{1}{2}} \leq \Lambda_{+ \frac{1}{2}, I_j}^3  \right\}.
	\end{equation*}
\end{enumerate}
\begin{figure}[hbtp]
\centering

\begin{subfigure}[b]{0.49\textwidth}
    \centering
    \includegraphics[width = 0.653\textwidth,trim=70 0 95 30,clip]{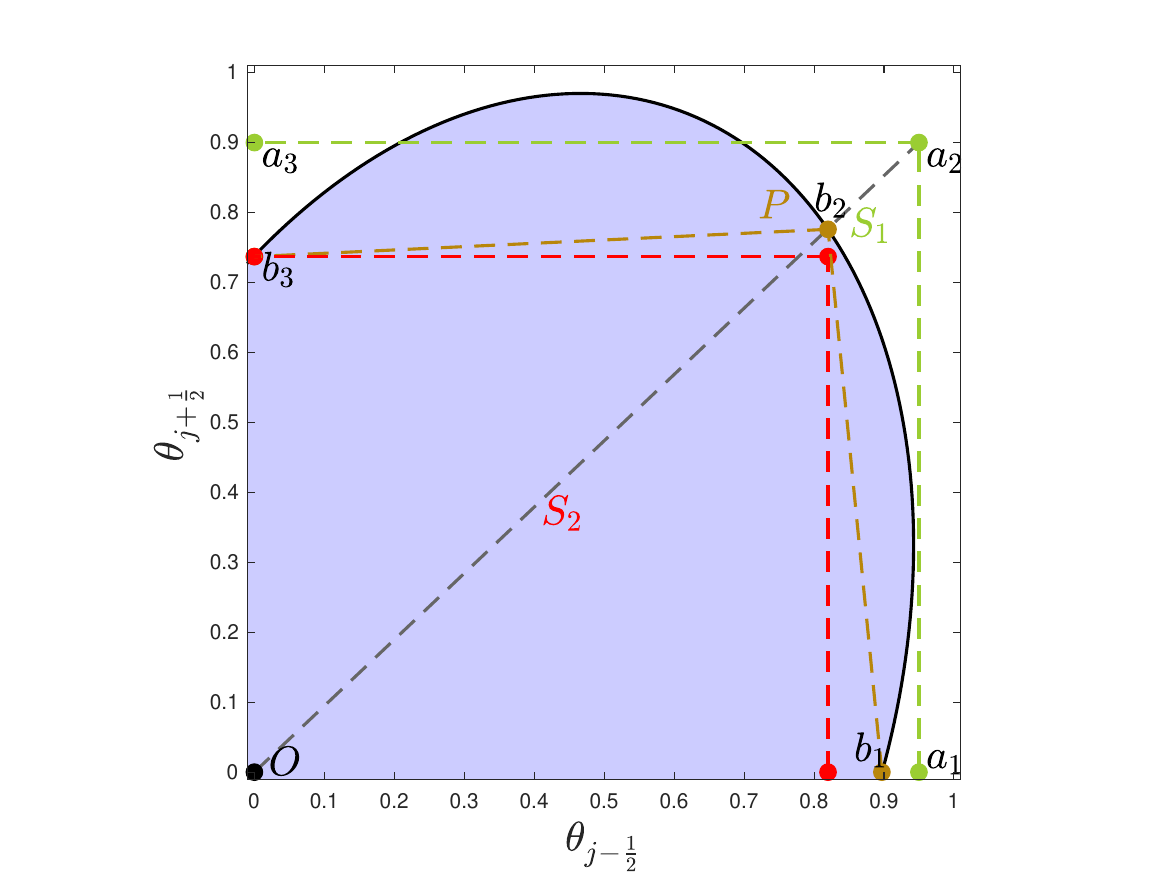}
    \caption{Limit $v_j^{n+1} \geq (v_{\min})_j^{n+1}$.}
    \label{fig:sub1}
\end{subfigure}
\hfill
\begin{subfigure}[b]{0.49\textwidth}
    \centering
    \includegraphics[width = 0.653\textwidth,trim=70 0 95 30,clip]{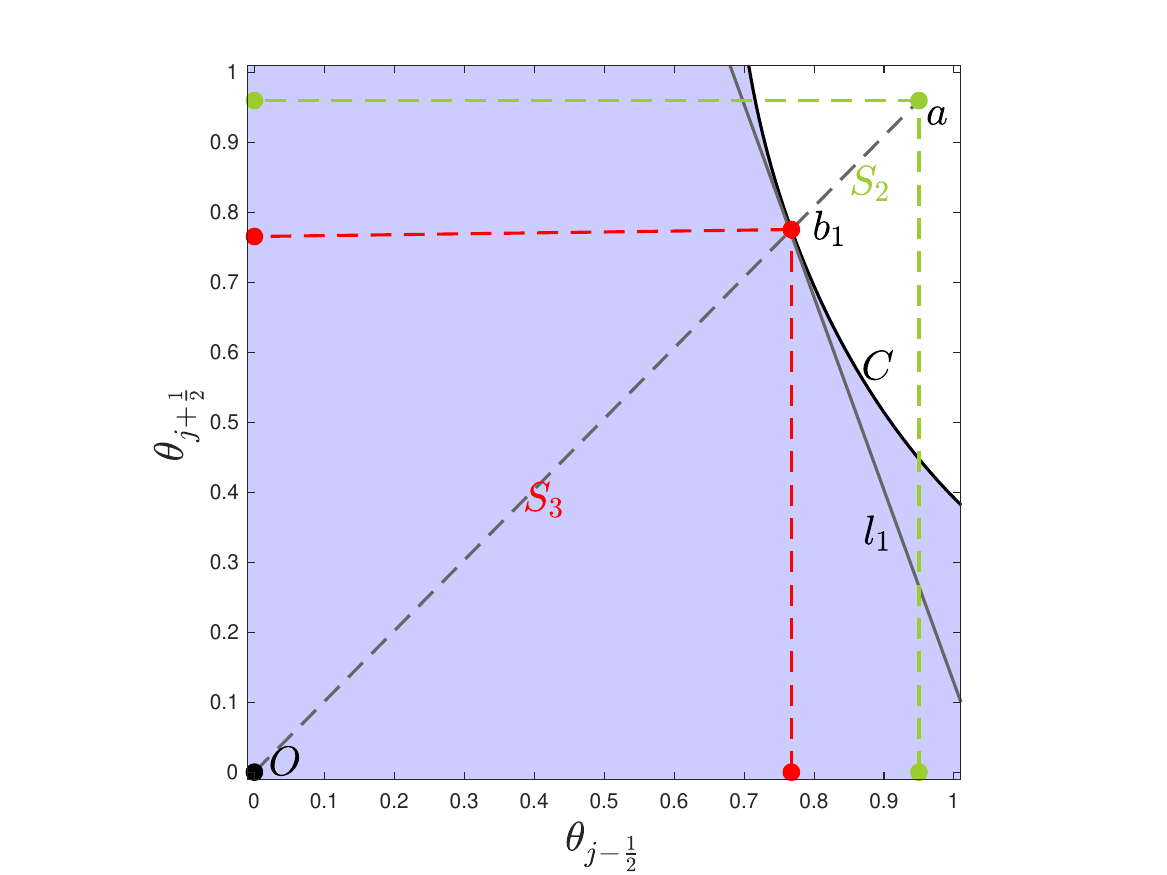}
    \caption{Limit $v_j^{n+1} \leq (v_{\max})_j^{n+1}$, case(b).}
    \label{fig:sub2}
\end{subfigure}

\vspace{1em}

\begin{subfigure}[b]{0.32\textwidth}
    \centering
    \includegraphics[width = \textwidth,trim=70 0 95 30,clip]{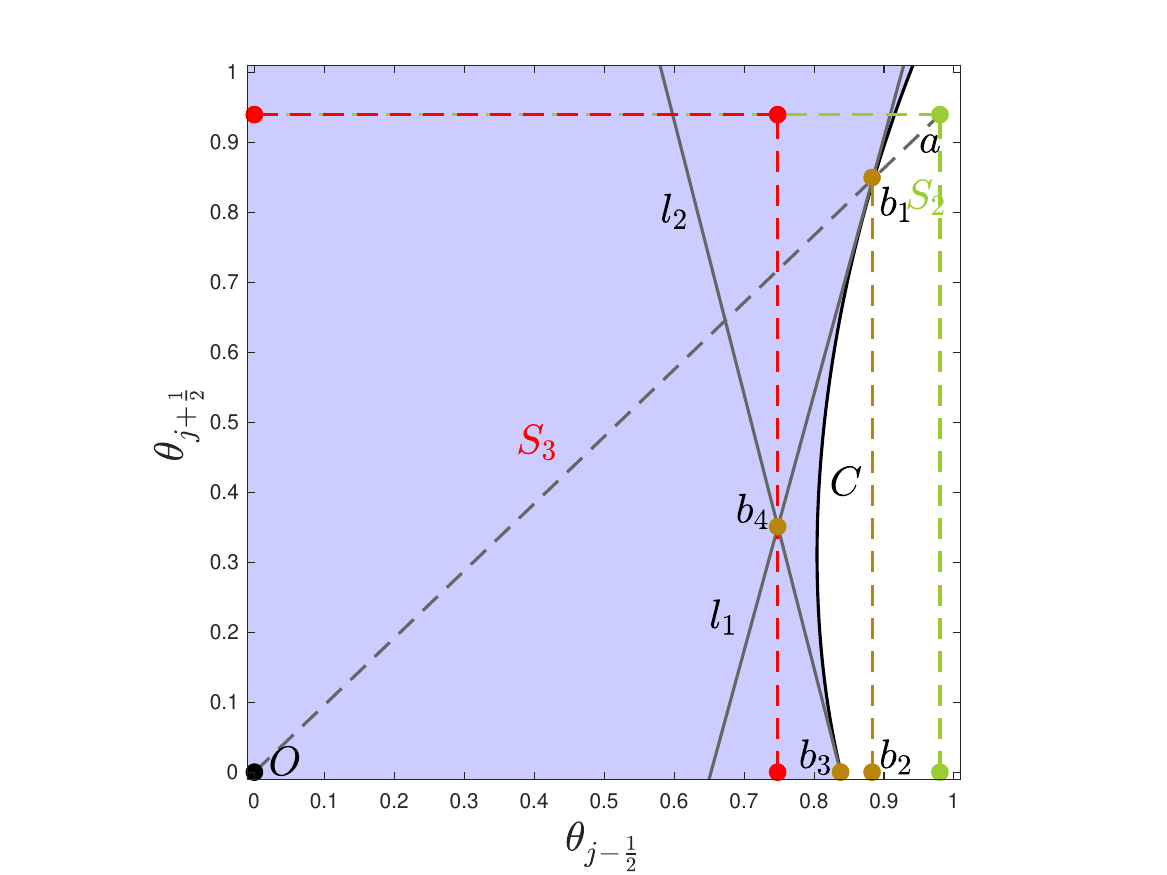}
    \caption{Limit $v_j^{n+1} \leq (v_{\max})_j^{n+1}$, case(c)(i).}
    \label{fig:sub3}
\end{subfigure}
\hfill
\begin{subfigure}[b]{0.32\textwidth}
    \centering
    \includegraphics[width = \textwidth,trim=70 0 95 30,clip]{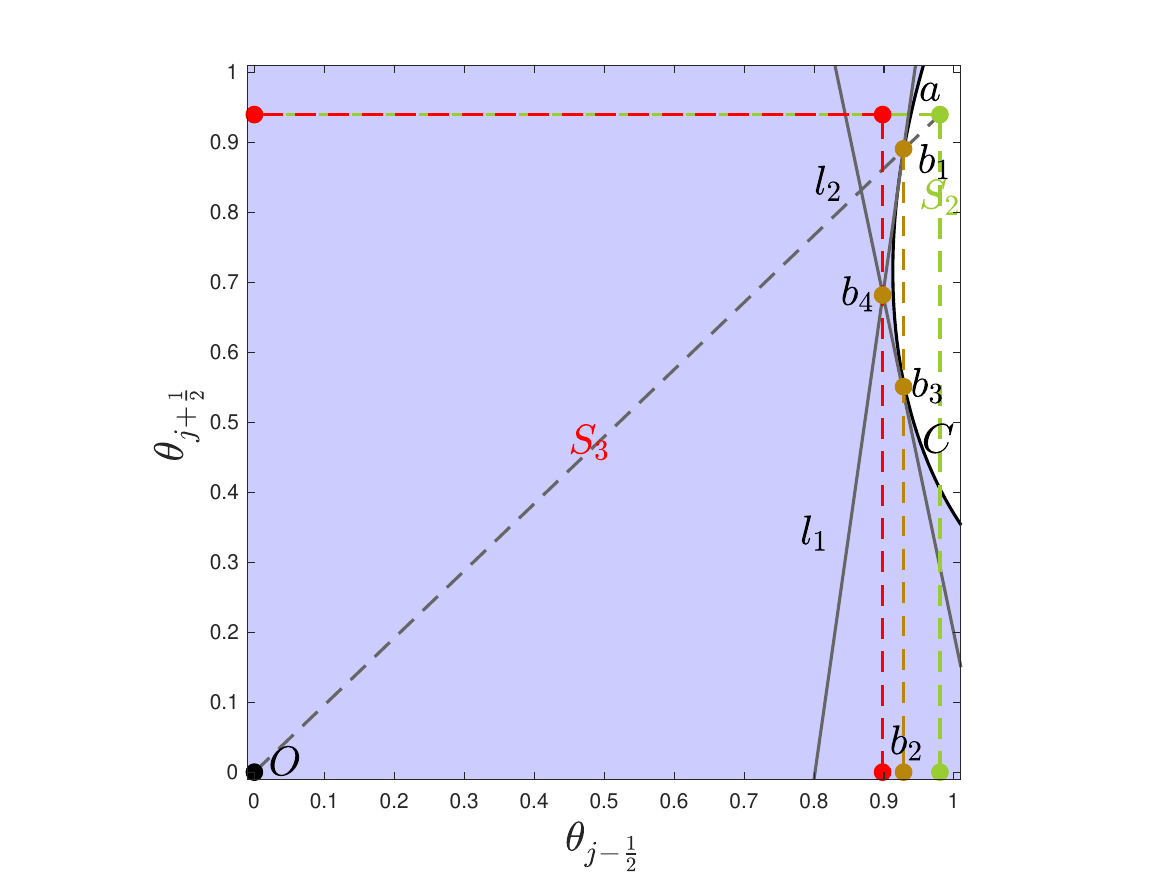}
    \caption{Limit $v_j^{n+1} \leq (v_{\max})_j^{n+1}$, case(c)(ii).}
    \label{fig:sub4}
\end{subfigure}
\hfill
\begin{subfigure}[b]{0.32\textwidth}
    \centering
    \includegraphics[width =  \textwidth,trim=70 0 95 30,clip]{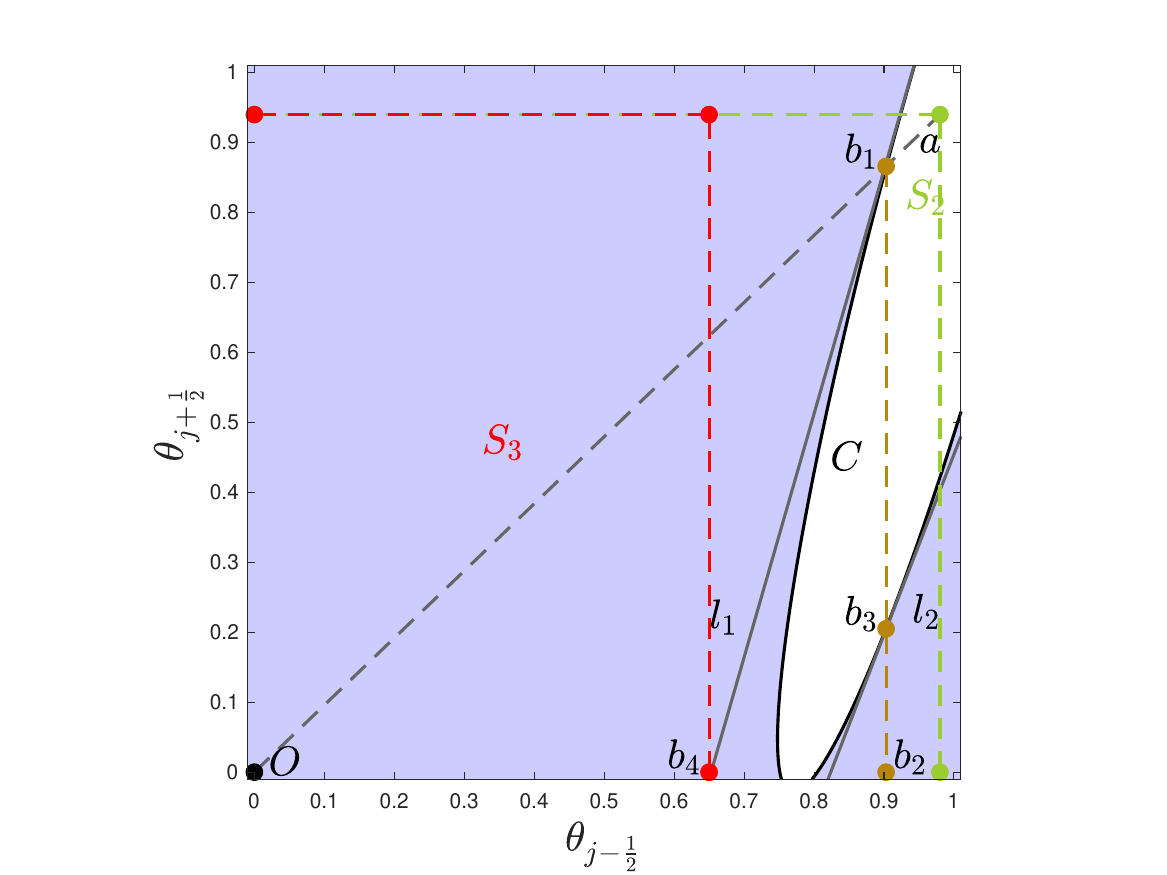}
    \caption{Limit $v_j^{n+1} \leq (v_{\max})_j^{n+1}$, case(c)(ii).}
    \label{fig:sub5}
\end{subfigure}

\caption{Limit $(v_{\min})_j^{n+1} \leq v_j^{n+1} \leq (v_{\max})_j^{n+1}$.}
\label{fig:total}
\end{figure}

\begin{remark}
	The case $m(l_1) = m(Oa)$ implies that the point $a$ belongs to $\Omega_{{\max};j}^{n + 1}$.
\end{remark}

\section{Numerical Examples}\label{0sec5}

To demonstrate the high-order accuracy, bound-preserving property, and overall effectiveness of the proposed schemes, we conduct several rigorous numerical tests on Temple-class systems and ARZ traffic models on networks. We employ a fifth-order finite difference WENO scheme for spatial discretization \cite{shu1998essentially,jiang1996efficient} and a third-order Runge-Kutta scheme for time integration \cite{gottlieb2001strong}, combined with the proposed BP flux limiters. Unless stated otherwise, the Courant-Friedrichs-Lewy (CFL) number is set to $0.6$, and the time step should also satisfy the CFL condition:
\begin{equation*}
	\Delta \tau_n \leq \frac{{\rm CFL} \Delta \xi}{\max_{j = 1}^{N_\xi} \left\{ \vert \lambda_{j, *}^n \vert + c_j^n\right\}},
\end{equation*}
where $\vert \lambda_{j, *}^n \vert$ denotes the maximum eigenvalue at the time level $\tau_n$ for the cell $I_j$ in the Temple-class system.
For all test, we set $J(\xi, 0) \equiv 1$, which means that the physical domain $\mathcal{D}_p$ is equal to the computational domain $\mathcal{D}_c$ at time level $\tau_0$. All numerical results for Temple-class system are presented based on the physical location at the final time.

\subsection{Numerical experiments for Temple-class systems}\label{sec:single}

\begin{exa} \label{exam1}
	Consider the following smooth initial condition for ARZ model:
	\begin{equation*}
		\phi(\xi, 0) \equiv 0.5, \quad v(\xi, 0) = 0.1 + 0.4 \cos{2 \pi \xi},
	\end{equation*}
	with parameters $v_{\rm ref} = 0.4$ and $\gamma = 0$. The computational domain is $\xi \in [0, 1]$ with periodic boundary conditions. We compute the numerical solution up to a final time of $t = 0.1$ using our BP scheme. Since obtaining the exact solution for this initial condition is challenging, we use the numerical solution computed by the BP scheme on a fine mesh ($N_\xi = 2560$) as a reference solution. Our scheme employs a moving mesh, which may result in slight variations in the physical locations for different mesh sizes. In this example, Lagrange interpolation is used to compute the numerical errors for different $N_\xi$: 
	\begin{equation*}
		L_1 \; {\rm error} = \sum_{i = 1}^{N_\xi} \left\vert \mathcal{P} \left( \left(x_{i;j}^{\rm Ref}, k_{i;j}^{\rm Ref}\right)_{j = 1}^6, x_i\right) - k_i \right\vert,
	\end{equation*}
	where $k_i$ represents the numerical result of $k$ at the final time physical location $x_i$, and $k_{i;j}^{\rm Ref}, j = 1, \cdots, 6$ are the six values closest to $x_i$ on the reference solution, corresponding to the physical locations $x_{i;j}^{\rm Ref}$. The function 
	\begin{equation*}
		\mathcal{P} \left( \left(x_{i;j}, k_{i;j}\right)_{j = 1}^s, x_i\right)
	\end{equation*}
	 is the Lagrange interpolation polynomial used to estimate the value at $x_i$ based on the $s$ given data $(x_{i;j}, k_{i;j})$. The results, which demonstrate clear convergence orders, are presented in Table \ref{ErrorTable}. Without the BP limiter, the numerical solution exceeds $v_{\max}$ at approximately $t \approx 0.050893$ for $N_\xi = 320$ and $t \approx 0.0020089$ for $N_\xi = 640$.

	\begin{table}[th!]
		\renewcommand\arraystretch{1.2}
		\caption{\sf Example \ref{exam1}, the errors and orders of $k$.}
		\begin{center}
			\begin{tabular}{clclclc}
				
				\toprule[1.5pt]
				
				$N_\xi$ &$L_1$ error & order & $L_2$ error & order & $L_\infty$ error & order \\ 
				
				\midrule[1.5pt]
				
				10  & 8.76e-04 & --   & 1.21e-03 & --   & 1.77e-03 & --  \\ 
				20  & 9.73e-05 & 3.17 & 1.12e-04 & 3.44 & 1.88e-04 & 3.24\\ 
				40  & 7.03e-06 & 3.79 & 9.73e-06 & 3.52 & 1.86e-05 & 3.34\\ 
				80  & 3.62e-07 & 4.28 & 5.95e-07 & 4.03 & 1.90e-06 & 3.29\\ 
				160 & 8.22e-09 & 5.46 & 1.60e-08 & 5.22 & 5.87e-08 & 5.02\\ 
				320 & 1.29e-10 & 5.99 & 2.11e-10 & 6.24 & 7.62e-10 & 6.27\\ 
				640 & 2.08e-12 & 5.96 & 2.81e-12 & 6.23 & 2.05e-11 & 5.21\\ 
				
				\bottomrule[1.5pt]
				
			\end{tabular}
		\end{center}
		\label{ErrorTable}
	\end{table}
	
\end{exa}

\begin{exa} \label{exam2}
	In this example, we investigate the conservativity of our schemes by considering the following Riemann problem based on the ARZ model:
	\begin{equation}
	\label{RiemannConservativity}
	    (\phi, v)(\xi, 0) = \left\{
	    \begin{aligned}
	        &(0.6, 0.6) \quad {\rm if} \; \xi \in [-0.2, 0.2], \\
	        &(0.5, 0.6) \quad {\rm otherwise}, 
	    \end{aligned}
	    \right.
	\end{equation}
	with $v_{\rm ref} = 1$, $\gamma = 2$, and $T = 1$. The computational domain is $\xi \in [-2, 2]$ with periodic boundary conditions. We define the error in conservativity as follows:
	\begin{equation*}
	\label{consereq}
	    {\rm err}_{J\phi} = \left\vert \sum_{i = 1}^{N_\xi} (J\phi)_i^{N_T} - \sum_{i = 1}^{N_\xi} (J\phi)_i^{0} \right\vert.
	\end{equation*}
	Table \ref{TRC} presents the values of ${\rm err}_{J\phi}$ for Example \ref{exam1} and the aforementioned Riemann problem, showing that our schemes approach conservative as $\Delta t \rightarrow 0$. Figure \ref{FRC} illustrates the results of \eqref{RiemannConservativity} with $N_\xi = 500$, demonstrating that our BP schemes outperform the NonBP scheme by preventing non-physical velocity overshoot.

	\begin{table}[h!]
		\renewcommand\arraystretch{1.2}
		\centering
		\caption{\sf The conservativity of $J\phi$.}
		\begin{tabular}{c cccc cc}
			
			\toprule[1.5pt]
			
			\multirow{2}{*}{$N_\xi$} & \multicolumn{4}{c}{Example \ref{exam1}} & \multicolumn{2}{c}{Example \ref{exam2}} \\
			
			\cmidrule(lr){2-5} \cmidrule(lr){6-7}
			
			& 20 & 80 & 320 & 1280 & 500 & 1000 \\
			
			\midrule[1.5pt]
			
			${\rm err}_{J\phi}$ & 2.22e-16 & 3.89e-16 & 1.11e-15 & 3.82e-14 & 1.85e-14 & 5.78e-14 \\ 
			
			\bottomrule[1.5pt]
			
		\end{tabular}
		\label{TRC}
	\end{table}
	 
	\begin{figure}[hbtp]
	    \begin{center}
	        \mbox{
	        {\includegraphics[width = 0.33\textwidth,trim=25 15 35 10,clip]{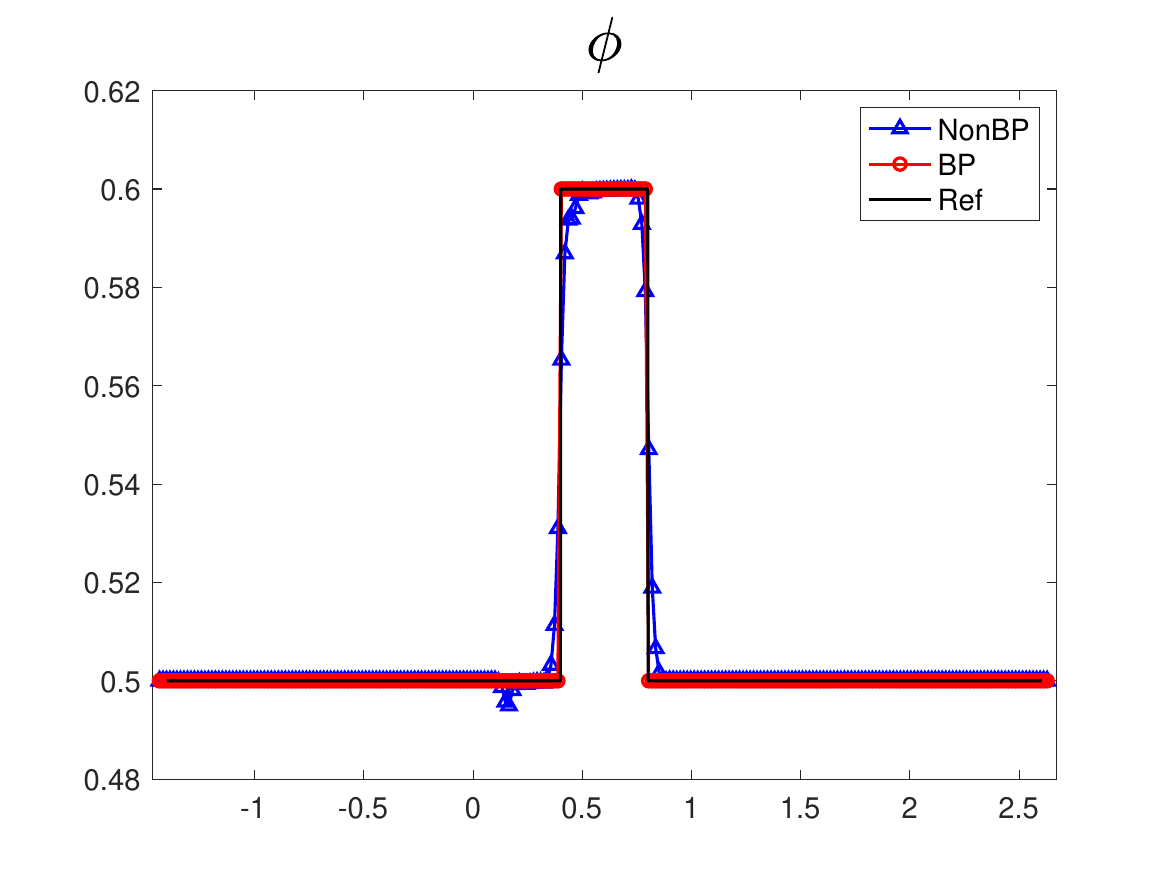}}
	        {\includegraphics[width = 0.33\textwidth,trim=25 15 35 10,clip]{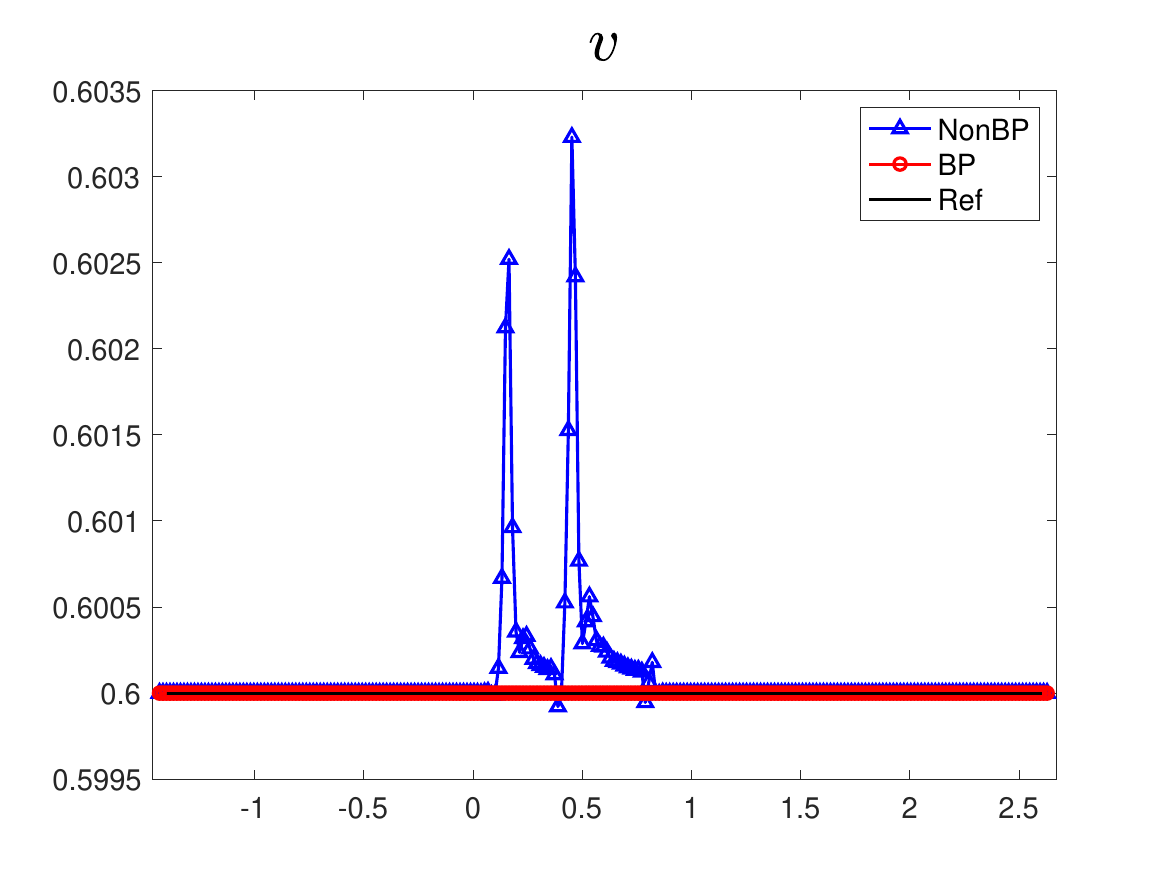}}
	        {\includegraphics[width = 0.33\textwidth,trim=25 15 35 10,clip]{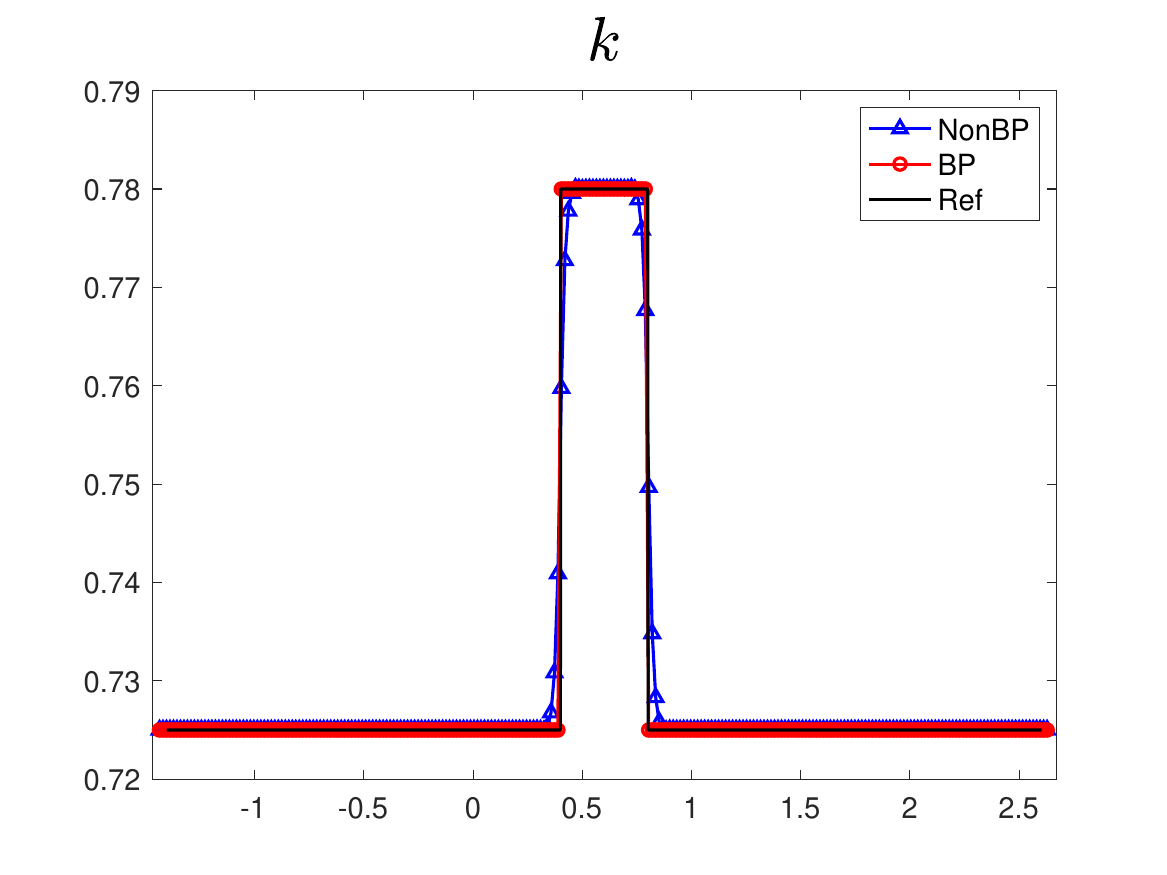}}}\\
	        \caption{\sf Example \ref{exam2}, the numerical resluts obtained by using the BP scheme and NonBP scheme, $t = 1$, and $\Delta \xi = 0.008$.}
	        \label{FRC}
	    \end{center}
	\end{figure}
	
\end{exa}

\begin{exa} \label{exam3}
	To demonstrate the high resolution of our BP scheme, we design the following initial values for the sedimentation model:
	\begin{equation*}
	\begin{aligned}
	    & \phi (\xi, 0) = 0.4, \\
	    & v (\xi, 0) = \left\{
	    \begin{aligned}
	        & 0.1, \quad {\rm if} \; \xi \in [0, 0.5] \cup [3.5, 4], \\
	        & 0.1 + 0.01 * (3.5 - \xi) (\xi - 0.5) \sin{(10 \pi (\xi - 0.5)(3.5 - \xi)}), \quad {\rm otherwise}.   
	    \end{aligned}
	    \right. 
	\end{aligned}
	\end{equation*}
	The computational domain is $\xi \in [0, 4]$ with $\gamma = 2$. Figure \ref{T8} presents the results of the BP scheme ($N_\xi = 500$) alongside the first-order BP scheme results for both $N_\xi = 500$ and $N_\xi = 40000$, with the latter serving as a reference solution.  It is evident that the low-order BP method significantly dampens the magnitude of the transmitted wave, while the high-order BP method resolves the waves with excellent accuracy.

	\begin{figure}[hbtp]
	    \begin{center}
	        \mbox{
	        {\includegraphics[width = 0.33\textwidth,trim=25 15 35 10,clip]{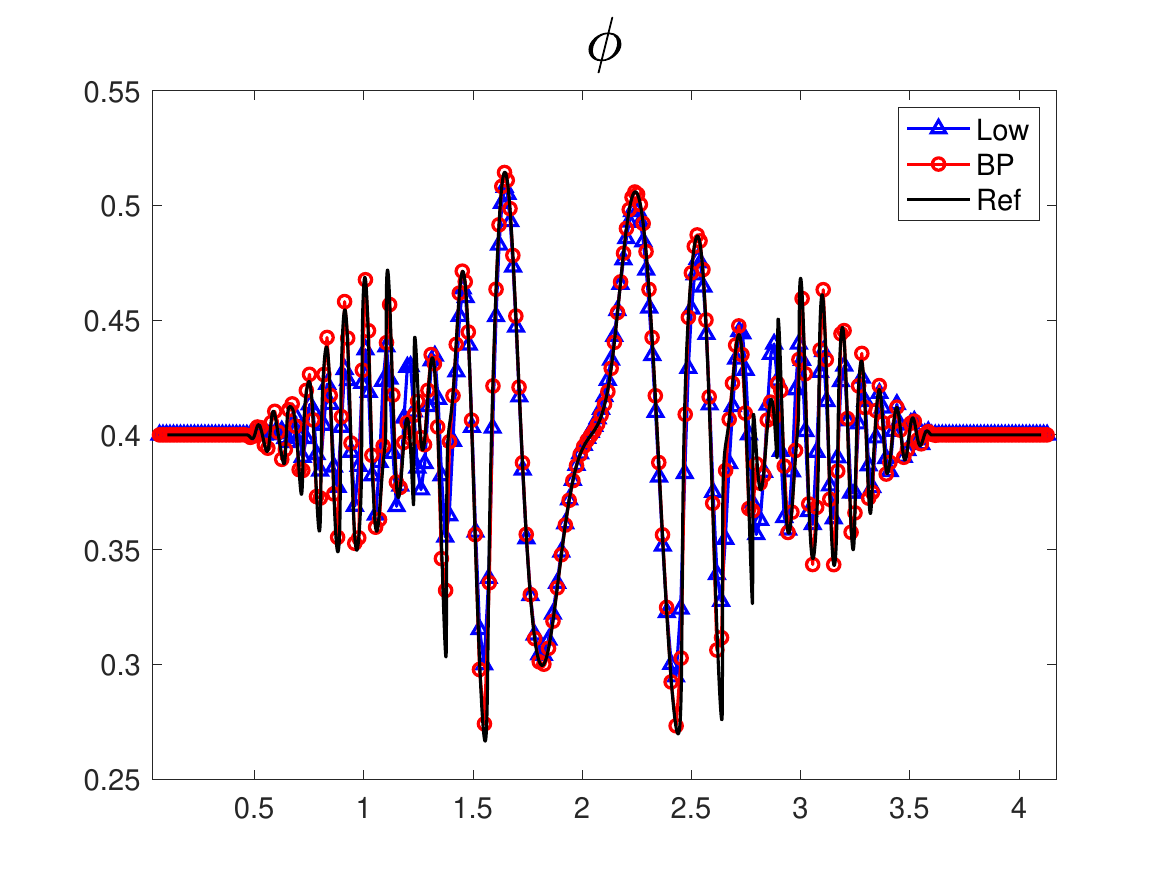}}
	        {\includegraphics[width = 0.33\textwidth,trim=25 15 35 10,clip]{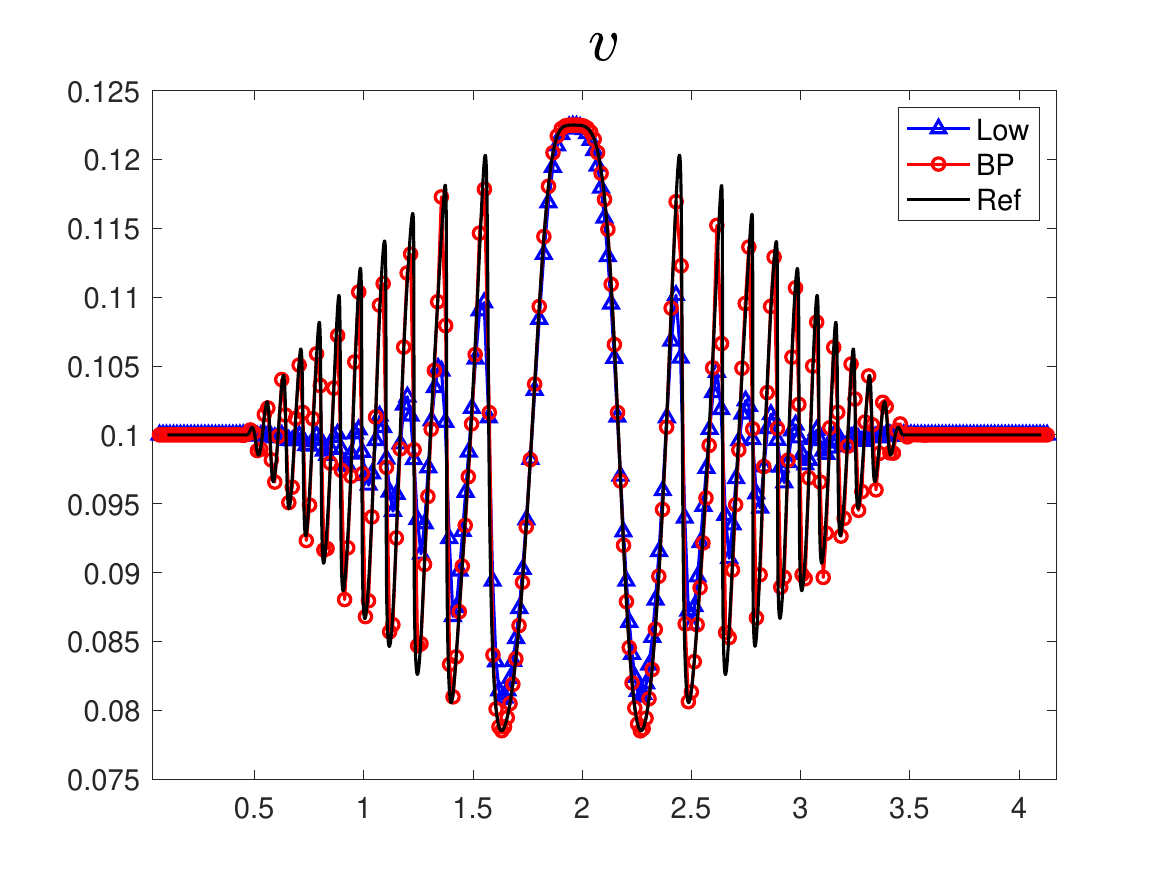}}
	        {\includegraphics[width = 0.33\textwidth,trim=25 15 35 10,clip]{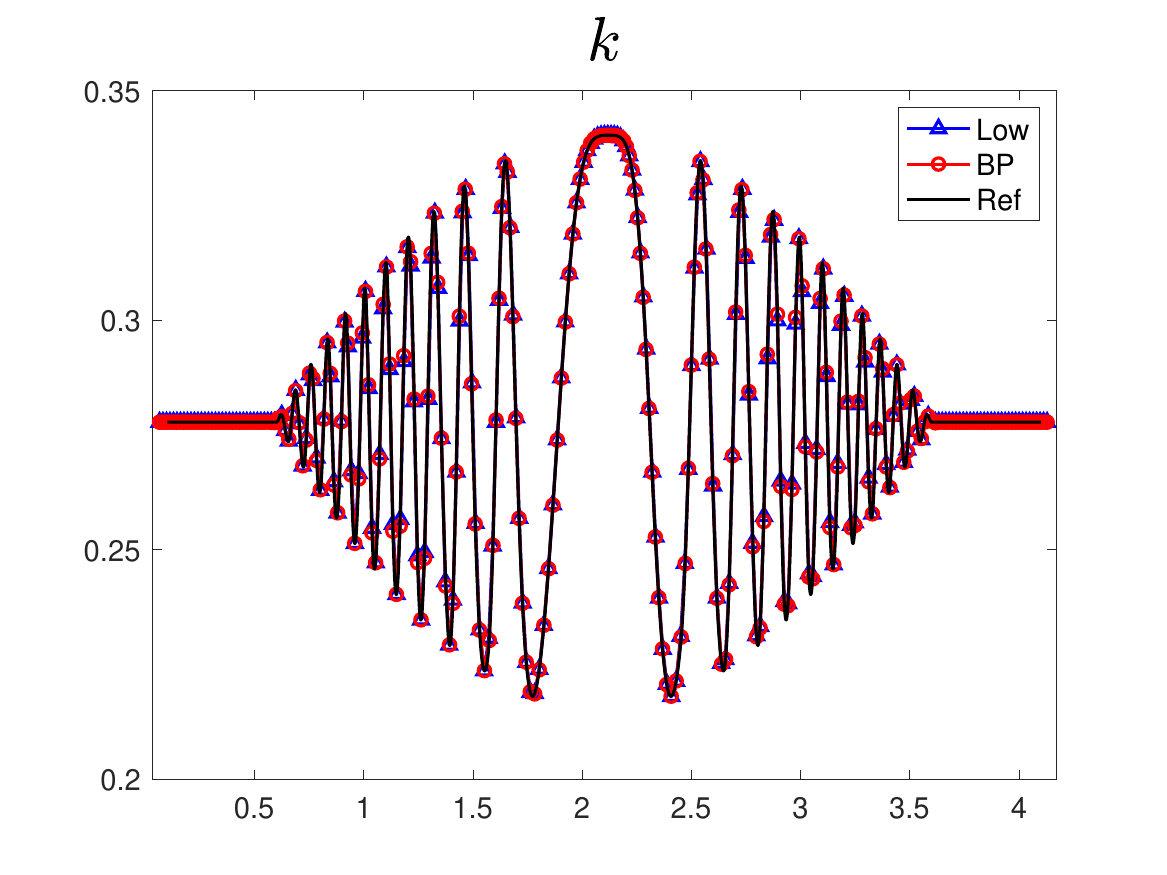}}}\\
	        \caption{\sf Example \ref{exam3}, Low: low-order BP scheme with $N_\xi = 500$; BP: high-order BP scheme with $N_\xi = 500$; Ref: low-order BP scheme with $N_\xi = 40000$.}
	        \label{T8}
	    \end{center}
	\end{figure}

\end{exa}

\begin{exa} \label{exam4}
	\begin{table}[h!]
		\renewcommand\arraystretch{1.2}
		\centering
		\caption{\sf Initial conditions of Riemann problem tests.}
		\begin{tabular}{c cc cc c}
			
			\toprule[1.5pt]
			
			\multirow{2}{*}{Test} & \multicolumn{2}{c}{Riemann problem left state} & \multicolumn{2}{c}{Riemann problem right state} &  \multirow{2}{*}{$\gamma$ values}\\
			
			\cmidrule(lr){2-3} \cmidrule(lr){4-5}
			
			& $\phi_L$ & $v_L$ & $\phi_R$ & $v_R$ & \\
			
			\midrule[1.5pt]
			
			T1 & 0.8  & 0.4    & 0.1  & 0.4    & 2 \\
			T2 & 0.5  & 0.1    & 1e-8 & 0.4    & 1 \\
			T3 & 0.8  & 0.4    & 1e-10 & 0.4    & 0 \\
			T4 & 0.55 & 0.0405 & 0.1  & 0.0405 & 2 \\ 
			T5 & 0.8  & 0.024  & 0.1  & 0.243  & 2 \\ 
			
			\bottomrule[1.5pt]
			
		\end{tabular}
		\label{Riemann}
	\end{table}
	Firstly, we consider Riemann problems for the ARZ model with the following initial conditions:

	\begin{equation*}
	    (\phi, v) (\xi, 0) \equiv \left\{
	    \begin{aligned}
	        &(\phi_L, v_L) \quad &&{\rm if} \; \xi < 0,\\
	        &(\phi_R, v_R) \quad &&{\rm otherwise}.
	    \end{aligned}
	    \right.
	\end{equation*}
	Two different initial conditions (Test T1 and T2) for the Riemann problem are listed in Table \ref{Riemann}, and the final solutions are plotted in Figure \ref{T4}. We observe that the BP scheme avoids velocity overshoot and preserves the wave structure more effectively than the NonBP scheme. Notably, when $\phi$ is near vacuum (Test T2), the NonBP scheme produces a negative $\phi$ and breaks down after a few time steps.
	\begin{figure}[hbtp]
	    \begin{center}
	        \mbox{
	        {\includegraphics[width = 0.33\textwidth,trim=25 15 35 10,clip]{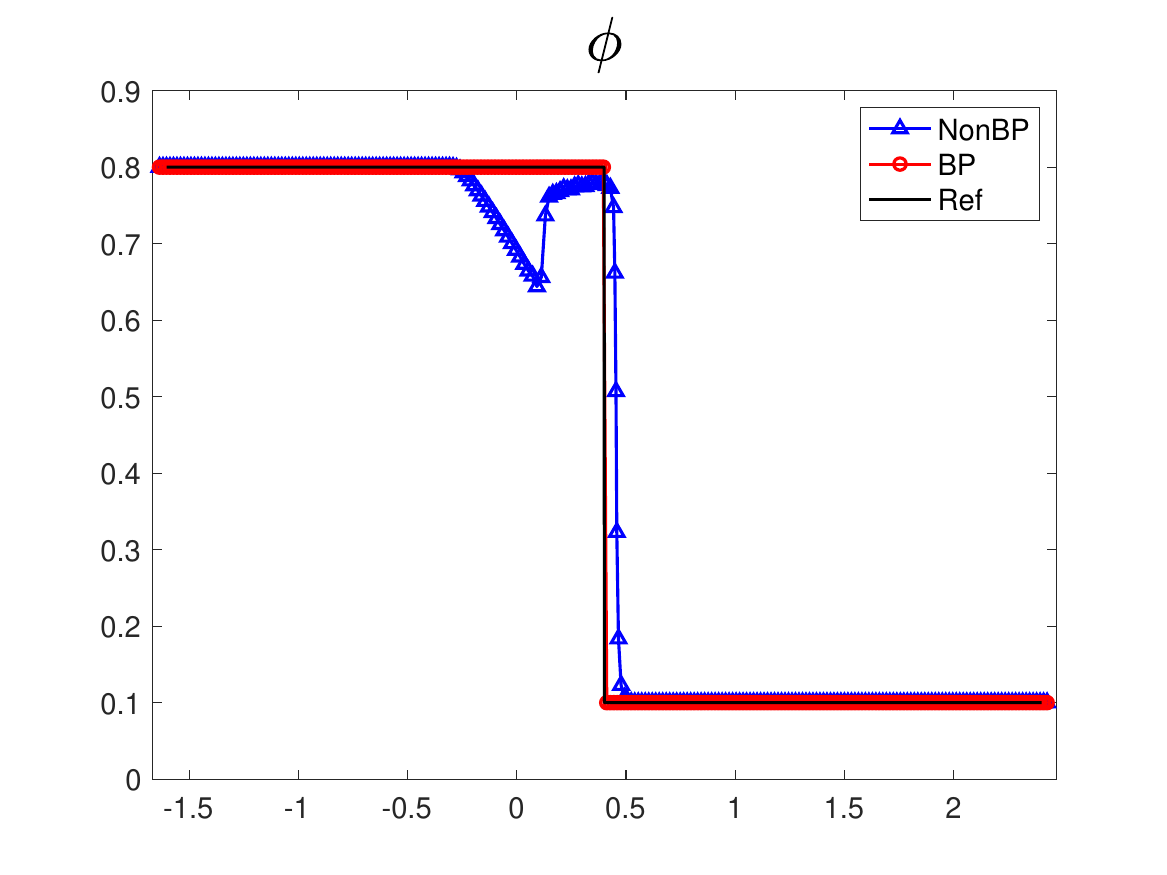}}
	        {\includegraphics[width = 0.33\textwidth,trim=25 15 35 10,clip]{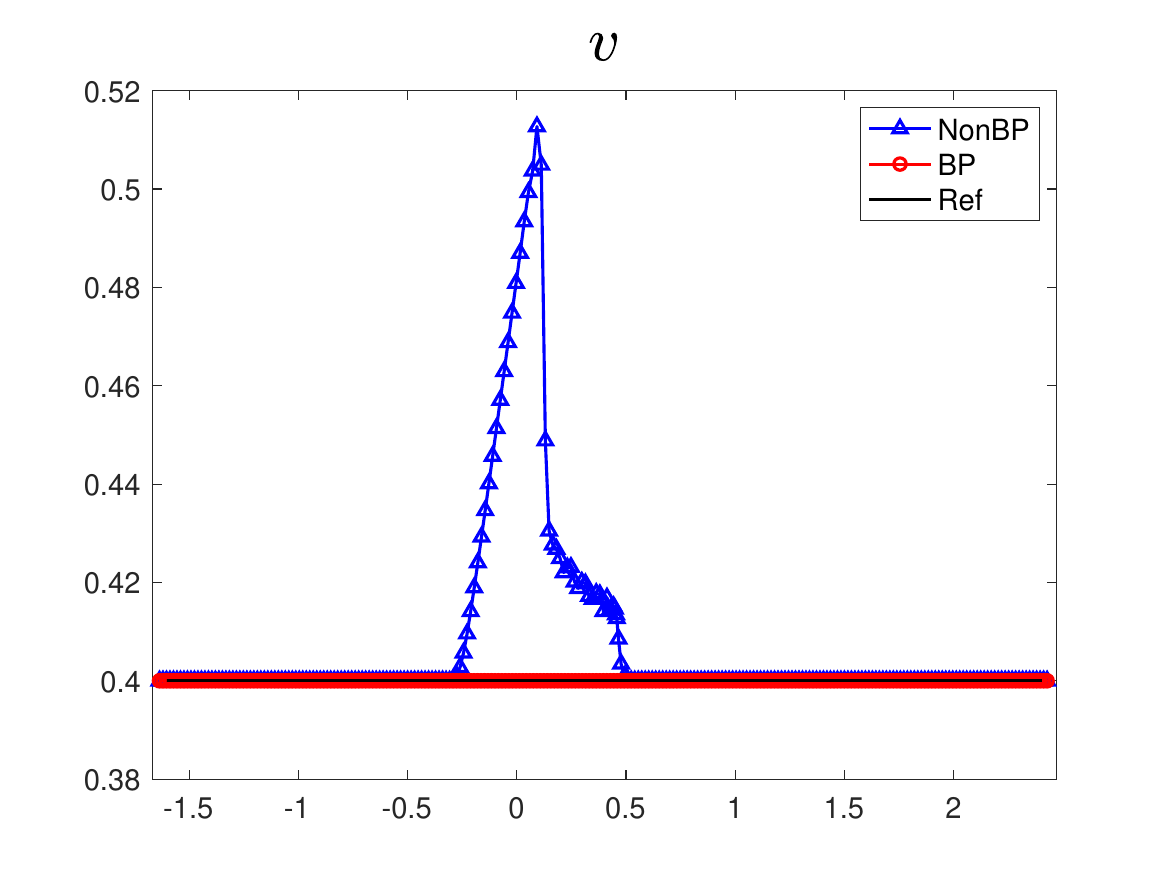}}
	        {\includegraphics[width = 0.33\textwidth,trim=25 15 35 10,clip]{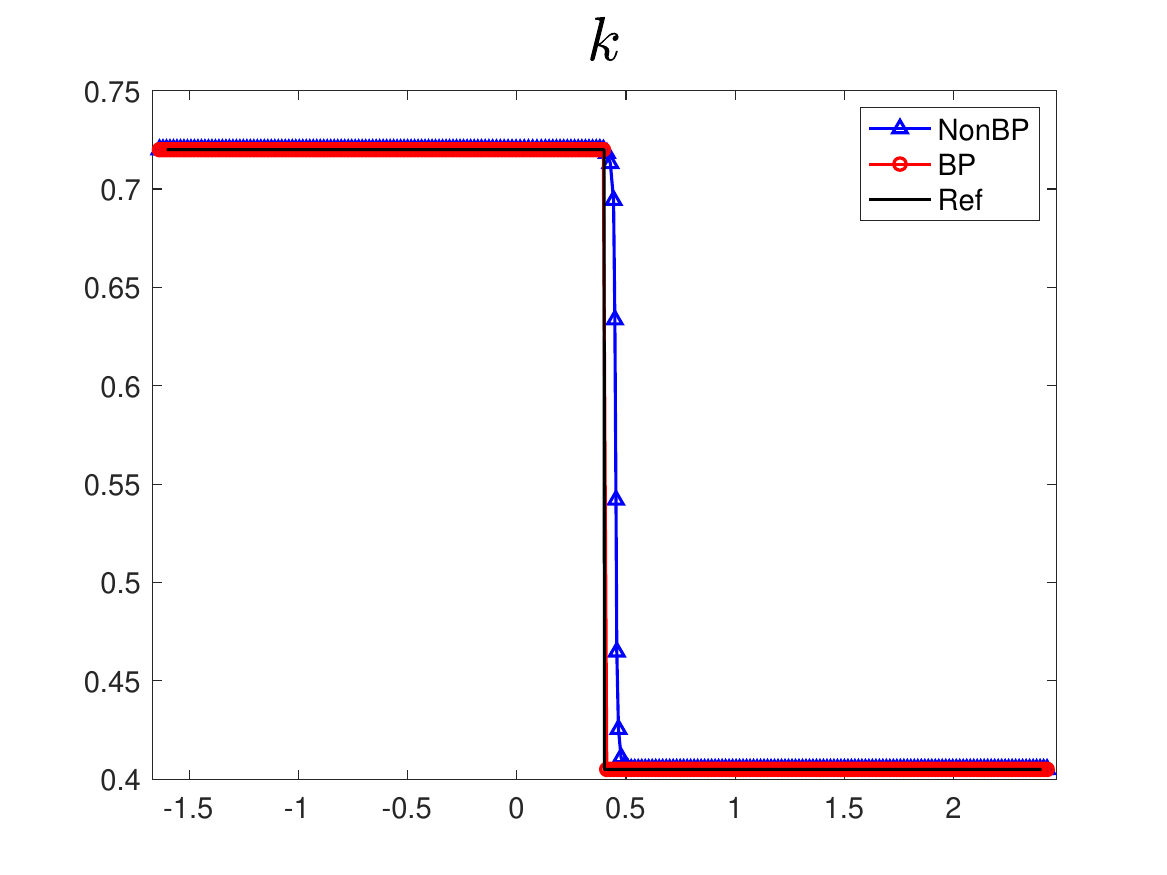}}}\\
	        \mbox{
	        {\includegraphics[width = 0.33\textwidth,trim=25 15 35 10,clip]{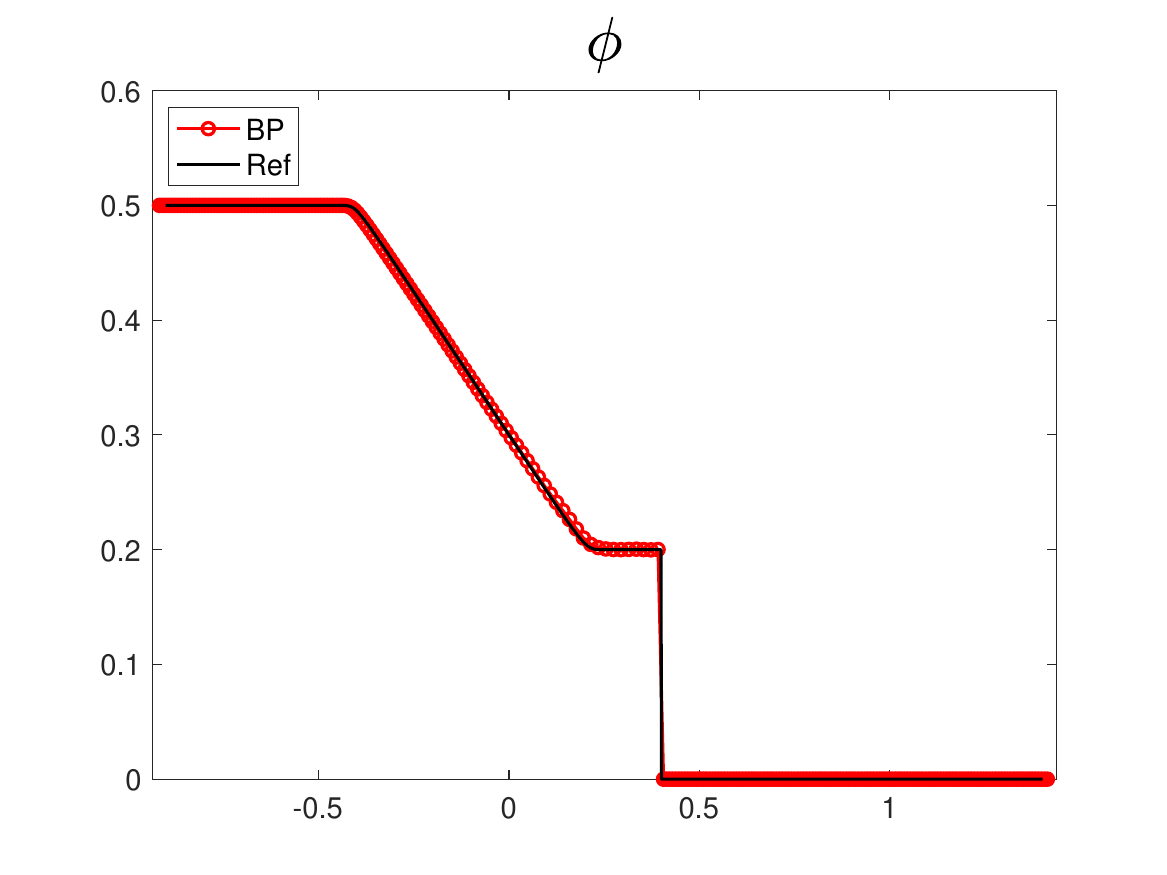}}
	        {\includegraphics[width = 0.33\textwidth,trim=25 15 35 10,clip]{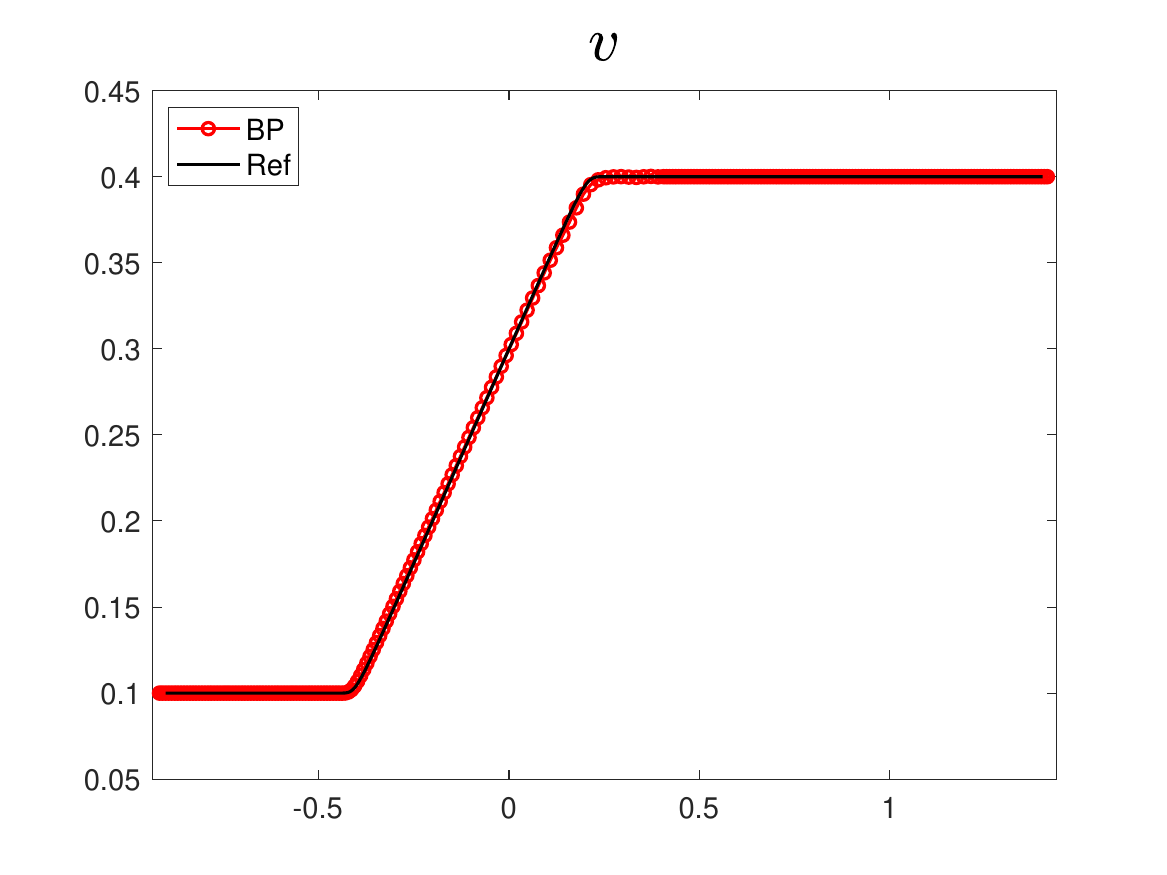}}
	        {\includegraphics[width = 0.33\textwidth,trim=25 15 35 10,clip]{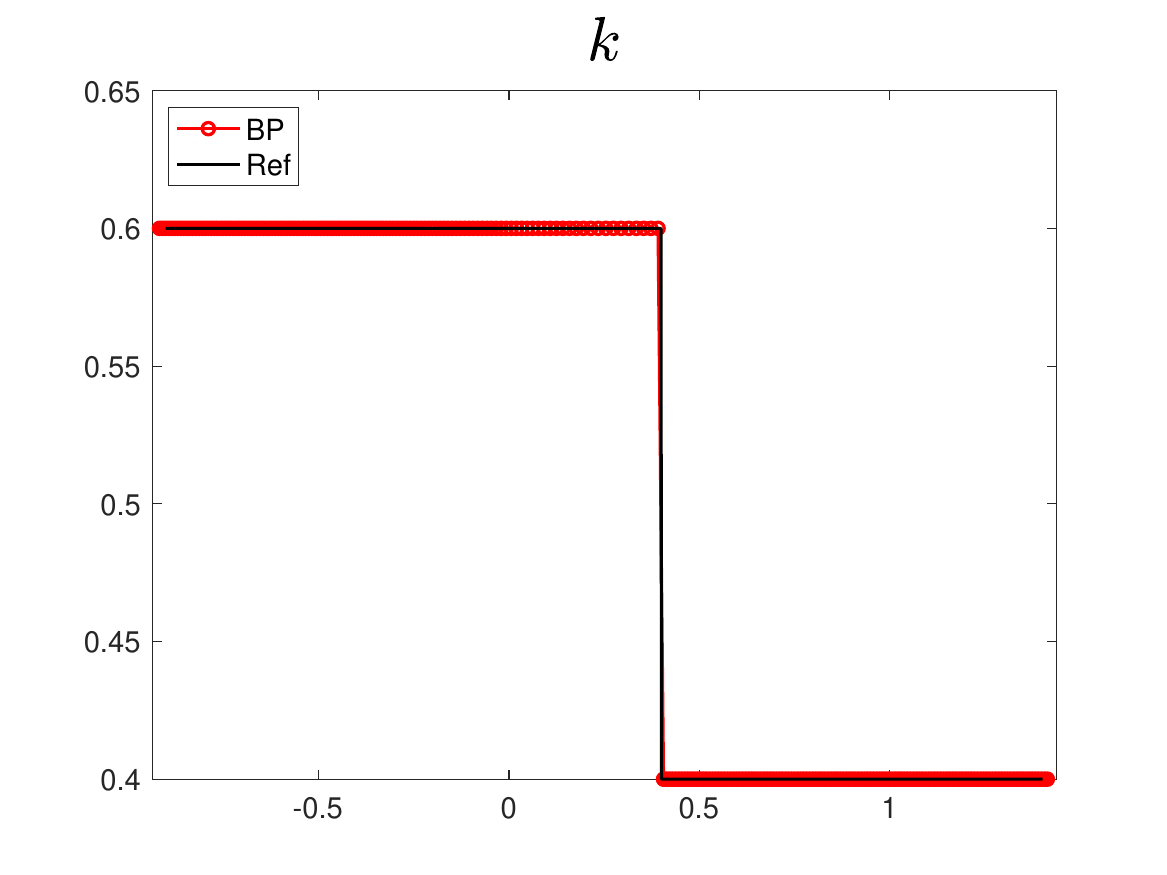}}}\\
	        \caption{\sf Example \ref{exam4}, $t = 1$, and $N_\xi = 500$. Top: Test T1; Bottom: Test T2.}
	        \label{T4}
	    \end{center}
	\end{figure}

	To demonstrate the superiority of this algorithm and the importance of maintaining the speed upper bound, we constructed a specific test, referred to as T3, to compare the BP scheme presented in this paper with the BP-OEDG scheme in \cite{CHEN20251135007}, which does not maintain the speed upper bound. The relevant results are presented in Figure \ref{T09}, which illustrates that the absence of a preserved speed upper bound results in significant instability under conditions of extremely low density.

	\begin{figure}[hbtp]
	    \begin{center}
	        \mbox{
	        {\includegraphics[width = 0.45\textwidth,trim=25 15 35 10,clip]{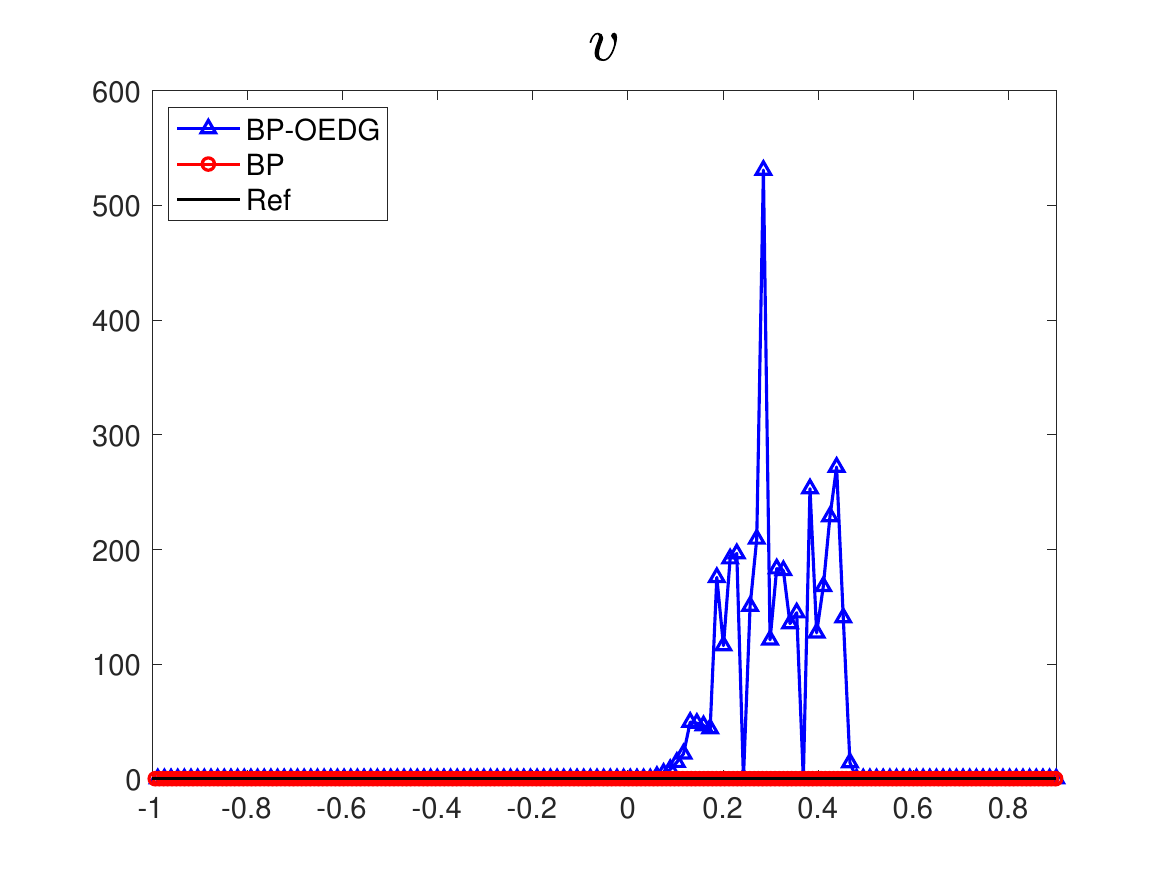}}
	        {\includegraphics[width = 0.45\textwidth,trim=25 15 35 10,clip]{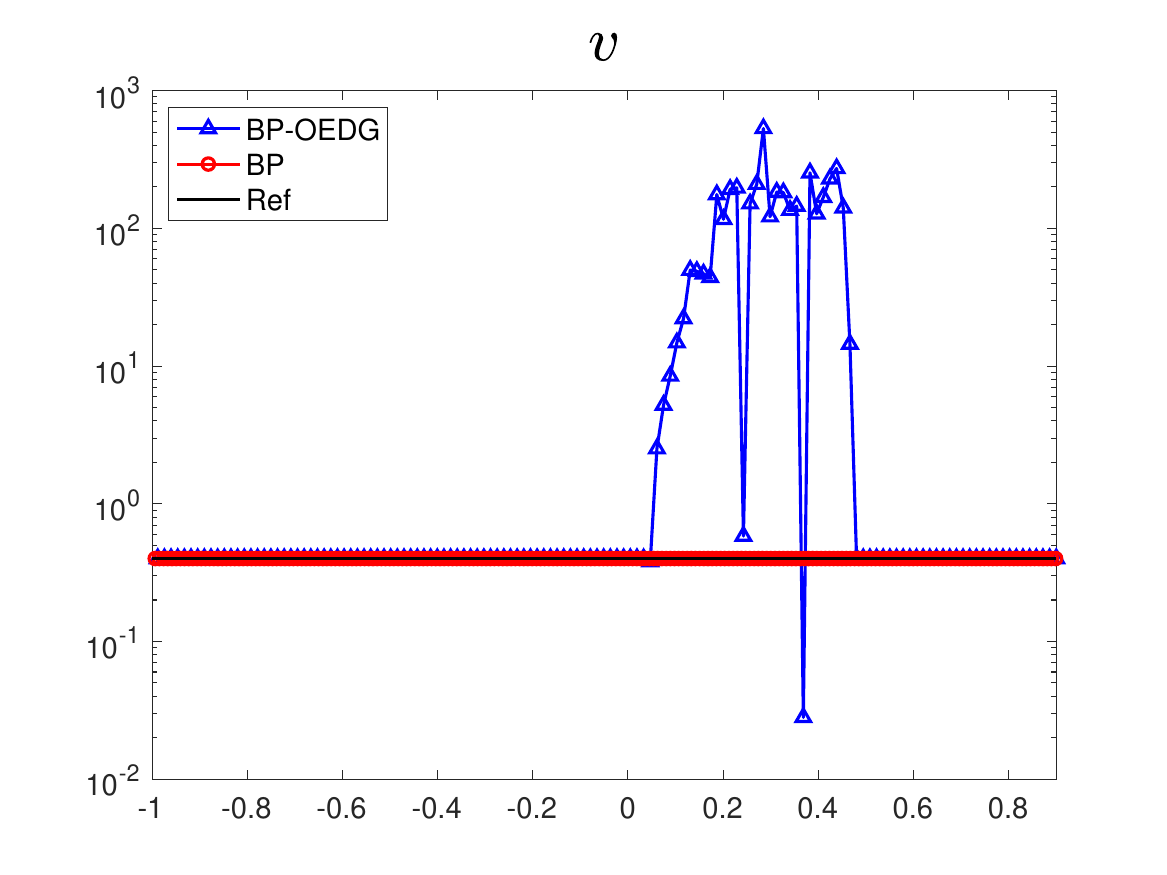}}}\\
	        \caption{\sf Example \ref{exam4}, $t = 0.03$, and $N_\xi = 500$. Left: linear scale; Right: logarithmic scale.}
	        \label{T09}
	    \end{center}
	\end{figure}

	Next, we test a benchmark problem from \cite{betancourt2018random}, which demonstrates the interaction between a contact discontinuity and a 1-shock. Here, both parameters $\gamma$ and $v_{\rm ref}$ are set to $3$. The initial data are given by
	\begin{equation}
	\label{RiemannARZ3}
	    (\phi, v)(\xi, 0) = \left\{
	    \begin{aligned}
	        &(0.4762, 0.092) \quad &&{\rm if} \; \xi < - 0.05, \\
	        &(0.2, 0.092) \quad &&{\rm if} \; \xi \in [- 0.05, 0.05], \\
	        &(0.4, 0.036) \quad &&{\rm otherwise}.
	    \end{aligned}
	    \right.
	\end{equation}
	From \eqref{RiemannARZ3}, a contact discontinuity occurs at $\xi = -0.05$, while a 1-wave connects the data at both ends of $\xi = 0.05$. As the two waves will collide at approximately $t \approx 0.9$, we test this example at final times $t = 0.5$ and $t = 1$ , with the results shown in Figure \ref{Tt5}. Obviously, the numerical solutions obtained from the NonBP scheme progressively deteriorate, while the BP scheme accurately captures the location of discontinuities and effectively suppresses unphysical oscillations.

	\begin{figure}[hbtp]
	    \begin{center}
	        \mbox{
	        {\includegraphics[width = 0.33\textwidth,trim=25 15 35 10,clip]{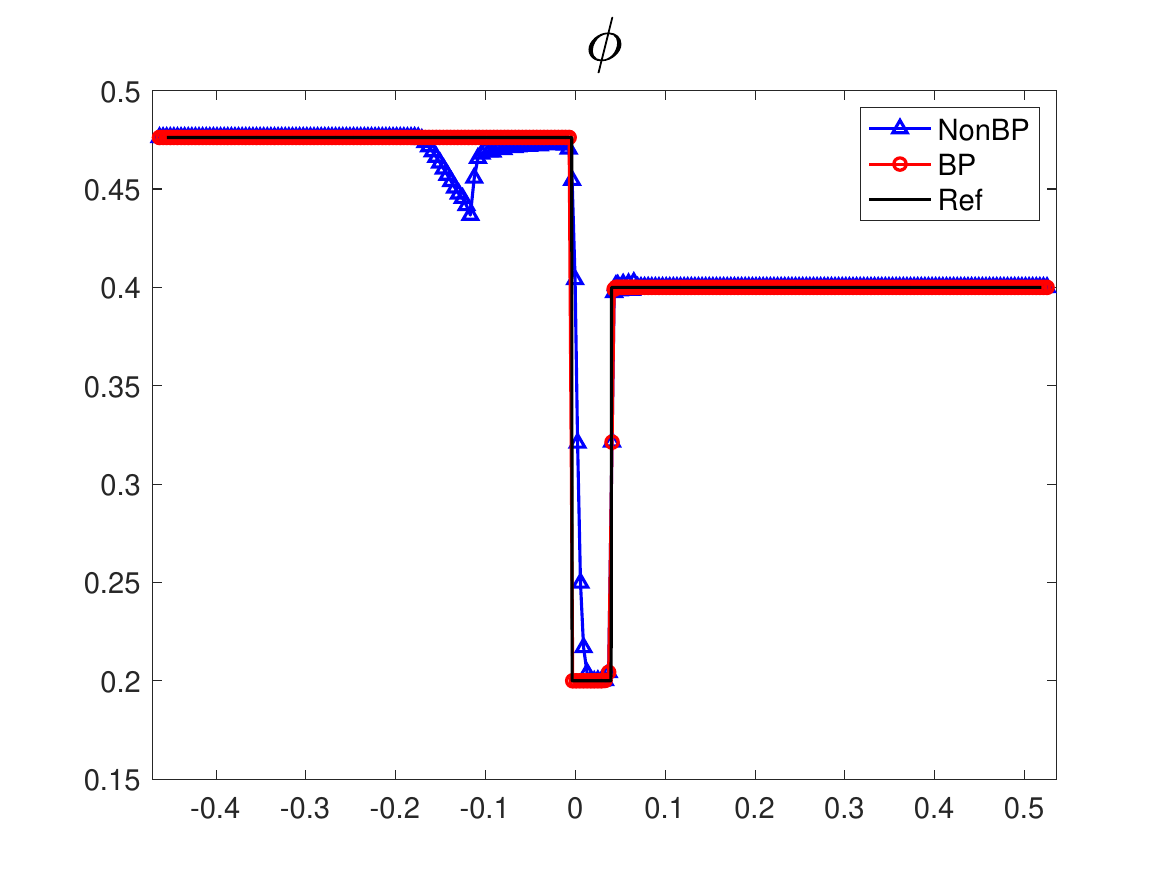}}
	        {\includegraphics[width = 0.33\textwidth,trim=25 15 35 10,clip]{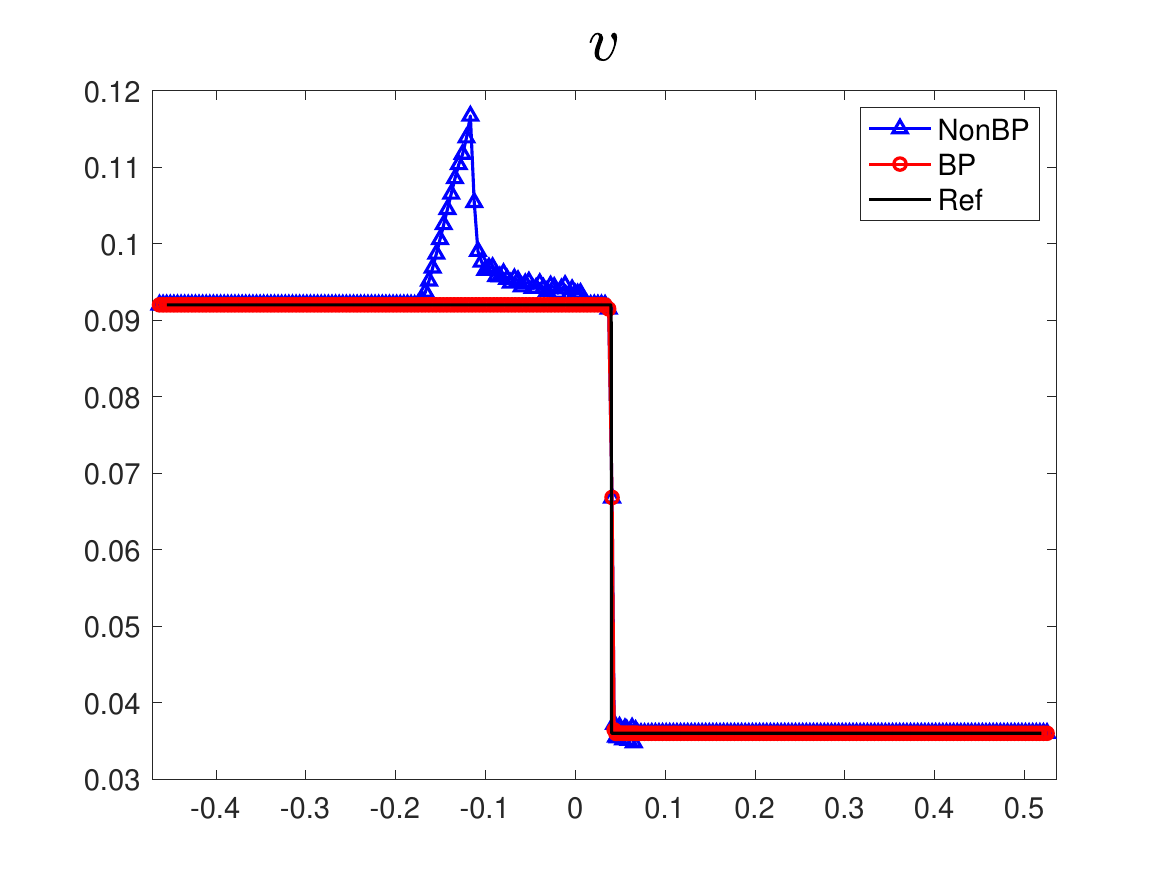}}
	        {\includegraphics[width = 0.33\textwidth,trim=25 15 35 10,clip]{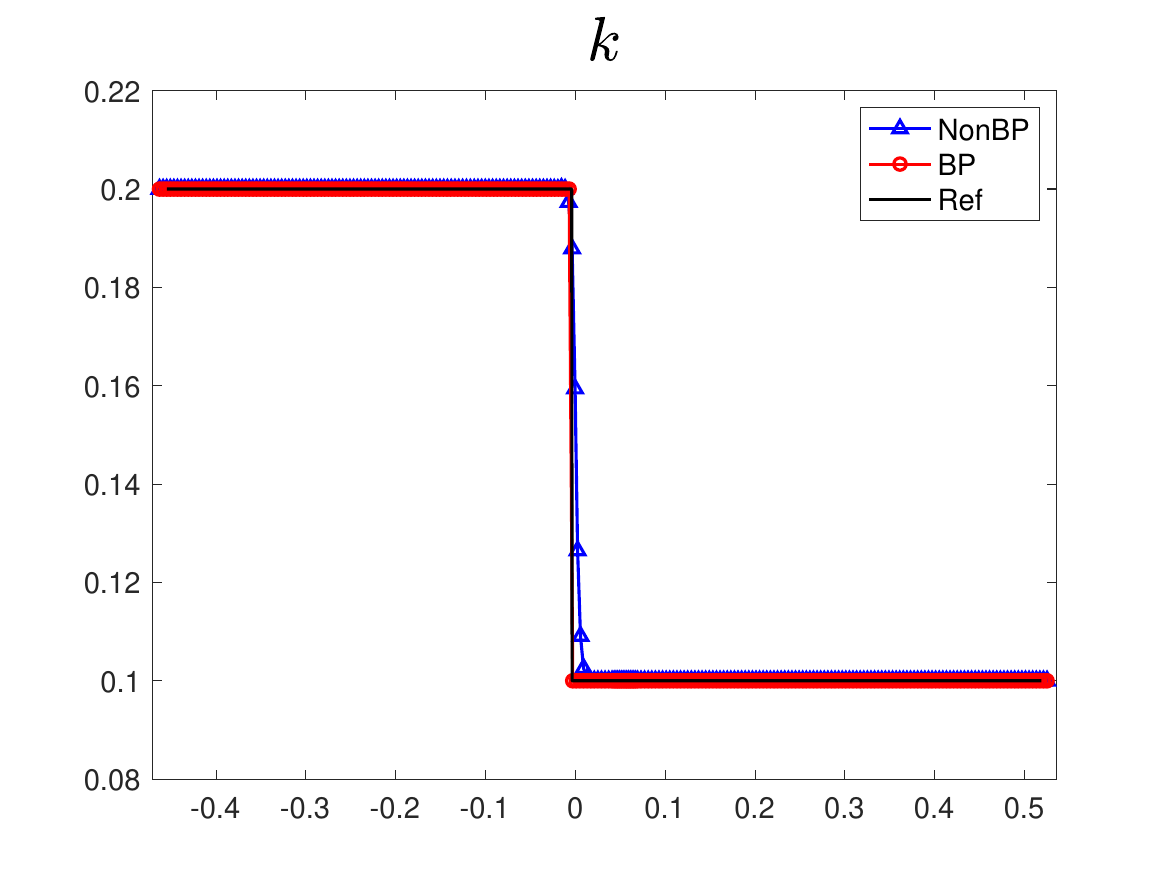}}}\\
	        \mbox{
	        {\includegraphics[width = 0.33\textwidth,trim=25 15 35 10,clip]{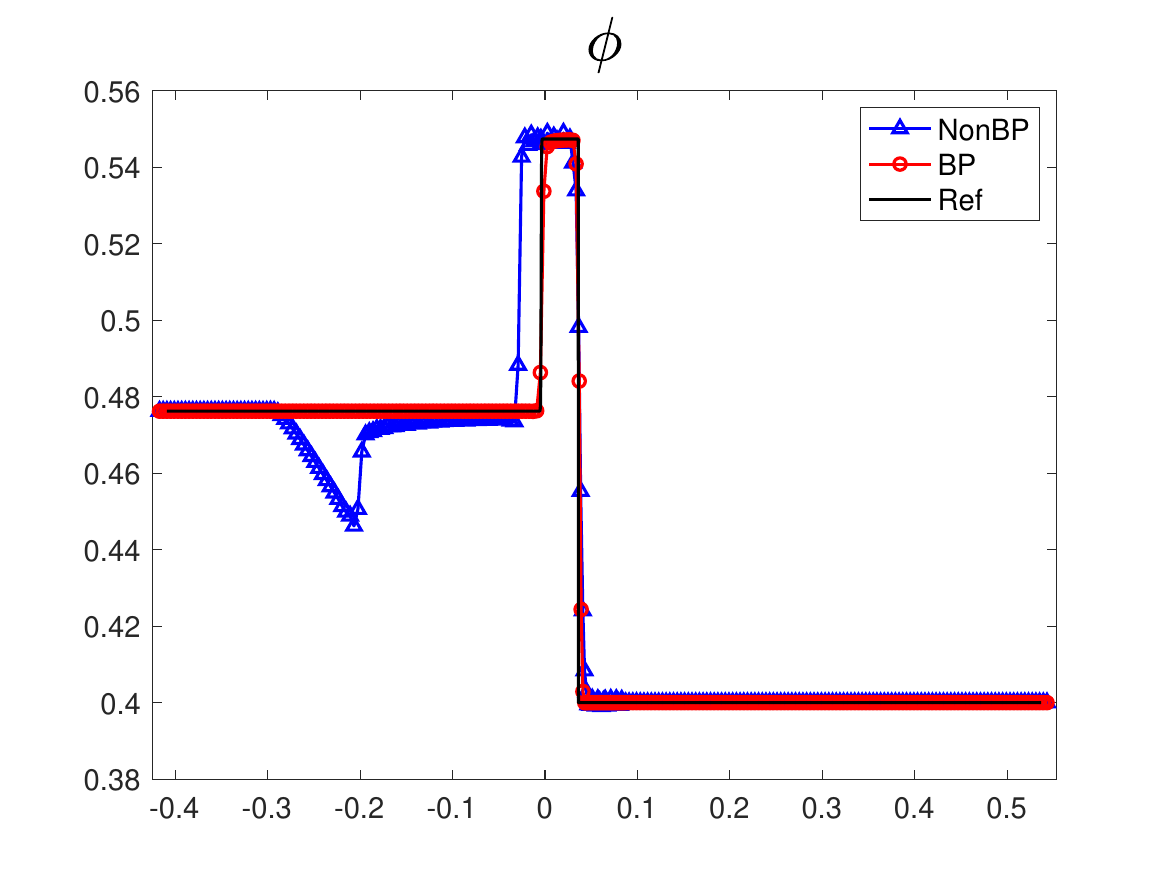}}
	        {\includegraphics[width = 0.33\textwidth,trim=25 15 35 10,clip]{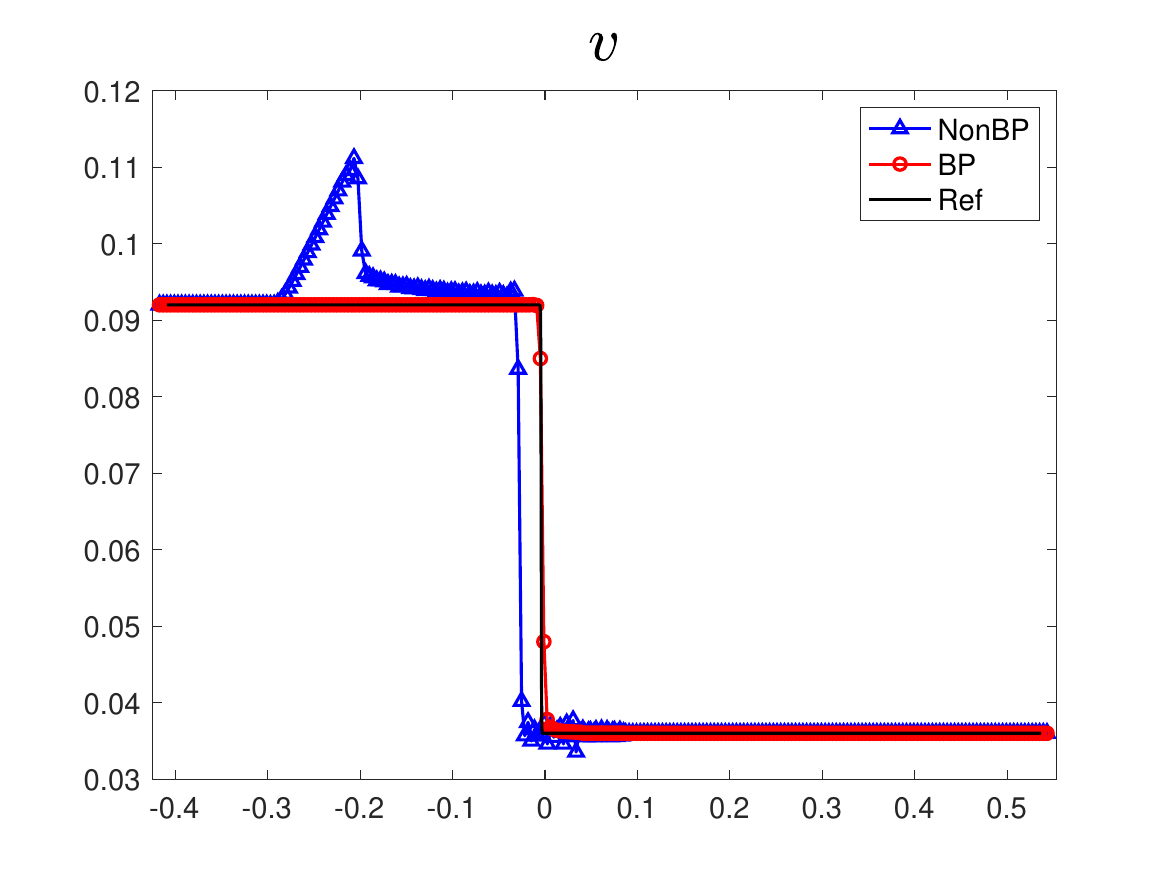}}
	        {\includegraphics[width = 0.33\textwidth,trim=25 15 35 10,clip]{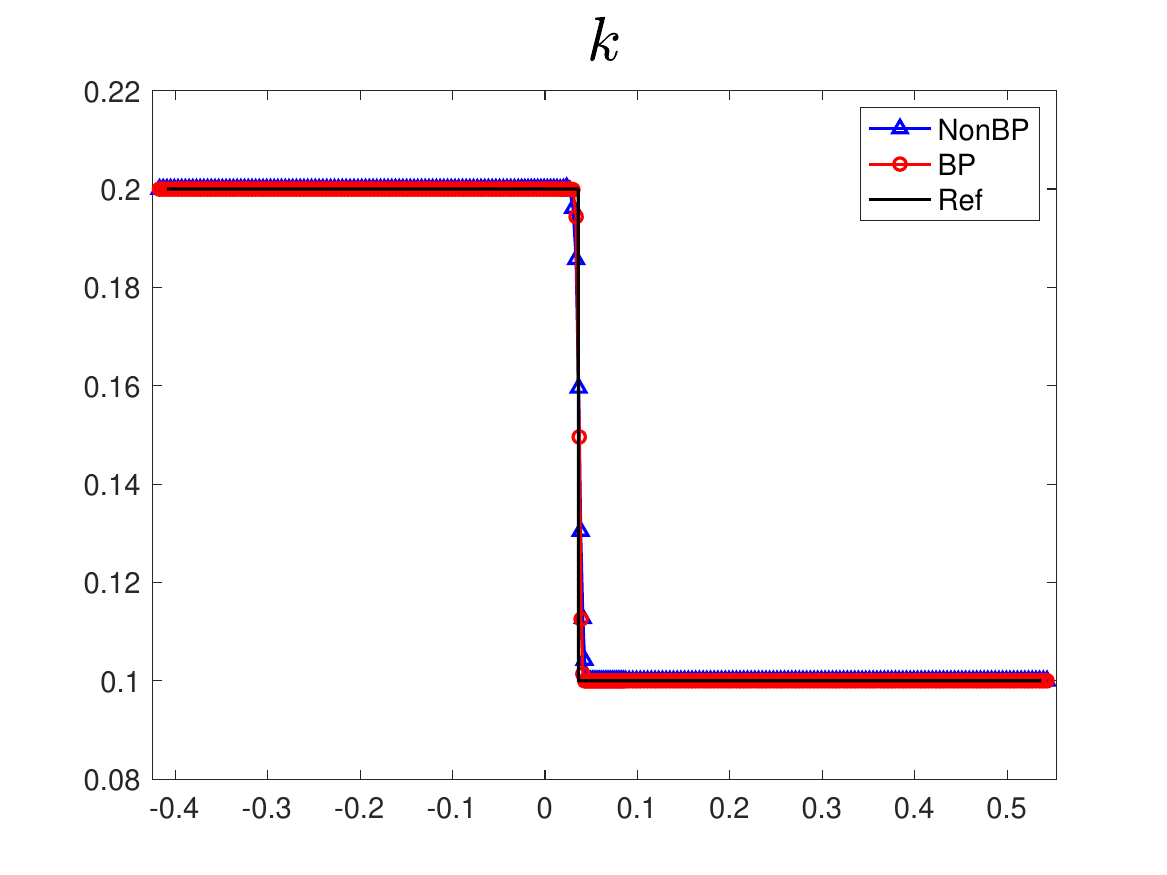}}}\\
	        \caption{\sf Example \ref{exam4}, initial data \eqref{RiemannARZ3} with $N_\xi = 500$. Top: $t = 0.5$; Bottom: $t = 1$.}
	        \label{Tt5}
	    \end{center}
	\end{figure}

\end{exa}

\begin{exa} \label{exam5}

	We now examine two Riemann problems associated with the sedimentation model \cite{betancourt2018random}, with initial conditions T4 and T5 listed in Table \ref{Riemann}. Figure \ref{789} and \ref{086} display the numerical results of T4 and T5 at $t = 1$, respectively. In Figure \ref{789}, the NonBP scheme produces an anomalous solution for the isolated contact discontinuity, while the solution generated by BP scheme is consistent with the reference solution. A mixed 1-wave, composed of both a shock and a rarefaction, as well as a contact discontinuity, is shown in Figure \ref{086}. These results further validate that the BP scheme significantly outperforms the NonBP scheme.

	\begin{figure}[hbtp]
	    \begin{center}
	        \mbox{
	        {\includegraphics[width = 0.33\textwidth,trim=25 15 35 10,clip]{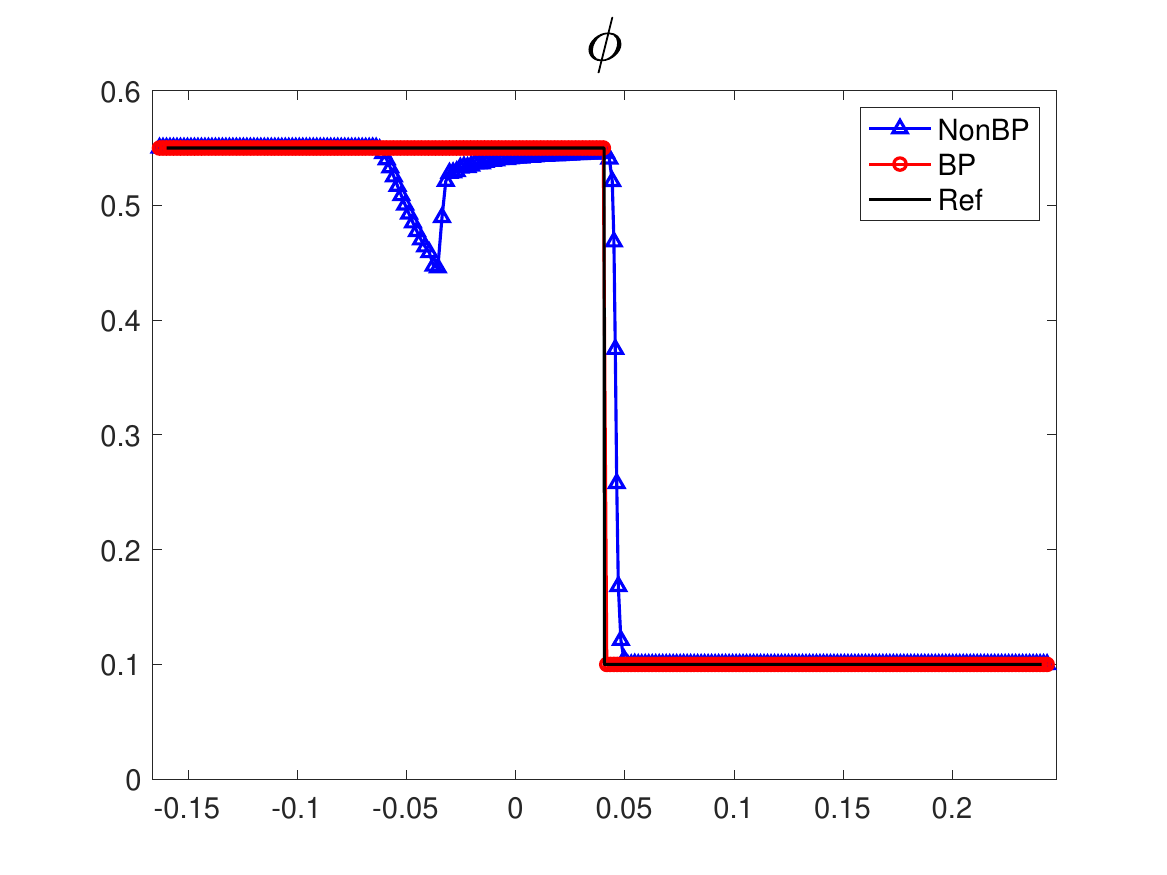}}
	        {\includegraphics[width = 0.33\textwidth,trim=25 15 35 10,clip]{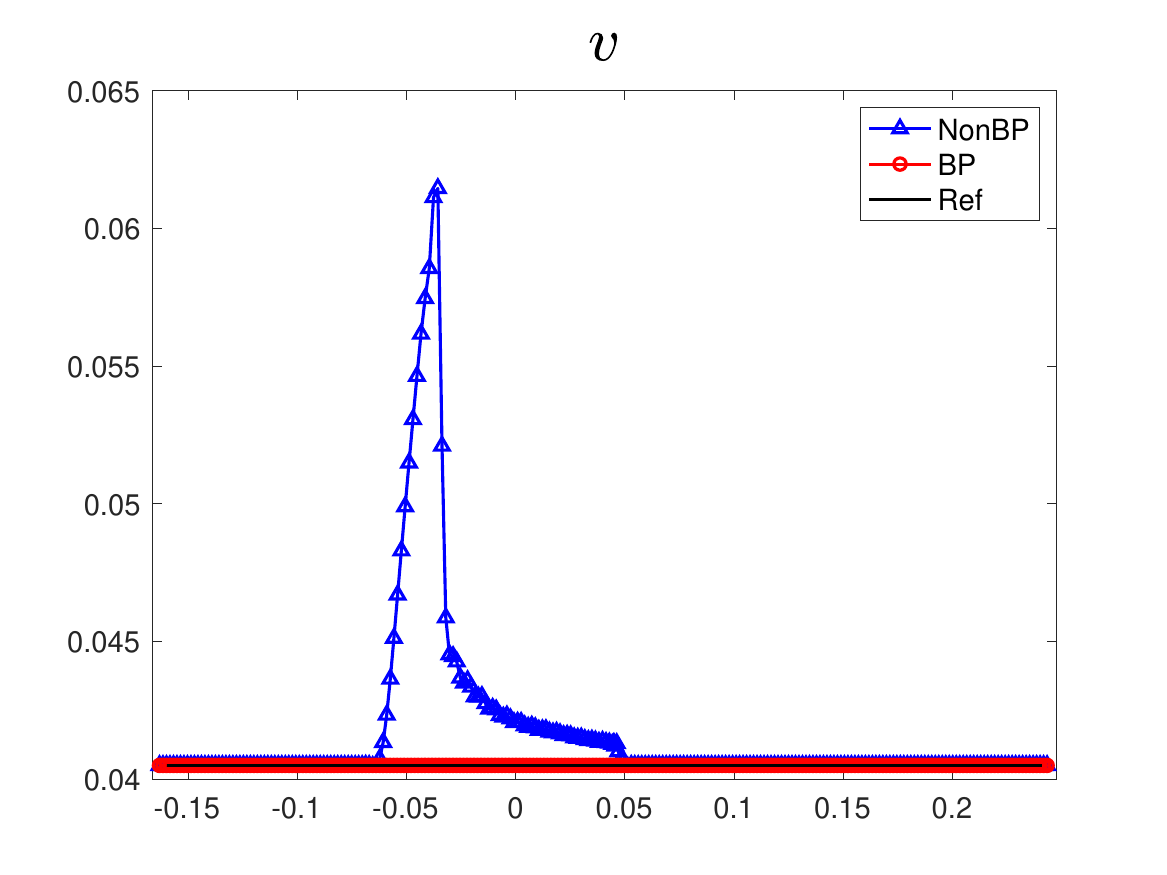}}
	        {\includegraphics[width = 0.33\textwidth,trim=25 15 35 10,clip]{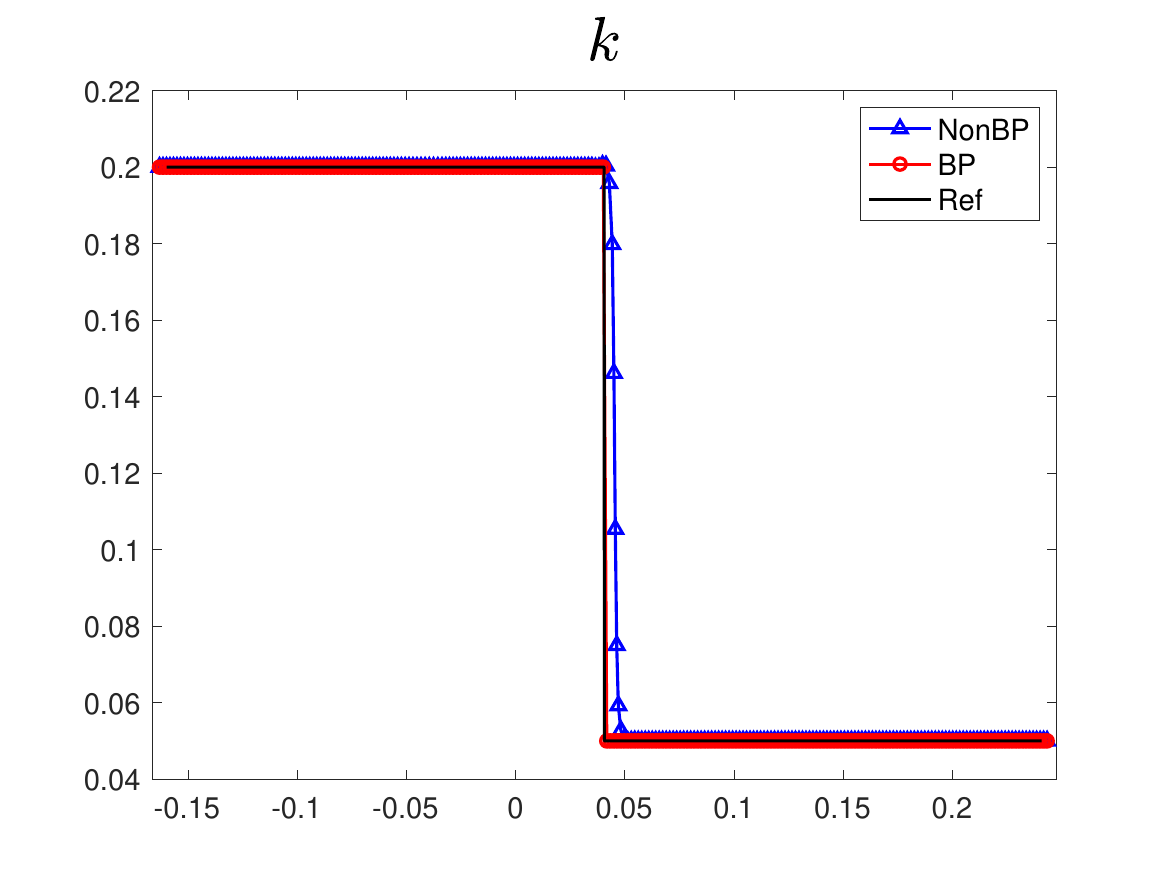}}}\\
	        \caption{\sf Example \ref{exam5}, Test T4, $t = 1$, and $N_\xi = 500$.}
	        \label{789}
	    \end{center}
	\end{figure}

	\begin{figure}[hbtp]
	    \begin{center}
	        \mbox{
	        {\includegraphics[width = 0.33\textwidth,trim=25 15 35 10,clip]{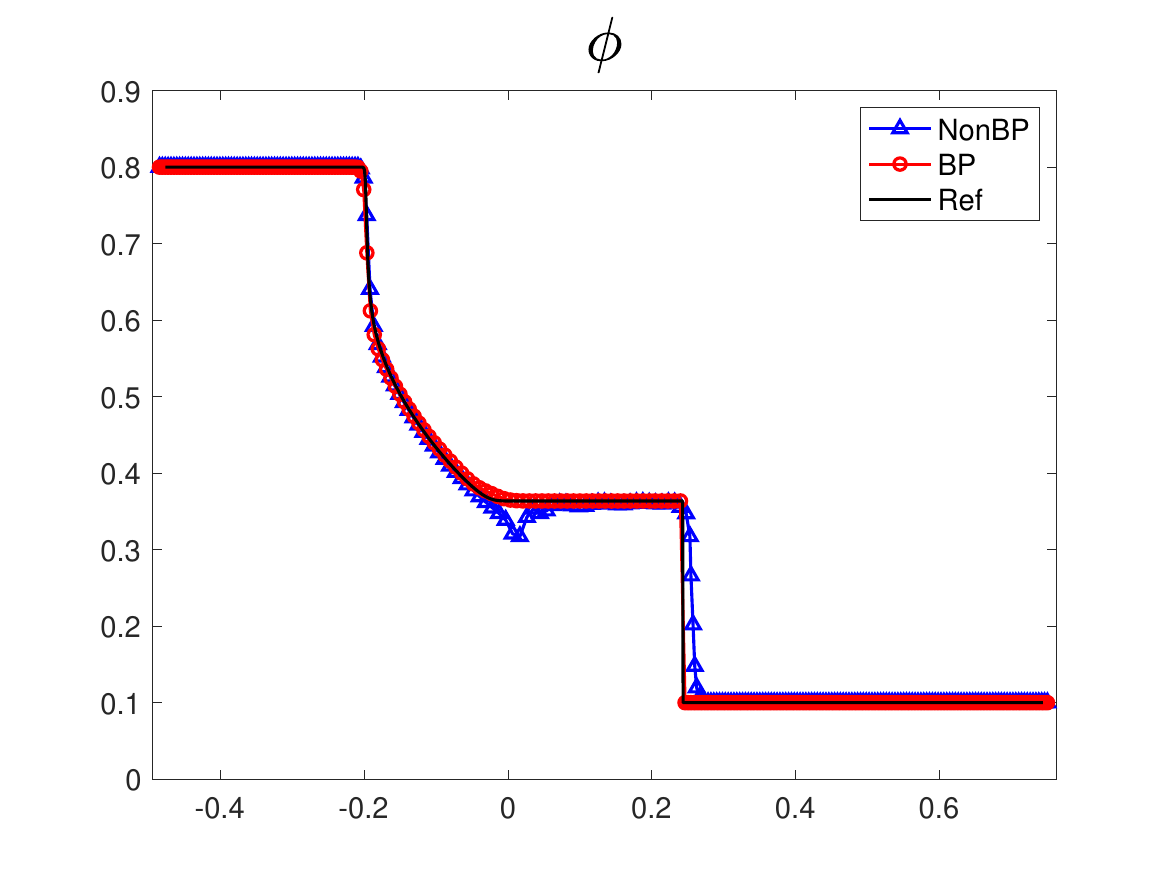}}
	        {\includegraphics[width = 0.33\textwidth,trim=25 15 35 10,clip]{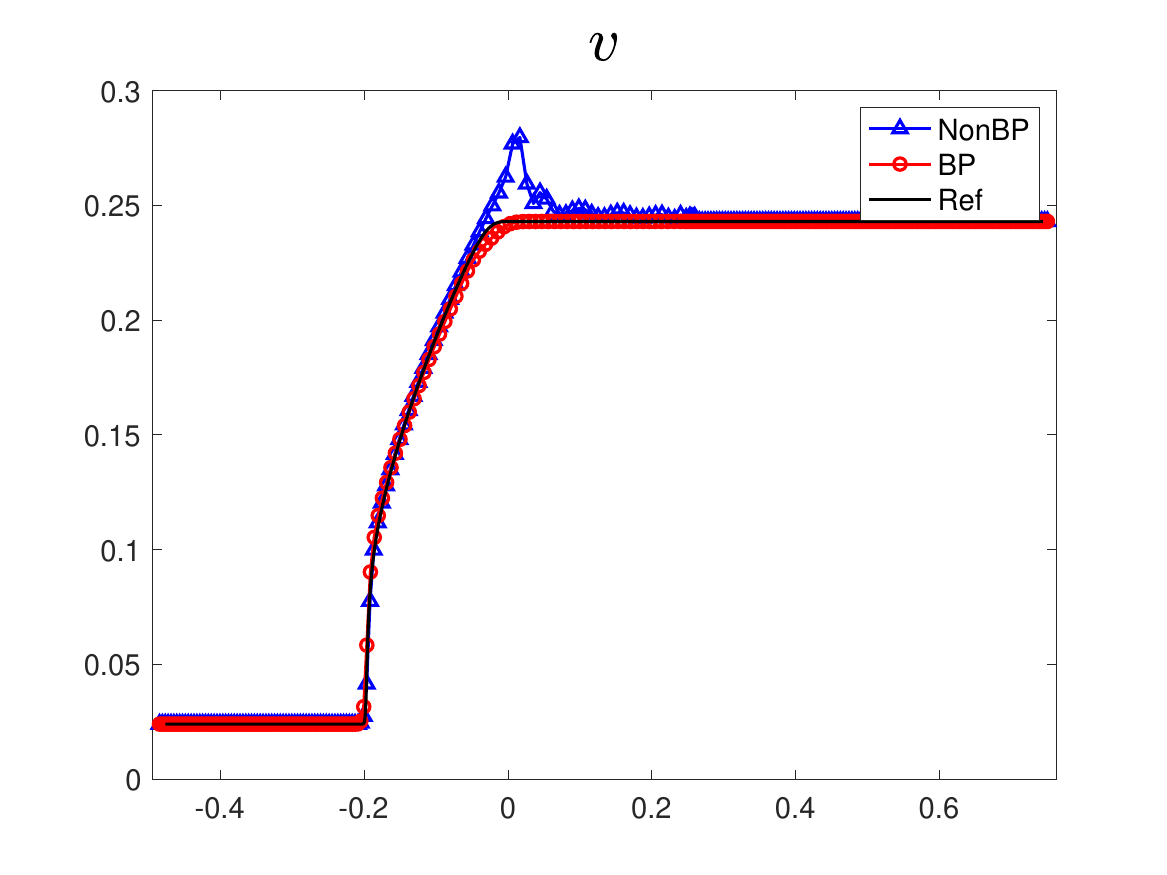}}
	        {\includegraphics[width = 0.33\textwidth,trim=25 15 35 10,clip]{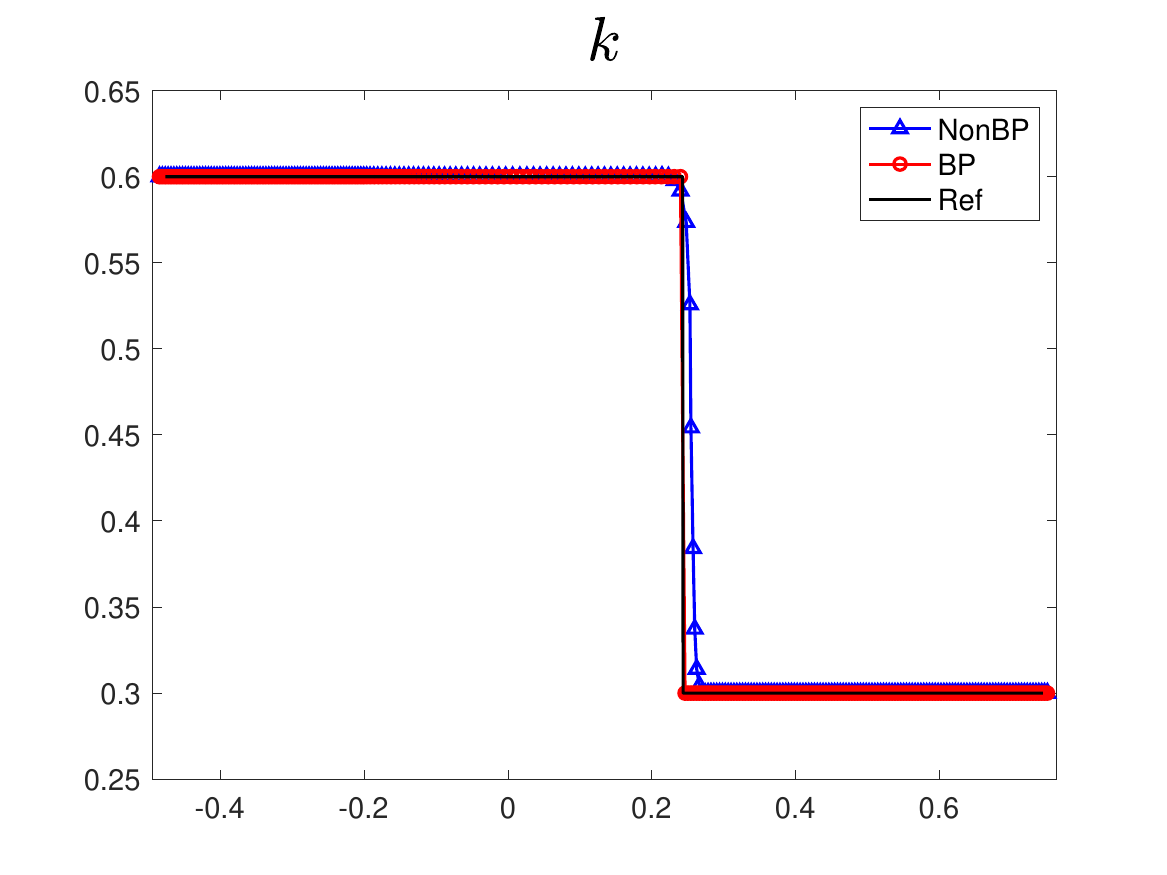}}}\\
	        \caption{\sf Example \ref{exam5}, Test T5, $t = 1$, and $N_\xi = 500$.}
	        \label{086}
	    \end{center}
	\end{figure}
\end{exa}

\subsection{Numerical experiments for ARZ model on road networks}\label{network}

In this section, we will extend the BP schemes to ARZ model on road networks. 
Consider a road network consists of $N_R$ road segments and $N_S$ junctions, we let $x^{(r)}_{L}$ and $x^{(r)}_{R}$ respectively label the physical locations of entry and exit of Road $r$.
The numerical solution on the $r$-th road segment is denoted by $\BU^{(r)}_h$. Let $\delta_{\ell}^-$ (resp. $\delta_{\ell}^+$) denote the set of indices representing incoming (resp. outgoing) roads connected to the Junction $\ell$. 
At road junctions, proper coupling conditions are needed to allocate the traffic fluxes from incoming roads ($i \in \delta^-_\ell$) to outgoing roads ($j \in \delta^+_\ell$).
Given the traffic states in the vicinity of the junction, namely,
\[
\textrm{the input of coupling condition:}
\quad
\left \{ \BU^{(i)}_h \left(x_L^{(i)} \right) \right \}_{i \in \delta_\ell^-} {\rm and} \quad
\left\{ \BU^{(j)}_h \left(x_R^{(j)} \right) \right\}_{j \in \delta_\ell^+},
\]
a coupling condition should be employed to determine the corresponding traffic states immediately outside the physical domain, that is,
\[
\textrm{the output of coupling condition:}
\quad
\left \{ \BU^{(i)}_h \left(x \leq x_L^{(i)} \right) \right \}_{i \in \delta_\ell^-} {\rm and} \quad
\left\{ \BU^{(j)}_h \left(x \geq x_R^{(j)} \right) \right\}_{j \in \delta_\ell^+},
\]
which are required to evolve the numerical solutions $\{\BU^{(r)}_h\}_{r=1}^{N_r}$ to the next time step, and are also used in updating global invariant domains (see \eqref{764}) and local invariant domains (see \eqref{eq:1795}). 
\begin{figure}[hbtp]
    \begin{center}
        \mbox{
        {\includegraphics[width = 0.6\textwidth,trim=80 365 80 320,clip]{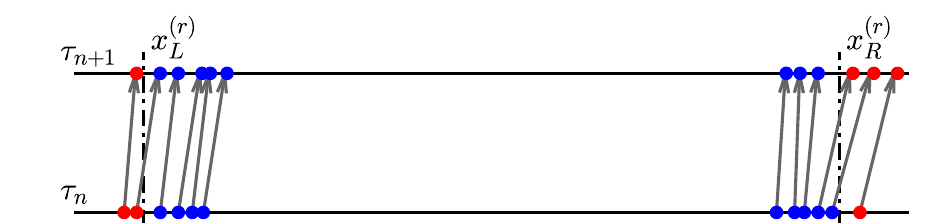}}}\\
        \caption{\sf Moving mesh.}
        \label{08}
    \end{center}
\end{figure}

The numerical scheme presented in this paper utilizes a moving mesh approach. Figure \ref{08} demonstrates the mesh movement within the physical domain between time levels $\tau_n$ and $\tau_{n+1}$. 
In this figure, blue points represent locations within the physical domain, whereas red points lie outside the physical domain (in the computational domain, red points correspond to ghost cells). 
As mesh movement does not ensure that computational points will exactly align with the physical boundaries 
at the subsequent time level, interpolation is required to estimate values at $x_L^{(r)}$ and $x_R^{(r)}$ using nearby interior points. 
In this work, we select the five closest points to $x_L^{(r)}$  (resp. $x_R^{(r)}$) within the physical domain and construct a fourth-degree polynomial to interpolate the values $\phi_L^{(r),*}$, $v_L^{(r),*}$, and $J_L^{(r),*}$ (resp. $\phi_R^{(r),*}$, $v_R^{(r),*}$, and $J_R^{(r),*}$) at the boundaries. Then setting
\begin{equation*}
	\phi_M^{(r)} = \max \{\phi_M^{(r),*}, \varepsilon \}, \quad v_M^{(r)} = \max \{v_M^{(r),*}, \varepsilon \}, \quad J_M^{(r)} = \max \{J_M^{(r),*}, \varepsilon \}, \quad M = L \; {\rm or} \; R,
\end{equation*}
makes it more physical, and we use the new values $\phi_M^{(r)}$, $v_M^{(r)}$, and $J_M^{(r)}$ to construct the value $\BU^{(r)}_h \left(x_M^{(r)} \right)$.

Various options for coupling condition are available in the literature, for example, \cite{garavello2006traffic,herty2006coupling,herty2006optimization,haut2007second,gottlich2021second}. 
For brevity, we consider HB network \cite{haut2007second} in the following numerical experiments.
For $j \in \delta_\ell^+$ and $i \in \delta_\ell^-$, $q_{ji} \geq 0$ denotes the traffic flux from Road $i$ to Road $j$, and we denote by $q_i^{-}$ (resp. $q_j^+$) the total incoming (resp. outgoing) flux of Road $i$ (resp. $j$), namely, $q_i^- = \sum_{j \in \delta_\ell^+} q_{ji}$, $q_j^+ = \sum_{i \in \delta_\ell^-} q_{ji}$.
For the HB network, the fluxes $\mathbf{q}_j = (q_{ji})_{i \in \delta_\ell^-}$ of an outgoing Road $j$, are proportional to a given vector $\boldsymbol{\beta} = (\beta_i)_{i \in \delta_\ell^-}$ , which is determined by demand from Road $i$.
For simplicity, in the following examples, we assume the lengths of all road segments to be 1.

\begin{exa} \label{net1}
	We examine an HB network with $\gamma = 1$, which includes one diverging junction, one incoming road (Road 1), and two outgoing roads (Roads 2 and 3). The initial conditions for these roads are specified as follows:
	\begin{equation*}
	\label{eqnet1}
	    \begin{aligned}
	        &\phi^{(1)}(\xi, 0) = 0.5, \quad \phi^{(2)}(\xi, 0) = \left\{
	        \begin{aligned}
	            &0.5 \quad &&{\rm if} \; \xi < 0.5, \\
	            &0.1 \quad &&{\rm otherwise},
	        \end{aligned}
	        \right. \\
	        &\phi^{(3)}(\xi, 0) = \left\{
	        \begin{aligned}
	            &0.6 \quad &&{\rm if} \; \xi \in [0.2, 0.4] \cup [0.6, 0.8], \\
	            &0.5 \quad &&{\rm otherwise},
	        \end{aligned}
	        \right. \\
	        &v^{(1)}(\xi, 0) = v^{(2)}(\xi, 0) = 0.5,\quad v^{(3)}(\xi, 0) = 1 - \phi^{(3)}(\xi, 0),
	    \end{aligned}
	\end{equation*}
	with $v_{\rm ref} = 1$ and a final time of $t = 0.2$. The traffic flow entering the junction from incoming Road 1 is equally distributed between outgoing Roads 2 and 3. To assess the local BP and global BP schemes, we present the results in Figure \ref{T9}. For the global BP scheme, a non-physical velocity overshoot is observed, as $v^{\rm G}_{\rm max} (\approx 0.85)$ significantly exceeds $v^{\rm L}_{\rm max}(\approx 0.50)$ near $x = 0.1$. In this scenario, enforcing a global constraint $v \leq v^{\rm G}_{\rm max}$ does not prevent the $v$-overshoot in the vicinity. The local BP scheme performs exceptionally well because $v^{\rm L}_{\rm max}$ is derived from the computational stencil.

	\begin{figure}[hbtp]
	    \begin{center}
	        \mbox{
	        {\includegraphics[width = 0.49\textwidth,trim= 25 55 41 75,clip]{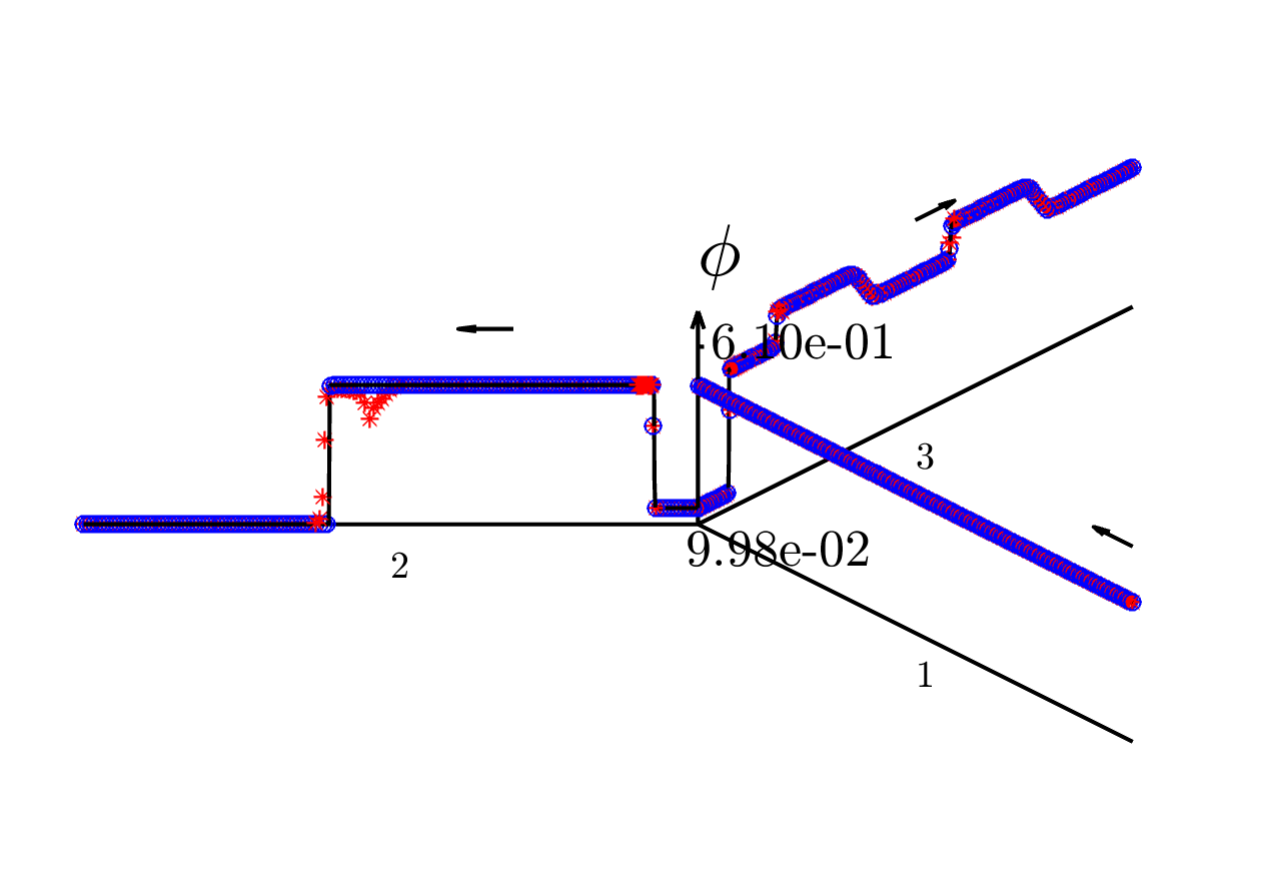}}
	        {\includegraphics[width = 0.49\textwidth,trim= 25 55 41 75,clip]{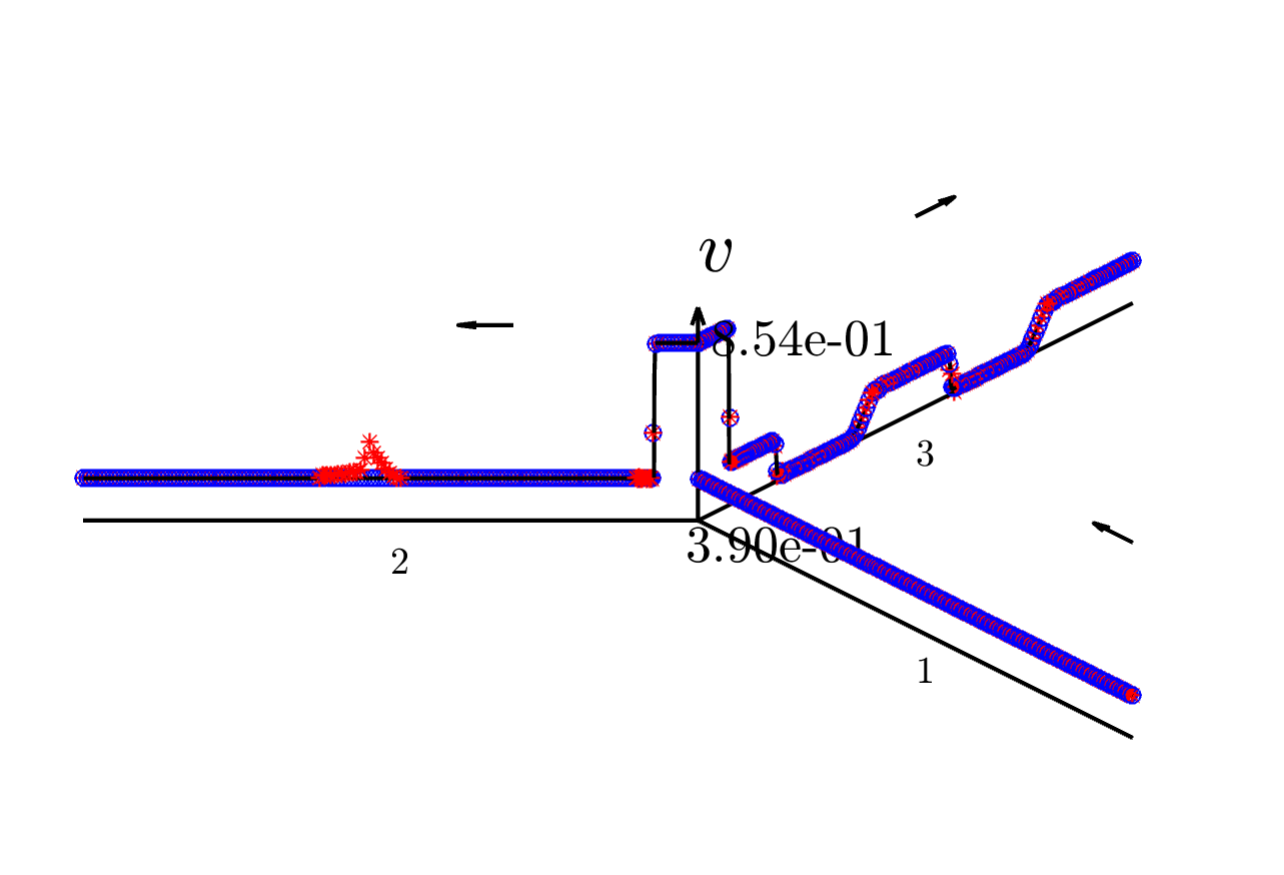}}}\\
	        \caption{\sf Example \ref{net1}, diverging junction, $t = 0.2$, and $N_\xi = 700$. Red: global BP scheme; Blue: local BP scheme; Black: reference solutions.}
	        \label{T9}
	    \end{center}
	\end{figure}

\end{exa}

\begin{exa} \label{net2}
	In this example, we consider an HB network comprising a merging junction connected to one outgoing road (Road 3) and two incoming roads (Road 1 and 2). The initial conditions for these roads are given by
	\begin{equation*}
	\label{eqnet2}
	    \begin{aligned}
	        &\phi^{(1)}(\xi, 0) =\phi^{(2)}(\xi, 0) = 0.4, \quad \phi^{(3)}(\xi, 0) = \left\{
	        \begin{aligned}
	            &0.4 \quad &&{\rm if} \; \xi < 0.5, \\
	            &10^{- 10} \quad &&{\rm otherwise},
	        \end{aligned}
	        \right. \\
	        &v^{(r)}(\xi, 0) = 0.4, \quad r = 1, 2, 3.
	    \end{aligned}
	\end{equation*}
	with $v_{\rm ref} = 1$, $ \gamma = 1$, and a final time of $t = 0.5$. The BP scheme is employed to address this scenario. The numerical results, presented in Figure \ref{T10}, clearly show that the discontinuities are well resolved without oscillations, even at very low densities. Notably, the NonBP scheme fails when simulating such low-density conditions.

	\begin{figure}[hbtp]
	    \begin{center}
	        \mbox{
	        {\includegraphics[width = 0.49\textwidth,trim= 25 55 41 60,clip]{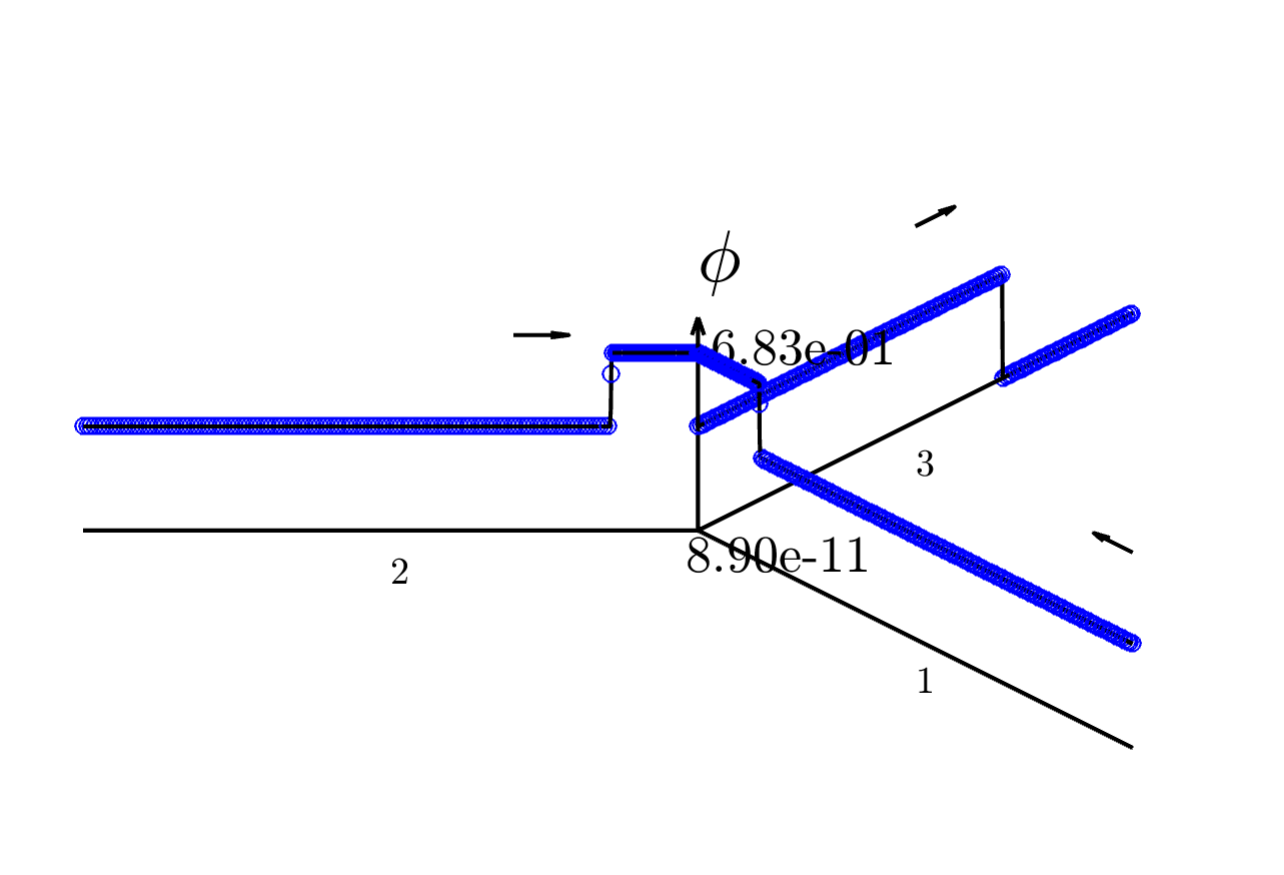}}
	        {\includegraphics[width = 0.49\textwidth,trim= 25 55 41 60,clip]{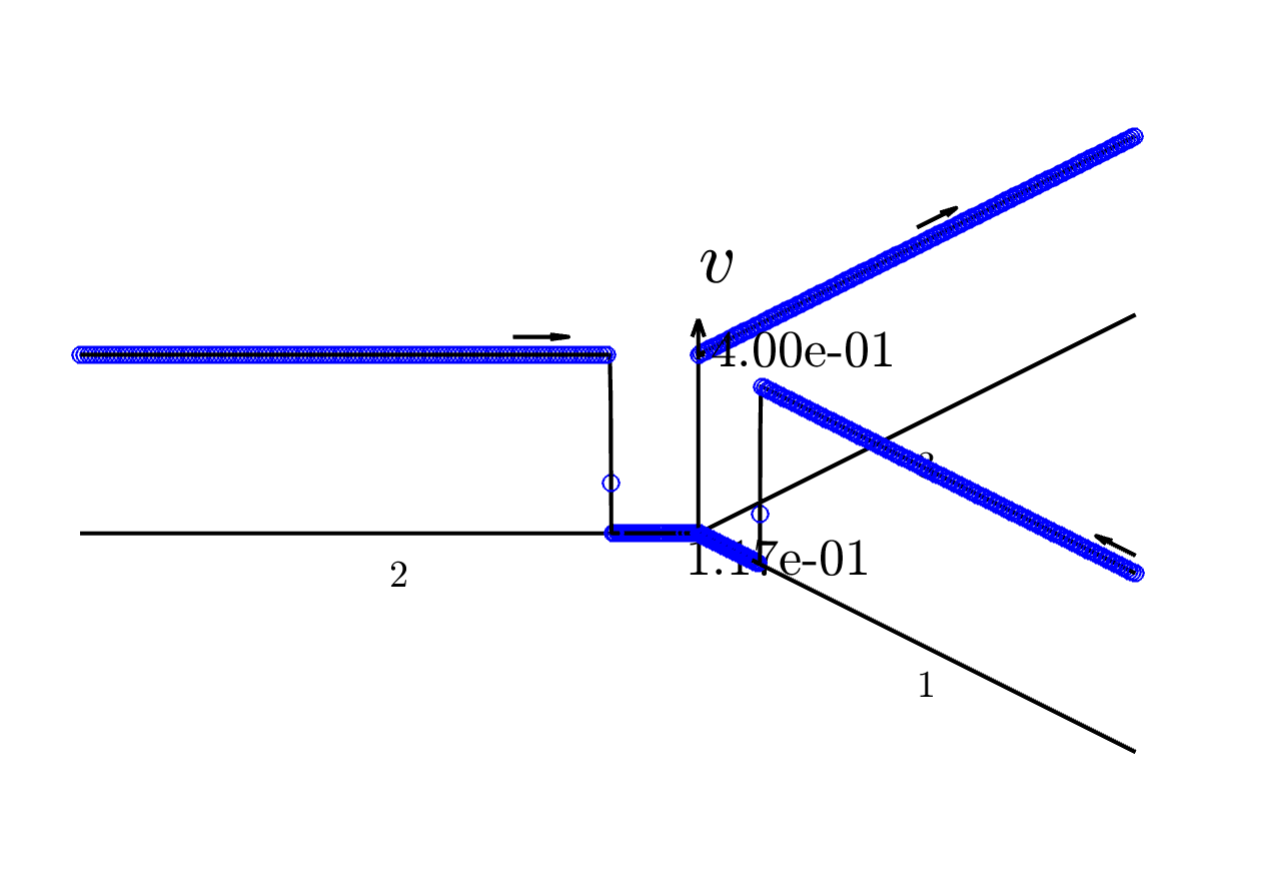}}}\\
	        \caption{\sf Example \ref{net2}, merging junction, $t = 0.5$, and $N_\xi = 700$. Blue: local BP scheme; Black: reference solutions.}
	        \label{T10}
	    \end{center}
	\end{figure}
\end{exa}

\begin{exa} \label{net3}
	A network comprising 30 roads and 24 junctions will be tested in the final example. As depicted in the left heatmap of Figure \ref{T11}, this network features five types of initial conditions for $\phi$ with a consistent $v^{(\cdot)}(\xi, 0) \equiv 0.5$ for all road segments:
	\begin{equation}
	\label{eqnet3}
	    \begin{aligned}
	        &\phi^{(1)}(\xi, 0) = \left\{
	        \begin{aligned}
	            &0.05 \quad &&{\rm if} \; \xi \in [0.4, 0.6], \\
	            &0.1 \quad &&{\rm otherwise},
	        \end{aligned}
	        \right. \quad 
	        &&\phi^{(2)}(\xi, 0) = \left\{
	        \begin{aligned}
	            &10^{- 10} \quad &&{\rm if} \; \xi \in [0.1, 0.5], \\
	            &0.1 \quad &&{\rm otherwise},
	        \end{aligned}
	        \right. \\
	        &
	        \begin{aligned}
	            &\phi^{(3)}(\xi, 0) = 0.1, \\
	            &\phi^{(4)}(\xi, 0) = 0.1 + 0.05 \sin{(2 \pi \xi)}, 
	        \end{aligned}
	        \quad &&\phi^{(5)}(\xi, 0) = \left\{
	        \begin{aligned}
	            &0.2 \quad &&{\rm if} \; \xi \in [0.4, 0.6], \\
	            &0.1 \quad &&{\rm otherwise}.
	        \end{aligned}
	        \right.
	    \end{aligned}
	\end{equation}
	We set $v_{\rm ref} = 1$, $\gamma = 2$, and a final time of $t = 0.5$. The heatmap displaying numerical results for $\phi$ is shown in the Figure \ref{T11}, and the resulting numerical solution on Roads 2, 4, and 5 are presented in Figure \ref{T12}.

	\begin{figure}[hbtp]
	    \begin{center}
	        \mbox{
	        {\includegraphics[width = 0.49\textwidth,trim= 25 55 41 38,clip]{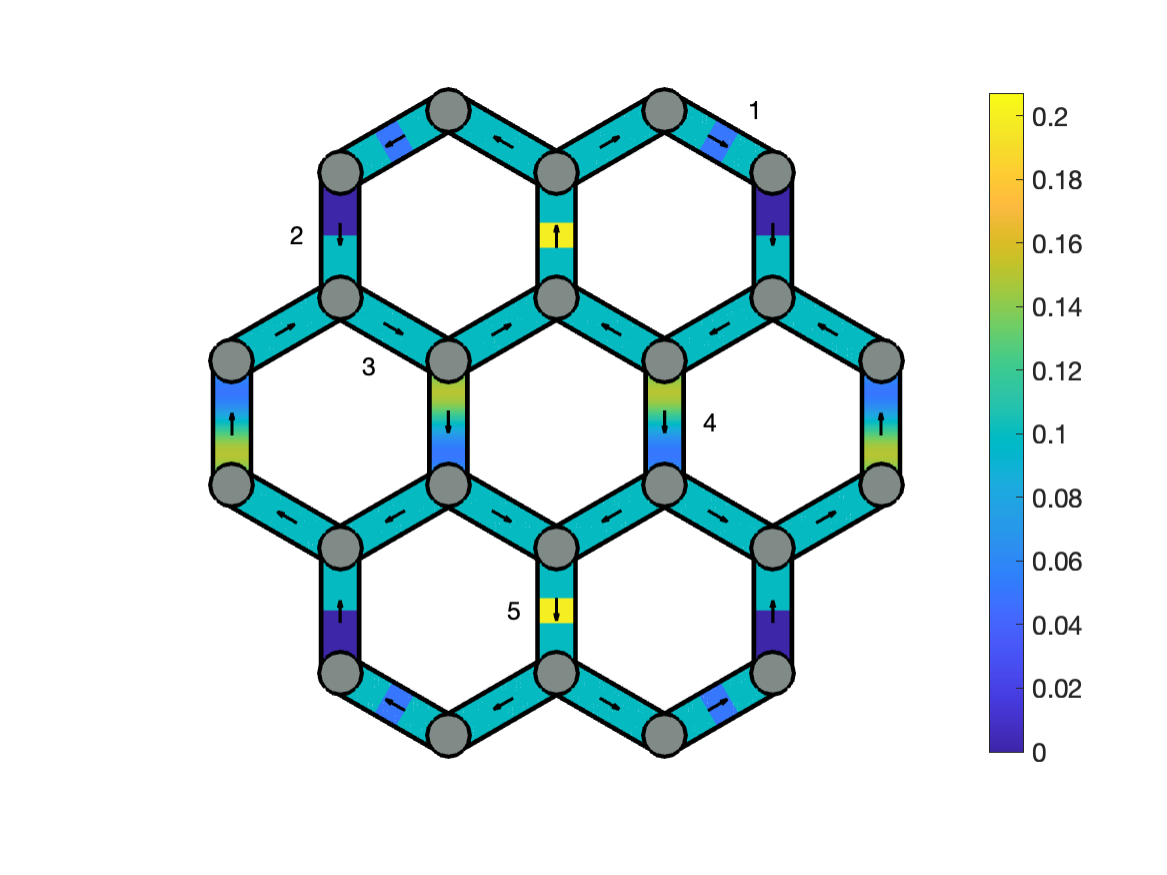}}
	        {\includegraphics[width = 0.49\textwidth,trim= 25 55 41 38,clip]{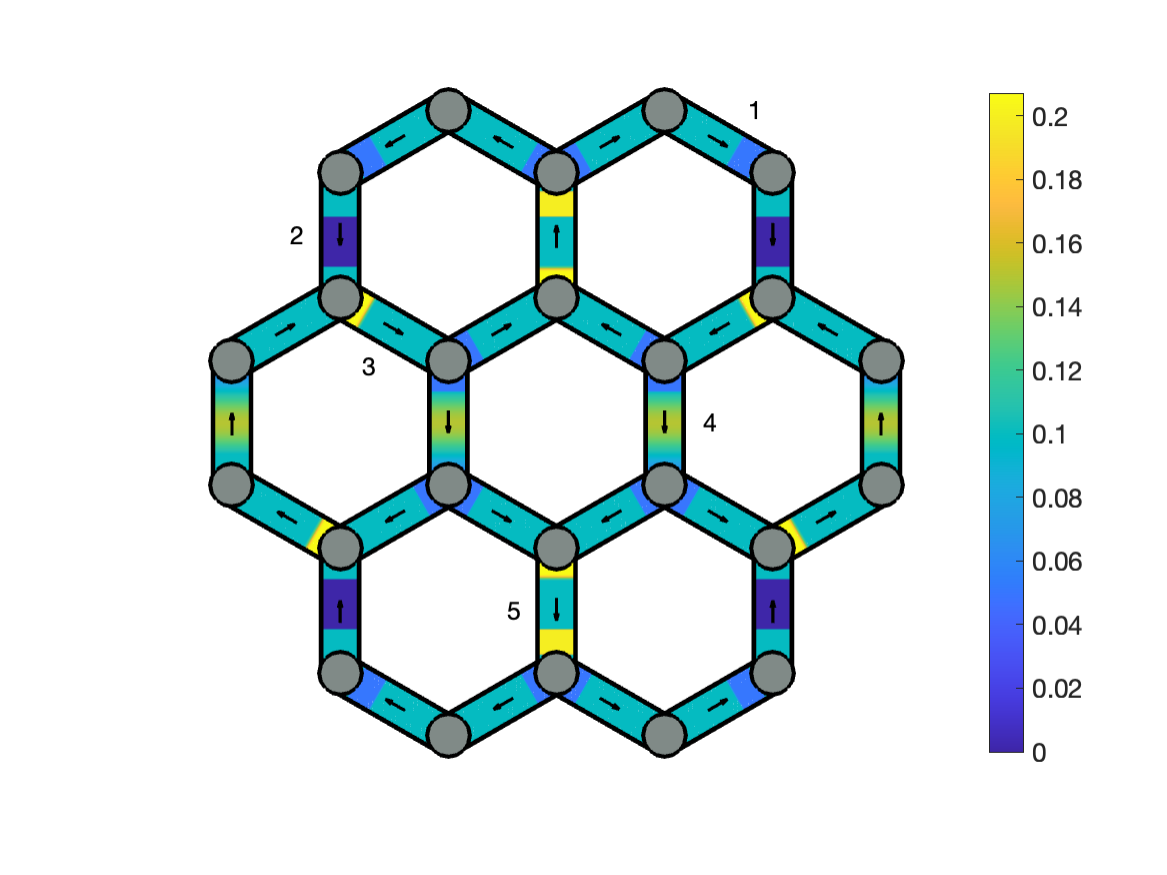}}}\\
	        \caption{\sf Example \ref{net3}, complex network, $N_\xi = 800$. Left: $\phi$ at $t = 0$; Right: $\phi$ at $t = 0.5$.}
	        \label{T11}
	    \end{center}
	\end{figure}

	\begin{figure}[hbtp]
	    \begin{center}
	        \mbox{
	        {\includegraphics[width = 0.33\textwidth,trim=25 15 35 10,clip]{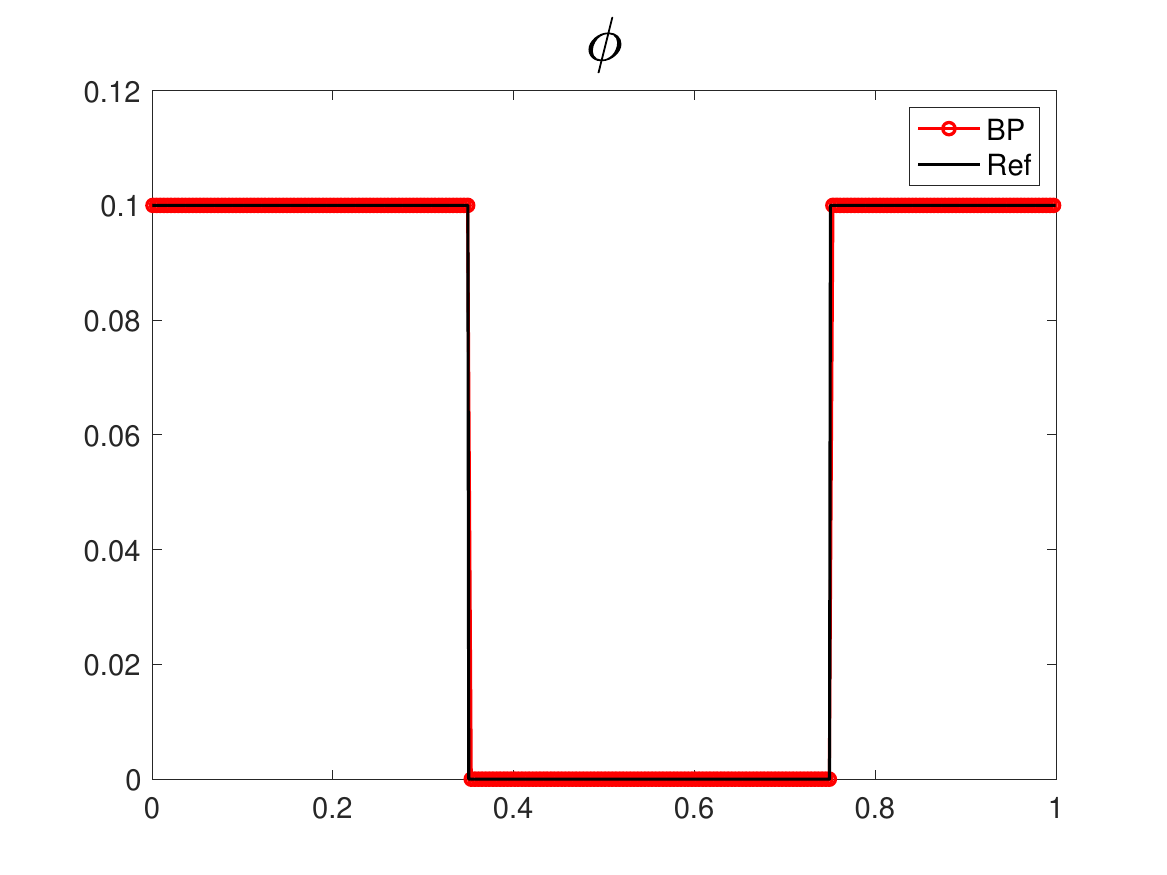}}
	        {\includegraphics[width = 0.33\textwidth,trim=25 15 35 10,clip]{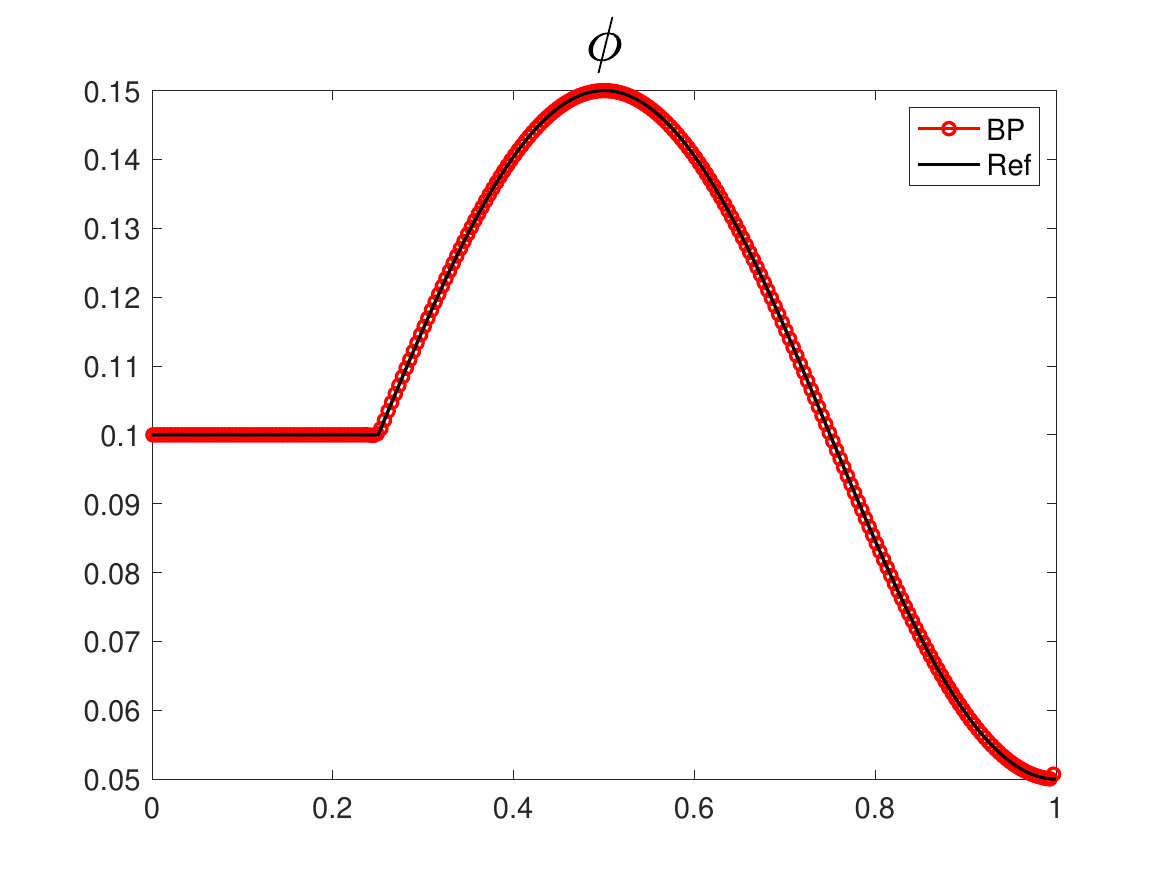}}
	        {\includegraphics[width = 0.33\textwidth,trim=25 15 35 10,clip]{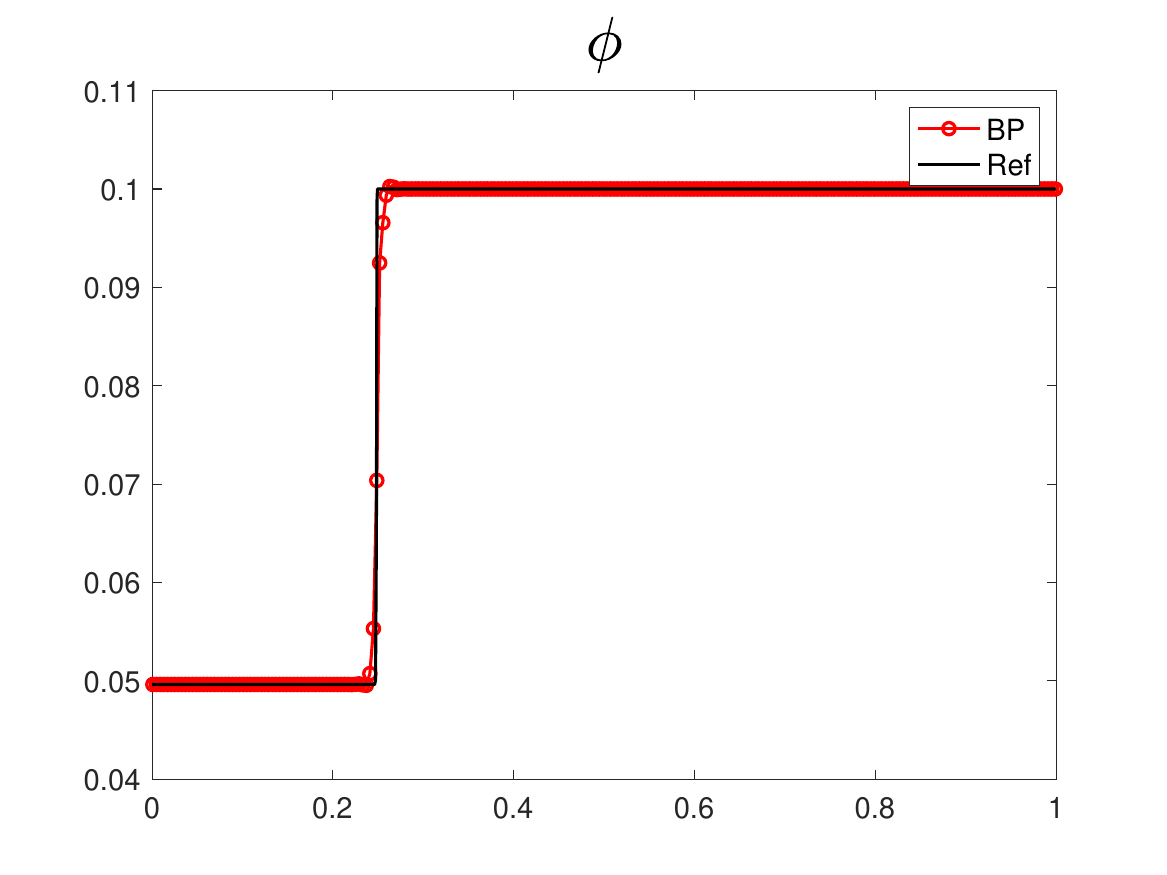}}}\\
	        \caption{\sf Example \ref{net3}, complex network, $t = 0.5$, and $N_\xi = 800$. Left: Road 2; Middle: Road 4; Right: Road 5.}
	        \label{T12}
	    \end{center}
	\end{figure}
\end{exa}

\section{Conclusion} \label{0sec6}
In this study, we investigated the numerical approximation of Temple-class systems, which are applicable across a variety of fields including one-dimensional two-phase flow, elasticity, traffic flow, and sedimentation. Utilizing a moving mesh approach enabled us to navigate the challenges posed by non-convex invariant domains effectively. The development of a bound-preserving (BP) and conservative numerical scheme that meets the demanding requirements of these systems underscore the robustness and adaptability of our methodology.
Numerical experiments validate the effectiveness of our BP methods, demonstrating their capability to ensure the positivity of $\phi$ and the adherence of Riemann invariants to the invariant domains within the simulations. Through the adaptation of high-order schemes with a parameterized flux limiter, we have shown that it is feasible to achieve high-order accuracy while adhering to critical principles of bound preservation.
Furthermore, the versatility of our proposed scheme is evident in its application to ARZ traffic flow networks. The successful integration of BP schemes with non-convex sets, as demonstrated in this paper, sets the stage for continued research and development in the numerical analysis of hyperbolic systems.
In conclusion, our research tackles the complex challenges inherent in Temple-class systems and establishes a solid foundation for advancing non-convex BP schemes.
\section*{Acknowledgement}
	

\bibliographystyle{model1-num-names}
\bibliography{reference}

\end{document}